\documentclass[10pt]{article}

\usepackage[dvipdf]{graphicx}
\usepackage[latin1]{inputenc}
\usepackage{latexsym}
\usepackage{amsmath,amssymb,amsfonts,amsthm}
\usepackage{xcolor}
\usepackage{mathrsfs}

\usepackage{enumerate}
\numberwithin{equation}{section}

\usepackage[numbers,comma,sort]{natbib}
\usepackage{comment}
\usepackage{stmaryrd}
\definecolor{db}{RGB}{0, 0, 130}
\usepackage[colorlinks=true,citecolor=db,linkcolor=black,urlcolor=black,pdfstartview=FitH]{hyperref}
\usepackage{dsfont}
\usepackage{pdfsync}
\usepackage{breqn}

\definecolor{rp}{rgb}{0.25, 0, 0.75}
\definecolor{dg}{rgb}{0, 0.5, 0}

\newcommand{\R}{\mathbb{R}}

\newcommand{\N}{\mathbb{N}}
\newcommand{\EE}{\mathbb{E}}

\newcommand*\diff{\mathop{}\!\mathrm{d}}
\newcommand{\fBm}{B}

\makeatletter
\def\namedlabel#1#2{\begingroup
    #2
    \def\@currentlabel{#2}
    \phantomsection\label{#1}\endgroup
}
\makeatother

\newtheorem{definition}{Definition}[section]
\newtheorem{theorem}[definition]{Theorem}
\newtheorem{prop}[definition]{Proposition}
\newtheorem{corollary}[definition]{Corollary}

\newtheorem{lemma}[definition]{Lemma}

\newtheorem{proposition}[definition]{Proposition}
\newtheorem{remark}[definition]{Remark}

\author{El Mehdi Haress\footnote{Universit\'e Paris-Saclay, CentraleSup\'elec, MICS and CNRS FR-3487, France.} \textsuperscript{,}\footnote{\texttt{el-mehdi.haress@centralesupelec.fr}} \and 
Alexandre Richard$^*$\textsuperscript{,}\footnote{\texttt{alexandre.richard@centralesupelec.fr} \newline 
 E.H. acknowledges the support of the Labex de Math\'ematique Hadamard. This work is supported by the SIMALIN project ANR-19-CE40-0016 and the SDAIM project ANR-22-CE40-0015 from the French National Research Agency.}}

\title{ \Large{\textbf{Long time Hurst regularity of fractional SDEs and their ergodic means}}}

\begin{document}

\maketitle

\begin{abstract}
The fractional Brownian motion can be considered as a Gaussian field indexed by $(t,H)\in {\R_{+}\times (0,1)}$, where $H$ is the Hurst parameter. On compact time intervals, it is known to be almost surely jointly H\"older continuous in time and Lipschitz continuous in $H$. First, we extend this result to the whole time interval $\mathbb{R}_{+}$ and consider both simple and rectangular increments. 
Then we consider SDEs driven by fractional Brownian motion with contractive drift. The solutions and their ergodic means are proven to be almost surely H\"older continuous in $H$, uniformly in time. This result is used in a separate work for statistical applications. We also deduce a sensibility result of the invariant measure in $H$.

The proofs are based on variance estimates of the increments of the fractional Brownian motion and fractional Ornstein-Uhlenbeck processes, multiparameter versions of the Garsia-Rodemich-Rumsey lemma and a combinatorial argument to estimate the expectation of a product of Gaussian variables.
\end{abstract}


\medskip

\noindent\textbf{MSC2020 subject classification: }60G22, 60H10, 37H10.

\section{Introduction}\label{sec:intro}
We consider the fractional Brownian motion (fBm) $\fBm$ as a stochastic process of two parameters given by its Mandelbrot-Van Ness representation: $\forall t \in \R, \ \forall H \in (0,1)$,
\begin{align}\label{eq:fBm}
 \fBm_t^H = \frac{1}{\Gamma(H+\frac{1}{2})} 
 \int_{\R} \left((t-s)_{+}^{H-\frac{1}{2}} - (-s)_{+}^{H-\frac{1}{2}} \right) \diff W_s ,
\end{align}
where the integral is in the Wiener sense. For fixed $H$, this process appears in models of physics \cite{Russo}, biology \cite{ROT}, finance \cite{GJR} to cite but a very few. Except when $H=\frac{1}{2}$ (Brownian case), $H$ usually encodes the long-range memory of the model or its roughness. It is therefore important to estimate its value precisely and to evaluate the sensitivity of functionals of the model at hand, with respect to $H$. Concerning the estimation of $H$ in a parametric setting, we refer for instance to \cite{BerzinEtAl,Tudor} and references therein. 
As for the sensitivity in the Hurst parameter, it has been studied in various situations and is an important topic in modeling: in \cite{KochNeuenkirch}, the fractional Brownian motion is proven to be infinitely differentiable with respect to its Hurst parameter for a fixed time; in \cite{jolis2010continuity,jolis2007continuity}, the law of (multiple) integrals with respect to the fBm are proven to be continuous in $H$;  
in \cite{richard2015fractional}, the H\"older continuity in $H$ is studied for generalised fractional Brownian fields (over compact index sets); in \cite{DeVecchiEtAl}, the  solution to rough differential equations driven by a 1D fBm is proven to be continuous with respect to $H$ when $H \in (\frac{1}{3}, \frac{1}{2}]$; and in \cite{GiordanoJolisQS}, the laws of  quasilinear stochastic wave and heat equations with additive fractional noise are proven to be continuous in $H$.
Finally in \cite{richard2016h}, the difference between functionals of a fractional stochastic differential equation (SDE) and its Markovian counterpart ($H=\frac{1}{2}$) are proven to be of order $|H-\frac{1}{2}|$, both for the law of the solution on a compact time interval, and for the law of a singular functional (namely the first hitting time).

~

In this paper, the main results concern the joint time and Hurst regularity of the fBm on unbounded time intervals, and of ergodic fractional SDEs. These results have direct applications to determine statistically the Hurst parameter from discrete observations, see the last paragraph of the Introduction and our companion paper \cite{HRstat}. In particular, Proposition \ref{cor:supH-bound}, Theorem \ref{th:ergodic-OU} and Theorem \ref{cor:ergodic-OU-discret} are useful in this statistical perspective. 

~

To introduce our methodology, it is natural and interesting to start by discussing the joint time and Hurst regularity of the fBm on a non-compact time interval. In the case of ergodic SDEs, we will then elaborate on the arguments of the fBm case. 

Almost sure regularity estimates in $t$ and $H$ can be obtained classically on compact intervals: considering the moments of the increments of $B$ in time, and applying Kolmogorov's continuity theorem, one gets that $B$ admits a modification (that we still denote by $B$) which almost surely has a finite $(H-\varepsilon)$-H\"older norm on $[0,T]$. However this norm depends on $T$. 
One can also derive from \eqref{eq:fBm} an upper bound on the moments of $B^H_{t}-B^{H'}_{t}$
 (see \cite{richard2015fractional,Decreusefond}) and deduce that
 $B$ is almost surely $(1-\varepsilon)$-H\"older continuous in the Hurst parameter, uniformly in $t\in [0,T]$. Altogether we know that for $\varepsilon\in (0,1)$ and a compact interval $\mathcal{H}\subset(0,1)$, there exist a continuous modification of $B$ (still denoted by $B$ thereafter) and a random variable $\mathbf{C}_{T}$ with finite moments such that almost surely
\begin{align}\label{eq:simpleInc1}
\forall H,H'\in \mathcal{H},\, \forall t,t'\in [0,T], ~ |\fBm_t^H -\fBm_{t'}^{H'}| \leq \mathbf{C}_T \left(|t-t'|^{H\wedge H'-\varepsilon} + |H-H'|^{1-\varepsilon}\right) .
\end{align}
We observe that the law of the iterated logarithm for fractional Brownian motion \cite{Orey,MonradRootzen} shows that for fixed $H$, $B^H$ cannot grow much faster than $t^H$ at infinity. Hence $\mathbf{C}_{T}$ is expected to grow slightly faster than $T^H$.
For multiparameter processes, it is also natural and sometimes more useful to consider the rectangular increments (see e.g. \cite{Khosh,hu2013multiparameter}, as well as \cite{Xiao} for considerations on the joint time and space regularity of local times of Gaussian processes). In fact, such increments are also useful for statistical purposes as they are crucial in the proof of Proposition \ref{cor:supH-bound}, which is then used in \cite{HRstat}. For $B$, the rectangular increment between $(t,H)$ and $(t',H')$ reads
\begin{align}\label{eq:rect-increments}
\Box_{(t,H)}^{(t',H')} B = B_{t}^H-B_{t'}^H-B_{t}^{H'}+B_{t'}^{H'} .
\end{align}
Compared to \eqref{eq:simpleInc1}, the upper bound on $|\Box_{(t,H)}^{(t',H')} B|$ is now the product $\mathbf{C}_{T} \, |t-t'|^{H\wedge H'-\varepsilon} \, |H-H'|^{1-\varepsilon}$. 
The results for these two types of increments (simple and rectangular) are presented respectively in Theorem \ref{thm:whole-regularity} and Theorem \ref{thm:regularityGarsia-rectangular} for parameters $H$ in a compact subset of $(0,1)$ and $t$ on the whole half-line, hence making explicit the dependence of $\mathbf{C}_T$ with respect to $T$ and providing sharp H\"older exponents. For the simple increments, this reads
\begin{align}\label{eq:reg-intro}
\forall t' \ge t \ge 0,\, \forall H,H' \in \mathcal{H}, ~ 
| \fBm_t^H - \fBm_{t'}^{H'} |  \leq \mathbf{C}\, (1+t')^{\max(\mathcal{H})+\varepsilon}  \big( 1 \wedge |t-t'|^{\min(\mathcal{H})} + |H-H'| \big)^{1-\varepsilon}
\end{align} 
almost surely, for some integrable random variable $\mathbf{C}$ that depends only on $\varepsilon$. The exponents in \eqref{eq:reg-intro} are close to optimal up to $\varepsilon$, as discussed in Remark~\ref{rk:optimality}.
The proofs of these theorems rely on the introduction of an auxiliary process $\mathbb{B}$ defined as
\begin{align*}%
\mathbb{B}_t^H = (1+t)^{-(\max \mathcal{H}+\varepsilon)} \, B_t^H , ~ t\in \R_+, H\in \mathcal{H} ,
\end{align*}
whose variance is bounded uniformly in time. Using multiparameter versions of the Garsia-Rodemich-Rumsey (GRR) lemma, the supremum of the increments of $\mathbb{B}$ can be controlled on compact sets by a quantity independent of time, which leads to the proof of Theorems \ref{thm:whole-regularity} and \ref{thm:regularityGarsia-rectangular}. 

~

In the second part of the paper, we study the fractional SDE 
\begin{align*}
Y_t^H = Y_0 + \int_0^t b(Y_s^H) \diff s + B_t^H ,
\end{align*}
 and its ergodic means. This equation is understood as a random ODE and is solved by classical fixed-point methods when $b$ is Lipschitz. See also \cite[Section 5.3]{nualartbook} for well-posedness when the noise is multiplicative and \cite{Hairer} for the existence and uniqueness of an invariant measure when $b$ is contractive. 
 If $b$ is Lipschitz and bounded, the Hurst regularity of the solution behaves as in \eqref{eq:reg-intro}, but with an extra exponential factor in time (see Remark~\ref{rk:boundedDrift}). Since we are interested in the long-time behaviour of solutions, it is natural to try to get rid of this exponential dependence in time. We first study this problem for the fractional Ornstein-Uhlenbeck (OU) process, see Proposition \ref{thm:regularity-OU}.  
Then assuming that the drift $b$ is contractive, we can control the trajectories of $Y^H$ by those of the OU process with the same parameter $H$. Then we obtain in Theorem~\ref{thm:regularity-SDE} that
almost surely, for all $t \ge 0$ and $H,H' \in \mathcal{H}$,
\begin{align*}
|Y_t^{H} - Y_{t}^{H'}| \leq \mathbf{C}\, (1+t)^{\varepsilon}\, |H-H'|^{1-\varepsilon} .
\end{align*}
Finally we focus on the regularity in the Hurst parameter of the ergodic means of $Y$. In Theorem~\ref{th:ergodic-OU}, we show that for $\mathcal{H}$ a compact subset of $(0,1)$, the following holds almost surely: for all $t \ge 0$ and $H,H' \in \mathcal{H}$,
\begin{align}\label{eq:SDEergodic-regularity}
\frac{1}{t+1} \int_0^{t+1} | Y_s^{H}-Y_s^{H'} |^2 \diff s  \leq \mathbf{C}\,   |H-H'|^{1-\varepsilon} .
\end{align}
The main technical tool to reach this result is an upper bound on the increments in $t$, $H$ and $H'$ of the ergodic mean of the fractional OU process (Proposition \ref{prop:ergodic-OU}). As an intermediate step, we relied on a Gaussian equality (Lemma \ref{lem:gaussian-product}): if $Z=(Z_{1},\dots, Z_{n})$ is a centred Gaussian vector with $\EE Z_{i}^2 = \EE Z_{j}^2$, then $\mathbb{E} [ \prod_{i=1}^n ( Z_i^2 - \mathbb{E} Z_i^2)]$ can be expressed as a sum of products of the covariances $\EE[Z_{i}Z_{j}]$ for $i\neq j$. 

In fact, when the drift $b$ is Lipschitz and coercive, we know that the SDE has a unique invariant measure $\mu_H$ (see \textit{e.g.} \cite{Hairer} and \cite[Lemma 3$(ii)$]{cohen2011approximation}). Hence in Corollary \ref{cor:reg-invariant}, we deduce from \eqref{eq:SDEergodic-regularity} that the invariant measure $\mu_H$ is almost $1/2$-H\"older continuous in $H$.

By similar computations, we derive in Theorem \ref{cor:ergodic-OU-discret} a discrete analog of \eqref{eq:SDEergodic-regularity}. Namely if $M^H$ denotes the solution of the Euler scheme associated to the SDE satisfied by $Y^H$ with time-step $\gamma>0$, then almost surely, for all $N\in \N^*$ and $H,H' \in \mathcal{H}$,
\begin{align*}
\frac{1}{N} \sum_{k=1}^{N} | M_{k \gamma}^{H}-M_{k \gamma}^{H'} |^2  \leq \mathbf{C}   |H-H'|^{1-\varepsilon} .
\end{align*}

\paragraph{Statistical applications.} 
In \cite{HRstat}, these results are used to build an ergodic statistical estimator of $H$. In fact, assuming that the drift $b$ is Lipschitz and coercive, we know that the SDE has a unique invariant measure $\mu_H$ (see \textit{e.g.} \cite{Hairer} and \cite[Lemma 3$(ii)$]{cohen2011approximation}). Suppose that the solution $Y^H$ is observed at discrete times $\{kh , k \in \mathbb{N} \}$ for some time-step $h >0$, then by \cite[Proposition 3.3]{panloup2020general},
\begin{align}\label{eq:invariant-measure-approx}
\lim_{n \rightarrow \infty} d\left( \frac{1}{n} \sum_{k=0}^{n-1} \delta_{Y_{k h}^H}, \mu_H\right) = 0 ,
\end{align}
where $d$ is any distance bounded by the Wasserstein distance and $\delta_{Y_{kh}^H}$ denotes the Dirac measure at $Y_{kh}^H$. Assuming that the invariant measure identifies the Hurst parameter, it is then natural to consider the estimator 
\begin{align*}
\widehat{H}^n := \underset{K \in (0,1)}{\text{argmin}} \ d\left(\frac{1}{n} \sum_{k=0}^{n-1} \delta_{Y_{k h}^H}, \mu_K\right) . 
\end{align*}
In practice, $\mu_K$ is usually not known and can for instance be approximated using an Euler scheme with time-step $\gamma$. For $K \in (0,1)$, we define
\begin{align*}
M^K_0 = Y^K_0, \quad  M^K_{(k+1)\gamma}  = M^K_{k \gamma} + \gamma b(M^K_{k\gamma}) + B^K_{(k+1)\gamma}-B^K_{k\gamma}.
\end{align*}
Similarly to \eqref{eq:invariant-measure-approx}, by \cite[Proposition 4.1]{cohen2011approximation}, we have
\begin{align*}
\lim_{\gamma \rightarrow 0, \, N \rightarrow \infty} d\left( \frac{1}{N} \sum_{k=0}^{N-1} \delta_{M_{k \gamma}^K}, \mu_K\right) = 0 .
\end{align*} 
Therefore in practice, for $n, N \in \mathbb{N}$ and $ \gamma >0 $, the estimator reads
\begin{align}\label{eq:estimatorH}
\widehat{H}^{n,N,\gamma} = \underset{K \in (0,1)}{\text{argmin }} d\left(\frac{1}{n} \sum_{k=0}^{n-1} \delta_{Y_{k h}^H}, \frac{1}{N} \sum_{k=0}^{N-1} \delta_{M_{k \gamma}^K} \right) .
\end{align}
Studying the convergence of $\widehat{H}^{n,N,\gamma}$ therefore boils down to studying the argmin of the random function $f: K \mapsto d(\frac{1}{n} \sum_{k=0}^{n-1} \delta_{Y_{k h}^H}, \frac{1}{N} \sum_{k=0}^{N-1} \delta_{M_{k \gamma}^K})$. This requires in particular some regularity properties of $f$. For $K,K'$ in $(0,1)$, we have by the triangle inequality
\begin{align*}
|f(K) -f(K') | \leq d \left( \frac{1}{N} \sum_{k=0}^{N-1} \delta_{M_{k \gamma}^K}, \frac{1}{N} \sum_{k=0}^{N-1} \delta_{M_{k \gamma}^{K'}} \right).
\end{align*}
For instance, if $d$ is the $2$-Wasserstein distance, we get
\begin{align*}
|f(K) -f(K') |^2 \leq \frac{1}{N} \sum_{k=0}^{N-1}  |M_{k \gamma}^K- M_{k \gamma}^{K'} |^2 .
\end{align*}   
So it becomes necessary to know the regularity of ergodic means in the Hurst parameter. This is the main focus of Theorem \ref{th:ergodic-OU} and Theorem \ref{cor:ergodic-OU-discret}. Using these results, we show in \cite[Theorem 2.6]{HRstat} that $\widehat{H}^{n,N,\gamma}$ converges to $H$ as $n,N \rightarrow \infty$ and $\gamma \rightarrow 0$, and obtain a rate of convergence \cite[Theorem 2.8]{HRstat}.
 Ergodic estimators of $H$ offer a practical way of estimating the Hurst parameter, but we also refer to~\cite{kubilius2012rate,Tudor} for alternative methods with vanishing time step (without ergodicity assumptions).

\paragraph{Organisation of the paper.} 
In Section \ref{sec:lower-upper-b}, we obtain sharp upper bounds on the moments of the simple and rectangular increments of both $B$ and $\mathbb{B}$. Theorems \ref{thm:whole-regularity} and \ref{thm:regularityGarsia-rectangular} are stated and proven in Section \ref{sec:mainresults}.
The regularity for solutions of fractional additive SDEs, Theorem \ref{thm:regularity-SDE}, is presented in Section \ref{sec:application-reg} and the regularity for ergodic means, Theorem \ref{th:ergodic-OU}, appears in Section \ref{sec:application-ergodic}. In Section \ref{sec:application-discrete}, we consider discrete-time SDEs.
 The analysis in both Section \ref{sec:application} and \ref{sec:application-discrete} is based on a comparison with fractional OU processes, which are studied in Appendix \ref{sec:app}. We prove upper bounds on their moments in Lemma \ref{lem:Ubar-increments}, and on moments of ergodic means of both OU and discrete OU processes in Lemma \ref{lem:ergodic-U2} and Lemma \ref{lem:ergodic-U2-dsicrete}. Finally the proofs of the results from Section~\ref{sec:mainresults} on the rectangular increments are gathered in Appendix~\ref{app:rect-inc}.

\paragraph{Notations.}
We denote by $C$ a constant independent of time and the Hurst parameter. It may change from line to line. We also denote by $\llbracket 1, N \rrbracket$ the set of integers between $1$ and $N \in \mathbb{N} \setminus \{ 0 \}$.

\section{Bounds on the variance of the increments of $B$ and $\mathbb{B}$}\label{sec:lower-upper-b}

We first study  the regularity in $t$ and $H$ of the moments of the increments of $B$ and $\mathbb{B}$. As mentioned in the Introduction, this permits to introduce smoothly the methodology that is based on moment estimates and variations of the GRR lemma, and that will also be used and developed with further arguments in Sections~\ref{sec:application} and \ref{sec:application-discrete}.

\subsection{Simple increments}
The first proposition below is essential to obtain a.s. bounds on the simple and rectangular increments of the fBm. It will also be useful in Sections~\ref{sec:application} and \ref{sec:application-discrete} for the regularity of fractional SDEs and their ergodic means.

\begin{prop}\label{prop:ub-justfBm}
Let $\mathcal{H}$ be a compact subset of $(0,1)$. There exists a constant $C$ such that for all $t'\ge t \ge 0$ and $H,H' \in \mathcal{H}$:
\begin{align}\label{eq:variance-ub}
\EE\left( B_t^H - B_{t'}^{H'} \right)^2 & \leq C\, |t-t'|^{2H'} + C\, \big(t^{2H}\vee t^{2H'}\big) \, (\log^2(t)+1)\,  |H-H'|^2  .
\end{align}
\end{prop}

\begin{proof}
Using $\fBm_t^{H'}$ as a pivot term, we have
\begin{align}\label{eq:first-ub}
\mathbb{E} \left( \fBm_t^H - B_{t'}^{H'} \right)^2 & \leq 2 \mathbb{E} \left(\fBm_t^H - B_{t}^{H'} \right)^2 + 2 \mathbb{E} \left(\fBm_{t}^{H'} - B_{t'}^{H'} \right)^2 \nonumber \\
&= 2 \mathbb{E} \left(\fBm_t^H - B_{t}^{H'} \right)^2 + 2 (t'-t)^{2H'}.
\end{align}
The estimation of $\EE\big( B_t^H - B_{t}^{H'} \big)^2$ has been done in a compact in \cite{Decreusefond} (see Lemma 3.2 therein) and in a more abstract framework in \cite{richard2015fractional} (see Corollary 3.4 therein for $t\in [0,1]$). For the sake of completeness, we provide here a different proof for $t\in \R_{+}$.

Using the Mandelbrot-Van Ness representation of the fBm $\fBm$, we get
\begin{align*}
\mathbb{E} &\left( \fBm_t^H - B_{t}^{H'} \right)^2 \\
&  =\EE \left( \int_{-\infty}^0 \big(  K_1(H,t,-s) -  K_{1}(H',t,-s) \big) \diff W_s + \int_0^t  \big(K_2(H,t,s)  -  K_2(H',t,s)\big)  \diff W_s \right)^2  ,
\end{align*}
where the kernels $K_{1}$ and $K_{2}$ are given by
\begin{align}\label{eq:noyaux}
K_1(H,t,-s) & :=  \frac{(t-s)^{H-1/2}-(-s)^{H-1/2}}{\Gamma(H+1/2)} \nonumber \\
K_2(H,t,s) & = \frac{(t-s)^{H-1/2}}{\Gamma(H+1/2)} .
\end{align}
By the isometry property of the Wiener integral, we get
\begin{align*}
\mathbb{E}\left( \fBm_t^H - \fBm_{t}^{H'} \right)^2 
&= \int_{0}^\infty \big( K_1(H,t,s) - K_1(H',t,s) \big)^2 \diff s  
 + \int_0^t \big( K_2(H,t,s) - K_2(H',t,s) \big)^2 \diff s  .
\end{align*}
Then with the notations $x_h := (H-H') h+H'$ and $\partial_{H} K_{1}$ (resp. $\partial_{H} K_{2}$) for the derivative of $K_{1}$ (resp. $K_{2}$) with respect to its first variable, we obtain
\begin{align}\label{eq:ub-firstdecomp3}
\mathbb{E}\left( \fBm_t^H - \fBm_{t}^{H'} \right)^2 & \leq 2 | H-H' |^2 \int_{0}^{1} \int_0^{\infty}  \left( \partial_H K_1(x_h,t,s)\right)^2  + \mathds{1}_{s\leq t}   \left( \partial_H K_2(x_h,t,s)\right)^2 \diff s  \diff h .
\end{align} 
The integrals $ \int_0^{\infty}  \left( \partial_H K_1(x_h,t,s)\right)^2 \diff s $ and $\int_0^{\infty} \mathds{1}_{s\leq t} \left( \partial_H K_2(x_h,t,s)\right)^2 \diff s $ will be treated separately, respectively in Step $1$ and Step $2$. These two quantities are of the same order and bounded by $C ( t^{2H}  \vee t^{2H'} ) \, (\log^2(t)+1)$.

\paragraph{Step 1: Uniform upper bound on $ \int_0^{\infty}  \left( \partial_H K_1(x_h,t,s)\right)^2 \diff s $.\\} 
First, the Gamma function is positive and bounded away from $0$ on $[\frac{1}{2},\frac{3}{2}]$. It follows that the derivative of $x\mapsto \frac{1}{\Gamma(x)}$ is also bounded on $[\frac{1}{2},\frac{3}{2}]$. 

Hence
\begin{align*}
\partial_H K_1(x,t,s) & =\frac{(t+s)^{x-1/2}\log(t+s)-s^{x-1/2}\log(s)}{\Gamma(x+1/2)} + ( (t+s)^{x-1/2}- s^{x-1/2})  \Big(\frac{1}{\Gamma}\Big)'(x+1/2) ,
\end{align*}
so that
\begin{align*}
(\partial_H K_1(x,t,s))^2 & \leq C \left( (t+s)^{x-1/2}\log(t+s)-s^{x-1/2}\log(s) \right)^2 + C \left((t+s)^{x-1/2}- s^{x-1/2} \right)^2 .
\end{align*}
After integrating over $s$ and applying the change of variables $u= s/t$, we get
\begin{align*}
\int_0^{\infty} &\left(\partial_H K_1(x_h,t,s)\right)^2 \diff s  \leq C \, t^{2x_h} \int_0^{\infty} \Bigl( (1+u)^{x_h-1/2}- u^{x_h-1/2} \Bigr)^2 \diff u\\
& + C \, t^{2x_h} \int_0^{\infty} \Big( \big((1+u)^{x_h-1/2}-u^{x_h-1/2}\big) \, \log t 
+ \big( (1+u)^{x_h-1/2}\log(1+u)-u^{x_h-1/2}\log(u) \big) \Big)^2 \diff u  .
\end{align*}
It follows that
\begin{align*}
\int_0^{\infty} &\left(\partial_H K_1(x_h,t,s)\right)^2 \diff s  \leq C  \, t^{2x_h }[\log^2(t)+1]  \\
&\quad \times \sup_{h \in \mathcal{H}} \left( \int_0^{\infty} \left( (1+s)^{h-1/2} - s^{h-1/2} \right)^2 \diff s  + \int_0^{\infty} \left( \log(1+s) (1+s)^{h-1/2} - \log(s) s^{h-1/2} \right)^2 \diff s \right) \\
& \leq C \,  t^{2x_h }[\log^2(t)+1] ~ \sup_{h \in \mathcal{H}} \Bigg( \int_0^{\infty} \left( (1+s)^{h-1/2} - s^{h-1/2} \right)^2 \diff s \\
&\quad + \int_0^{\infty} \left( \log(1+s) - \log(s) \right)^{2} (1+s)^{2h-1} \diff s +  \int_0^{\infty} \log(s)^2 \left( (1+s)^{h-1/2}- s^{h-1/2} \right)^2 \diff s \Bigg) \\
 &=: C \,  t^{2x_h }[\log^2(t)+1] \, \sup_{h \in \mathcal{H}} (I_{1}(h) + I_{2}(h) + I_{3}(h)) .
\end{align*}
Let us detail why $I_2(h)$ is finite for any $h \in (0,1)$. Integrability near $0$ comes from $\left( \log(1+s) - \log(s) \right)^{2} (1+s)^{2h-1} \underset{s\to 0+}{\sim} \log(s)^{2}$. For the integrability between $1$ and $+\infty$,
\begin{align*}
\left| \int_1^{\infty} \left( \log(1+s) - \log(s) \right)^2 (1+s)^{2h-1} \diff s \right| & = \left| \int_1^{\infty} \left( \int_0^1\frac{1 }{(s+r)} \diff r  \right)^2 (1+s)^{2h-1} \diff s \right|  \\
& \leq \int_1^\infty (1+s)^{2h-1} s^{-2} \diff s .
\end{align*} 
Integrability near $+\infty$ follows since $(1+s)^{2h-1} s^{-2} \underset{s\to+\infty}{\sim} s^{2h-3}$ and $h<1$. Hence $I_2(h) < + \infty$. Then one can check using similar techniques that for any $h\in (0,1)$, the integrals $I_{i}(h), \, i=1,3$ are finite. Using the same estimates, $h\mapsto I_{2}(h)$ is continuous on $\mathcal{H}$ by dominated convergence. Similarly, the mappings $h\mapsto I_{i}(h), \, i=1,3$ are continuous on the compact set $\mathcal{H}$.
Since $x_h \leq \max(H,H')$, we thus get
\begin{align}\label{eq:derive-H-F1}
\int_0^{\infty} \left(\partial_H K_1(x_h,t,s)\right)^2 \diff s & \leq C  \left( t^{2H} \vee t^{2H'} \right)[\log^2(t)+1] .
\end{align}

\paragraph{Step 2: Uniform upper bound on $\int_0^{t}  \left( \partial_H K_2(x_h,t,s)\right)^2 \diff s $.\\} 
Recall the expression of $K_2$ in \eqref{eq:noyaux}. Then we get
\begin{align*}
\partial_H K_2(x,t,s) & =  \frac{(t-s)^{x-1/2}\log(t-s)}{\Gamma(x+1/2)} + (t-s)^{x-1/2}\left(\frac{1}{\Gamma}\right)'(x+1/2) 
\end{align*}
and using the boundedness of $\Gamma$ and $(\frac{1}{\Gamma})'$ on $[\frac{1}{2},\frac{3}{2}]$, 
\begin{align*}
\int_0^{t}  \left( \partial_H K_2(x_h,t,s)\right)^2 \diff s & \leq C \int_0^{t}  (t-s)^{2x-1} \left( \log^2(t-s)+ 1  \right) \diff s .
\end{align*}
We now make the change of variables $u= s/t$ to get that
\begin{align*}
\int_0^{t}  \left( \partial_H K_2(x_h,t,s)\right)^2 \diff s & \leq C\, t^{2 x_h} [\log^2(t)+1] \int_0^{1}  (1-u)^{2x-1} \left( \log^2(1-u)+ 1 \right) \diff u .
\end{align*}
It follows that
\begin{align}\label{eq:derive-H-F2}
\int_0^{t}  \left( \partial_H K_2(x_h,t,s)\right)^2 \diff s & \leq C \left( t^{2H}  \vee t^{2H'} \right) [\log^2(t)+1] .
\end{align}

\paragraph{Step 3: Conclusion.} 
Plugging the bounds \eqref{eq:derive-H-F1} and \eqref{eq:derive-H-F2} into \eqref{eq:ub-firstdecomp3} we get
\begin{align*}
\begin{array}{ll}
\mathbb{E}\left( \fBm_t^H - \fBm_{t}^{H'} \right)^2 & \leq C \left( \left( t^{2H} \vee t^{2H'} \right) [\log^2(t)+1] \, | H-H' |^2 \right) .
\end{array}
\end{align*}
Using the above bound in \eqref{eq:first-ub} gives the desired result.
\end{proof}
As stated in the Introduction, the bound from Proposition \ref{prop:ub-justfBm} is not sufficient to deduce regularity estimates of the fBm which are uniform in time (i.e. for all $t\in \R_{+}$). Indeed, for $T \ge 1$, $\mathcal{H}$ a compact set of $(0,1)$ and $\mathcal{K}_T := [0,T] \times \mathcal{H}$, using a GRR argument (Lemma \ref{lem:GRR}), we would obtain as we shall see later that
\begin{align*}
& \mathbb{E} \sup_{\substack{(t,H),\, (t',H') \in \mathcal{K}_{T}\\ (t,H)\neq (t',H')}} \  \frac{| B_t^H - B_{t'}^{H'} |^p}{ \left( |t'-t|^{\min(\mathcal{H})  } + |H-H'| \right)^{p-\varepsilon}  }
\\ & \quad \leq C  \int_0^T \int_0^T \int_{\mathcal{H}} \int_{\mathcal{H}} \frac{\mathbb{E}| \fBm_s^h  -  \fBm_{s'}^{h'} |^p}{\left(|s-s'|^{\min(\mathcal{H})}+ |h-h'| \right)^{p-\varepsilon}} \diff h' \diff h \diff s' \diff s .
\end{align*}
In view of Proposition \ref{prop:ub-justfBm}, letting $T \rightarrow \infty$ yields an infinite bound which does not allow to conclude. That is why we consider the process
\begin{align}\label{eq:def-X-process}
\mathbb{B}_t^H = (1+t)^{-\alpha} B_t^H ,~t\in \R_{+},\, H\in (0,1),
\end{align}
for some $\alpha\geq 0$. The next proposition gives a bound similar to \eqref{eq:variance-ub} for $\mathbb{B}$, which will be useful to obtain the almost sure regularity estimates of the fBm via the GRR Lemma applied to $\mathbb{B}$ (Theorem \ref{thm:whole-regularity}). We will then minimise the value of $\alpha$ to keep $\mathbb{B}^H$ bounded in time: as explained in the Introduction, $\alpha$ is related to the growth of $B^H$ and is therefore expected to be close to $H$.

\begin{prop} \label{prop:variance-ub}
Let $\mathcal{H}$ be a compact subset of $(0,1)$. For $\alpha \ge 0$, consider the multiparameter process $\mathbb{B}$ defined in \eqref{eq:def-X-process}. There exists a constant $C$, independent of $\alpha$, such that
for all $t' \ge t \ge 0$ and $H,H' \in \mathcal{H}$,
\begin{align*}
\mathbb{E}\left( \mathbb{B}_t^H - \mathbb{B}_{t'}^{H'} \right)^2 & \leq C\, (1+t)^{-2 \alpha}  
        \left( (1\vee \alpha^2) |t'-t|^{2 H'} +  \big(t^{2H}\vee t^{2H'}\big) \, (\log^2(t)+1)\,  |H-H'|^2 \right) .
\end{align*}
\end{prop}

\begin{proof}
First we use the pivot term $\mathbb{B}_{t}^{H'}$ to get
\begin{align}\label{eq:VarIncX}
\mathbb{E}\left( \mathbb{B}_t^H - \mathbb{B}_{t'}^{H'} \right)^2 &\leq 2 \mathbb{E}\left( \mathbb{B}_t^H - \mathbb{B}_{t}^{H'} \right)^2 + 2 \mathbb{E}\left( \mathbb{B}_{t}^{H'} - \mathbb{B}_{t'}^{H'} \right)^2 \nonumber\\
&= 2 (1+t)^{-2\alpha}\, \mathbb{E}\left( B_t^H - B_{t}^{H'} \right)^2 + 2 \mathbb{E}\left( \mathbb{B}_{t}^{H'} - \mathbb{B}_{t'}^{H'} \right)^2 .
\end{align}
For the first term in the right-hand side of the previous inequality, use Proposition \ref{prop:ub-justfBm}.\\
For the second term, introduce the pivot term $(1+t')^{-\alpha} B_{t}^{H'}$ to obtain
\begin{align*}
\mathbb{E}\left( \mathbb{B}_t^{H'} - \mathbb{B}_{t'}^{H'} \right)^2  &\leq 2 \left((1+t)^{-\alpha} - (1+t')^{-\alpha}\right)^2\, \EE (B_{t}^{H'})^2 + 2 (1+t)^{-2\alpha}\, \EE \left(B_{t}^{H'} - B_{t'}^{H'}\right)^2 \\
&= 2 \left((1+t)^{-\alpha} - (1+t')^{-\alpha}\right)^2\, t^{2H'} + 2 (1+t)^{-2\alpha}\, |t-t'|^{2H'} .
\end{align*}
If $0\leq t'-t <1+t$, we use the inequality $(1+t)^{-\alpha} - (1+t')^{-\alpha} \leq \alpha\frac{t'-t}{(1+t)^{\alpha+1}}$ to get that for any $h\in [0,1]$,
\begin{align*}
\left((1+t)^{-\alpha} - (1+t')^{-\alpha}\right)^2\, (1+t)^{2h} &\leq \alpha^2 (t'-t)^{2h} \frac{(t'-t)^{2-2h} (1+t)^{2h}}{(1+t)^{2\alpha+2}}\\
&\leq \alpha^2 \frac{(t'-t)^{2h}}{(1+t)^{2\alpha}} .
\end{align*}
Now if $t'-t\geq 1+t$, 
\begin{align*}
\left((1+t)^{-\alpha} - (1+t')^{-\alpha}\right)^2\, (1+t)^{2h} 
&\leq   (1+t)^{-2\alpha} \, (1+t)^{2h}\\
&\leq (1+t)^{-2\alpha} \, (t'-t)^{2h} .
\end{align*}
Hence for any $t'\geq t\geq 0$,
\begin{align}\label{eq:tinc}
\left((1+t)^{-\alpha} - (1+t')^{-\alpha}\right)^2\, t^{2h}\leq \left((1+t)^{-\alpha} - (1+t')^{-\alpha}\right)^2\, (1+t)^{2h} \leq (1\vee \alpha^2) \frac{(t'-t)^{2h}}{(1+t)^{2\alpha}}.
\end{align}
Thus applying \eqref{eq:tinc} to $h=H'$, we get $\mathbb{E}\left( \mathbb{B}_t^{H'} - \mathbb{B}_{t'}^{H'} \right)^2  \leq  2(1+1\vee \alpha^2) (t'-t)^{2H'} (1+t)^{-2\alpha}$, and plugging this inequality in \eqref{eq:VarIncX} gives the result.
\end{proof}

\subsection{Rectangular increments}

For rectangular increments, we obtain results that are similar to Propositions \ref{prop:ub-justfBm} and \ref{prop:variance-ub}.
Recall that these increments are defined in \eqref{eq:rect-increments}. The proofs are postponed to Appendices \ref{app:proof-ub-justfBm-rectangular} and \ref{app:proof-variance-ub-rectangular}.

\begin{prop}\label{prop:ub-justfBm-rectangular}
Let $\mathcal{H}$ be a compact subset of $(0,1)$. There exists a constant $C$ such that for all $t'\ge t \ge 0$ and $H,H' \in \mathcal{H}$,
\begin{align*}
\mathbb{E}\left( \Box_{(t,H)}^{(t',H')} \fBm \right)^2 & \leq C \left( |t'-t|^{2H} \vee |t'-t|^{2H'} \right) (\log^2(|t'-t|)+1)\,   
        |H-H'|^2 .
\end{align*}
\end{prop}

Similarly to Proposition \ref{prop:variance-ub}, we now give an upper bound on the variance of the rectangular increments of the process $\mathbb{B}$. 
\begin{prop} \label{prop:variance-ub-rectangular}
Let $\mathcal{H}$ be a compact subset of $(0,1)$. For $\alpha \ge 0$, consider the multiparameter process $\mathbb{B}$ defined in \eqref{eq:def-X-process}. There exists a constant $C$ such that
 for all $t'\ge t\ge 0$ and $H,H' \in \mathcal{H}$,
\begin{align*}
\mathbb{E}\left( \Box_{(t,H)}^{(t',H')} \mathbb{B} \right)^2 & \leq C (1+t)^{-2 \alpha}\, |H-H'|^2
 \left( |t'-t|^{2H} \vee |t'-t|^{2H'} \right) 
 \left(1+  \log^2|t'-t| + \log^2 t \right) .
\end{align*}
\end{prop}

\section{Almost sure results on the whole half-line for the fBm}\label{sec:mainresults} 

\subsection{Statement of the results}
As mentioned in the Introduction, the joint regularity in $t$ and $H$ of the fBm can be classically obtained on a compact set $[0,T]\times \mathcal{H}$ using Garsia-Rodemich-Rumsey's (GRR) Lemma. This leads to inequality \eqref{eq:simpleInc1} where the random variable $\mathbf{C}_{T}$ depends on (and certainly grows with) $T$. However, as recalled in the paragraph before Proposition~\ref{prop:variance-ub}, this strategy fails if we let $T\to\infty$.
Based on the results of the previous section on the process $\mathbb{B}$ and a multiparameter extension of Garsia-Rodemich-Rumsey's (GRR) Lemma (Section \ref{subsec:GRR}), we obtain  joint regularity estimates in $t\in\R_{+}$ and $H\in \mathcal{H}$ for the usual increments of $B$ in Theorem \ref{thm:whole-regularity}. In particular, this first theorem permits to bound $\mathbf{C}_{T}$ by $\mathbf{C}\, T^{\max(\mathcal{H}) + 2\varepsilon}$.

\begin{theorem}\label{thm:whole-regularity}
Let $\mathcal{H}$ be a compact subset of $(0,1)$. 
Then for any $\varepsilon \in (0,1)$ and any $p\geq 1$,
 there exists a positive random variable $\mathbf{C}$ with a finite moment of order $p$ 
such that almost surely, for all $ t' \ge t \ge 0$ and all $H,H' \in \mathcal{H}$,
\begin{align*}
| \fBm_t^H - \fBm_{t'}^{H'} |  \leq \mathbf{C}\, (1+t')^{2\varepsilon \min(\mathcal{H}) + \max(\mathcal{H})}  \left( 1\wedge |t-t'|^{\min(\mathcal{H})} + |H-H'| \right)^{1-\varepsilon}. 
\end{align*}
\end{theorem}
The proof of Theorem \ref{thm:whole-regularity} is given in Section \ref{subsec:ProofTh1}. 

\begin{remark}\label{rk:optimality}
In the previous theorem, the rates are close to optimality up to the power $\varepsilon$:
\begin{itemize}
\item For $H=H'$, the growth as $t'\to\infty$ is found to be or order $(t')^{2\varepsilon + H}$, while the exact (asymptotic) rate given by the law of the iterated logarithm is $(t')^H \sqrt{2\log\log (t')}$, see \cite{Orey}.
\item For $H=H'$ again, looking at the local oscillations as $t$ and $t'$ are close, the H\"older regularity is also close to optimality, as the local variations of $B^H$ are known to be, asymptotically as $|t'-t|\to 0$, of order $|t'-t|^H \sqrt{2\big|\log |t'-t|\big|}$, see \cite[Eq. (7.6)]{LiShao}.
\item For $t=t'$, the oscillations in the Hurst parameter in the previous theorem are found to be of order $|H-H'|^{1-\varepsilon}$, which is close to optimality up to the power $\varepsilon$. Indeed, it was proved in \cite{KochNeuenkirch} that for each $t$, there exists  a random variable $\tilde{\mathbf{C}}$ such that $|B_{t}^H-B_{t}^{H'}|\leq \tilde{\mathbf{C}} |H-H'|$ a.s.
\end{itemize}
In the first two situations described above, the rates from \cite{Orey,LiShao} are asymptotic, while Theorem~\ref{thm:whole-regularity} holds for any $t'\geq t\geq 0$, so we might not expect better rates. The dependence in $H$ might be shown to be Lipschitz, but the techniques used in \cite{KochNeuenkirch} do not yield bounds that are uniform in $t\in \R_{+}$.
\end{remark}

\begin{remark}\label{rk:regB-}
Theorem \ref{thm:whole-regularity} still holds in $\mathbb{R}_-$: under the assumptions of Theorem \ref{thm:whole-regularity}, we have for all $t' \leq t \leq 0$ and $H,H' \in \mathcal{H}$ that
\begin{align*}
| \fBm_t^H - \fBm_{t'}^{H'} |  \leq \mathbf{C}\, (1+|t'|)^{2 \varepsilon \min(\mathcal{H}) + \max(\mathcal{H})}  \left( 1 \wedge |t-t'|^{\min(\mathcal{H})} + |H-H'| \right)^{1-\varepsilon}.
\end{align*}
This will be used in Lemma \ref{lem:Ubar-increments}.
\end{remark}
Let us focus now on the rectangular increments of $B$ defined in \eqref{eq:rect-increments}. As for the simple increments, we are not aware of any joint regularity results on a non-compact set. Hence similarly to the previous theorem, Theorem~\ref{thm:regularityGarsia-rectangular} below states the regularity both in time and $H$ for the rectangular increments. Unlike the previous theorem, observe that it suffices that only one of the quantities $|t-t'|$ or $|H-H'|$ vanishes for the rectangular increments to vanish as well.
\begin{theorem}\label{thm:regularityGarsia-rectangular}
Let $\mathcal{H}$ be a compact subset of $(0,1)$. Then for any $\varepsilon \in (0,1)$ and  any $p\geq 1$, there exists a positive random variable $\mathbf{C}$ with a finite moment of order $p$ such that almost surely, for all $t' \ge t \ge 0$ and all $H,H' \in \mathcal{H}$,
\begin{align*} 
| \Box_{(t,H)}^{(t',H')} \fBm | \leq \mathbf{C} \,  (1+t')^{2 \varepsilon \min(\mathcal{H})  + \max(\mathcal{H})} \, \left(|H-H'|~ \big( 1  \wedge |t-t'|^{\min(\mathcal{H})} \big) \right)^{1-\varepsilon} .
\end{align*}
\end{theorem}
The proof is given in Appendix \ref{subsec:ProofTh2}. It follows the same argument as in the simple increments case and is based on computations of the moments of $\mathbb{B}$ (see Proposition \ref{prop:ub-justfBm-rectangular}) and a multiparameter GRR lemma which is borrowed from \cite{hu2013multiparameter}. The same comments about optimality of the exponents as in Remark~\ref{rk:optimality} can be formulated here as well.

\smallskip

Finally, as another consequence of Proposition \ref{prop:ub-justfBm-rectangular}, we can also state a result in expectation, uniformly in $H\in \mathcal{H}$. This is a result for all $t$ and $t'$ fixed outside the expectation, so the upper bound does not grow in $(1+t')^{2 \varepsilon + \max(\mathcal{H})}$ as in the previous results.

\begin{proposition}\label{cor:supH-bound}
Let $\mathcal{H}$ be a compact subset of $(0,1)$ and let $q>0$. There exists a constant $C$ such that for all $t',t \ge 0$ and all $H \in \mathcal{H}$, 
\begin{align*}
\mathbb{E} \left( \sup_{H \in \mathcal{H}} | B_t^H - B_{t'}^{H} |^q \right) \leq C \left( |t-t'|^{ q \min(\mathcal{H})} \vee |t-t'|^{ q \max(\mathcal{H})} \right) [\log^q(|t'-t|)+1] .
\end{align*}
\end{proposition}
The proof of Proposition \ref{cor:supH-bound} is given in Section \ref{subsec:proofPropSupH}. Up to the logarithmic factor which is due to the supremum in $H$, observe that the regularity in $t$ is optimal. We also note that this result is useful in the proof of \cite[Proposition 4.3]{HRstat}.

\subsection{Garsia-Rodemich-Rumsey's lemma}\label{subsec:GRR}

The Garsia-Rodemich-Rumsey lemma is usually stated for functions of one parameter \cite{GRR}.
 A more general version is proven in \cite[Lemma 2]{gubinelli2016stochastic}, from which we deduce the following lemma.
 
 \begin{lemma}\label{lem:GRR} 
Let $d\in \N^*$, $a,b \in (0,1]$ and $c \in [0,a \wedge b]$. Define $\delta_{a,b,c}$ the distance on $\R \times (0,1)^d$ given for any $x=(x_{1},\dots,x_{d+1}), \, y=(y_{1},\dots,y_{d+1}) \in \R \times (0,1)^d$ by
\begin{align}\label{eq:defrho}
\delta_{a,b,c}(x,y) = |x_1-y_1|^{a} \wedge |x_1-y_1|^c+  \frac{1}{d} \sum_{i=2}^{d+1} | x_i - y_i|^b .
\end{align}
Let $f:\R \times (0,1)^d \to \R$ be a continuous function (for this metric) and $K$ be a compact subset of $\R \times (0,1)^d$. For any $p>0$ satisfying $p > 2 (a\wedge b)^{-1} (d+1)$, there exists a constant $C>0$ that depends only on $p$, $a$ and $b$, such that
\begin{align*} 
 \sup_{\substack{ x, y \in K\\ x \neq y} } \frac{| f(x)-f(y) | }{\delta_{a,b,c}(x,y)^{1-\frac{2(d+1)}{ (a \wedge b)p}} }  \leq C \left( \int_{K} \int_{K} \frac{| f(z)-f(z') |^p}{\delta_{a,b,c}(z,z')^{ p}} \diff z \diff z'  \right)^{\frac{1}{p}} .
\end{align*}
 \end{lemma}

\begin{proof}
Using the notations of \cite[Lemma 2]{gubinelli2016stochastic}, we take $\psi(u)= u^p$, $d(x,y) = \delta_{a,b,c}(x,y)$ and $m$ the Lebesgue measure. We have to compute a lower bound on $\sigma(r) := \inf_{x \in K} m(B(x,r))$, where $B(x,r)$ is the $\delta_{a,b,c}$-ball centred in $x$ with radius $r$. 
By the concavity of the function $u\mapsto u^{a\wedge b}$, we have that for all $u,v\in [0,1]$,
\begin{align*}
u^a+v^b \leq u^{a\wedge b} + v^{a\wedge b} &\leq 2^{1-a\wedge b} (u+v)^{a\wedge b} \leq 2 (u+v)^{a\wedge b} .
\end{align*}
Moreover, if $u \ge 1$ and $v \in [0,1]$, then $u^c + v^b \leq u^c + v^c \leq 2 (u+v)^{a \wedge b}$. Therefore, and by using Jensen's inequality, we have
\begin{align*}
|x_1-y_1|^{a} \wedge |x_1-y_1|^c+  \frac{1}{d} \sum_{i=2}^{d+1} | x_i - y_i|^b & \leq |x_1-y_1|^{a} \wedge |x_1-y_1|^c+ \left( \frac{1}{d} \sum_{i=2}^{d+1} | x_i - y_i| \right)^b \\
& \leq 2 \left( |x_1-y_1| + \frac{1}{d} \sum_{i=2}^{d+1} | x_i-y_i |  \Big)^{a \wedge b}  \right).
\end{align*}
Thus, for any $r> 0$, the set 
$$  A:= \left\{ y \in K:~ \Big( |x_1-y_1| + \frac{1}{d} \sum_{i=2}^{d+1} | x_i-y_i |  \Big)^{a \wedge b} \leq \frac{r}{2} \right\} $$
is included in $B(x,r)$. 
Hence 
$\sigma(r) \ge c_{a,b}\,  r^{\frac{d+1}{a \wedge b}}$, for some constant $c_{a,b}>0$. Using this inequality in the bound proposed in \cite[Lemma 2]{gubinelli2016stochastic} gives the desired result. 
\end{proof}

\subsection{Proof of Theorem \ref{thm:whole-regularity}}\label{subsec:ProofTh1}
 Denote $a= \min(\mathcal{H})$, $\bar{a}=\max(\mathcal{H})$ and let $\alpha>0$, $\eta \in (0,a)$ and $p > \frac{4}{a\varepsilon} \vee \eta^{-1}$ (one can always choose $p$ as large as necessary).

For any $T\geq 1$, let $\mathcal{K}_{T} := [0,T] \times \mathcal{H}$. Recall that $\mathbb{B}_t^H = (1+t)^{-\alpha} \fBm_t^H$. The idea of the proof is to obtain an upper bound which is independent of $T$ on
\begin{align*}
A_{T} := \mathbb{E} \sup_{\substack{(t,H),\, (t',H') \in \mathcal{K}_{T}\\ (t,H)\neq (t',H')}} \  \frac{| \mathbb{B}_t^H - \mathbb{B}_{t'}^{H'} |^p}{\left(|t'-t|^\eta \wedge |t'-t|^{a} + |H-H'|\right)^{p-\frac{4}{a}}}  .
\end{align*}
Provided $A_{T}$ is bounded in $T\in \R_{+}$, we would then get by monotone convergence that 
\begin{align}\label{eq:boundasinfty}
\sup_{\substack{(t,H),\, (t',H') \in \R_{+}\times \mathcal{H} \\ (t,H)\neq (t',H')}} \  \frac{| \mathbb{B}_t^H - \mathbb{B}_{t'}^{H'} |^p}{\left(|t'-t|^\eta \wedge |t'-t|^{a} + |H-H'|\right)^{p-\frac{4}{a}}} <\infty \quad  a.s.
\end{align}
The last step of the proof will then be to deduce from \eqref{eq:boundasinfty} the result on the increments of $B$.

First, we prove \eqref{eq:boundasinfty}. A consequence of Proposition \ref{prop:variance-ub} and Kolmogorov's continuity criterion (see e.g. \cite[Th. 2.3.1]{Khosh} in the multiparameter setting) is that the process $\{\mathbb{B}^H_{t},\, t\in [0,T],\, H\in \mathcal{H}\}$ has an almost sure continuous modification. We still denote by $\mathbb{B}_{t}^H$ this modification. Here the continuity is to be understood with respect to the distance $\delta_{a,1,\eta}$, defined in \eqref{eq:defrho}. Hence by Lemma~\ref{lem:GRR} (applied here for $a=\min(\mathcal{H})$, $b=1$, $c=\eta$, $d=1$ and $p>\frac{4}{a}$), there exists $C>0$ which does not depend on $T$ and $\mathcal{H}$ such that
\begin{align*} 
A_{T} \leq C \int_0^T \int_0^T \int_{\mathcal{H}} \int_{\mathcal{H}} \frac{\mathbb{E}|(1+s)^{-\alpha} \fBm_s^h  - (1+s')^{-\alpha} \fBm_{s'}^{h'} |^p}{\left(|s-s'|^\eta \wedge |s-s'|^{a} + |h-h'|\right)^{ p}} \diff h' \diff h \diff s' \diff s .
\end{align*}
Since the random variable $(1+s)^{-\alpha} \fBm_s^h  - (1+s')^{-\alpha} \fBm_{s'}^{h'}$ is Gaussian, 
\begin{align}\label{eq:decompAT}
A_{T}  &\leq  C \int_0^T \int_0^T\int_{\mathcal{H}} \int_{\mathcal{H}} \frac{\left(\mathbb{E}|(1+s)^{-\alpha} \fBm_s^h  - (1+s')^{-\alpha}\fBm_{s'}^{h'} |^2\right)^{p/2}}{ \left(|s-s'|^\eta  \wedge |s-s'|^a + |h-h'|\right)^{ p}}  \diff h' \diff h \diff s' \diff s \nonumber \\
&=: C \left( I_{1} + I_{2} + I_{3} \right),
\end{align}
where $I_{1}$ is the integral for $s'$ between $0$ and $(s-1)\vee 0$; $I_{2}$ for $s'$ between $(s-1)\vee 0$ and $(s+1)\wedge T$; and $I_{3}$ for $s'$ between $(s+1)\wedge T$ and $T$.

\paragraph{Bound on $I_1$ and $I_{3}$.} We only write the details for $I_{1}$, as $I_{3}$ can be treated similarly. In both cases one has $|s'-s|\ge 1$, hence $|s-s'|^\eta  \wedge |s-s'|^a + |h-h'| \geq |s-s'|^\eta$.
It comes
\begin{align*}
I_{1} & \leq C \int_0^T \int_0^{(s-1)\wedge 0} \int_{\mathcal{H}} \int_{\mathcal{H}} \left(\mathbb{E}|(1+s)^{-\alpha} \fBm_s^h|^2  + \EE|(1+s')^{-\alpha}\fBm_{s'}^{h'} |^2\right)^{p/2} (s-s')^{-p\eta} \diff h' \diff h \diff s' \diff s \\
&\leq C \int_1^T \int_0^{s-1} \int_{\mathcal{H}} \int_{\mathcal{H}} \left((1+s)^{-p\alpha} s^{ph} + (1+s')^{-p\alpha} (s')^{ph'} \right) (s-s')^{-p\eta} \diff h' \diff h \diff s' \diff s .
\end{align*}
Now use that $h$ and $h'$ are smaller than $\bar{a}$ to get
\begin{align*}
I_{1} & \leq C  \int_1^T \int_0^{s-1} \left( (1+s)^{-p\alpha} s^{p \bar{a}}+ (1+s')^{-p\alpha} \left( (s')^{p \bar{a}} \vee (s')^{pa} \right) \right) (s-s')^{-p\eta} \diff s' \diff s.
\end{align*}
Using now that $p\eta>1$,
\begin{align*}
I_{1} &\leq C  \int_1^T \left(  (1+s)^{-p\alpha} s^{p \bar{a}} (1-s^{1-p\eta}) +  \int_{0}^{s-1}  (1+s')^{-p\alpha+p \bar{a}} (s-s')^{-p\eta}  \diff s' \right) \diff s .
\end{align*}
Now we have $\int_1^T (1+s)^{-p\alpha} s^{p \bar{a}} (1-s^{1-p\eta}) \diff s \leq \int_1^T  (1+s)^{p(\bar{a}-\alpha)} \diff s \leq C (1+T)^{p(\bar{a}-\alpha)+1} $ and 
$ \int_1^T \int_{0}^{s-1}  (1+s')^{-p\alpha+p \bar{a}} (s-s')^{-p\eta}  \diff s'  \diff s = \int_{0}^{T-1} (1+s')^{p(\bar{a}-\alpha)} \int_{s'+1}^T (s-s')^{-p\eta} \diff s\diff s'$. Hence 
\begin{align}\label{eq:I1}
I_{1} \leq C\, (1+T^{p(\bar{a}-\alpha)+1}).
\end{align}
Proceeding similarly, the same bound holds for $I_{3}$.

\paragraph{Bound on $I_2$.} Here we use $|s-s'|\leq 1$ and Proposition \ref{prop:variance-ub} to obtain 
\begin{align*}
& \mathbb{E}|(1+s)^{-\alpha} \fBm_s^h  - (1+s')^{-\alpha}\fBm_{s'}^{h'} |^p \\ 
& \leq C\, \left(1+ s\wedge s'\right)^{p(\bar{a}-\alpha)} \left(1+ \log^2(s\wedge s') \right)^{\frac{p}{2}}  \left( |s'-s|^{p a} +(( s\wedge s')^{2h} \vee ( s\wedge s')^{2h'})  |h-h'|^p  \right) .
\end{align*}
Hence using that $(( s\wedge s')^{2h} \vee ( s\wedge s')^{2h'})  \leq 1+(s \wedge s')^2$, we get
\begin{align}\label{eq:I2}
I_2 &\leq C \int_0^T \int_{(s-1)\vee 0}^{(s+1)\wedge T} \int_{\mathcal{H}} \int_{\mathcal{H}} \left(1+ s\wedge s'\right)^{p(\bar{a}-\alpha)}  \left(1+\log^2(s\wedge s')\right)^{\frac{p}{2}} (1+(s \wedge s')^2 ) \diff h' \diff h \diff s' \diff s \nonumber\\
& \leq C \, (T^{p(\bar{a}-\alpha) +3}  \left(1+ \log^2 T \right)^{\frac{p}{2}}+1).
\end{align}

One can now plug \eqref{eq:I1} and \eqref{eq:I2}  into \eqref{eq:decompAT}, and take $\alpha>\frac{3}{p} + \bar{a}$ to conclude that $A_{T}$ is bounded uniformly in $T\in \R_{+}$. Hence \eqref{eq:boundasinfty} holds true. 

~

Let $\mathbf{C}$ denote the random variable
\begin{equation*}
\mathbf{C} := \sup_{\substack{(t,H),\, (t',H') \in \R_{+}\times \mathcal{H} \\ (t,H)\neq (t',H')}} \  \frac{| \mathbb{B}_t^H - \mathbb{B}_{t'}^{H'} |}{\left(|t'-t|^\eta \wedge |t'-t|^{a} + |H-H'|\right)^{1-\frac{4}{ap}}},
\end{equation*}
Then $\EE|\mathbf{C}|^p \leq \sup_{T>0} A_{T} < + \infty$. Let now $t'\geq t$ and observe that
\begin{align}\label{eq:diffB}
(1+t')^{-\alpha} |\fBm_t^H-\fBm_{t'}^{H'}| \leq |(1+t')^{-\alpha} - (1+t)^{-\alpha} | |\fBm_{t}^{H}| + | (1+t)^{-\alpha} \fBm_t^H - (1+t')^{-\alpha} \fBm_{t'}^{H'} | .
\end{align}

Apply \eqref{eq:boundasinfty} to $(1+t)^{-\alpha} |\fBm_{t}^{H}|$ (with the notations of \eqref{eq:boundasinfty}, take $t'=0$ and $H'=H$) so that for any $t\geq 0$,
\begin{align*}
(1+t)^{-\alpha} |\fBm_{t}^{H}| \leq \mathbf{C} \, \left( t^\eta \wedge t^a \right)^{1-\frac{4}{ap}} .
\end{align*}
Apply again \eqref{eq:boundasinfty} to $| (1+t)^{-\alpha} \fBm_t^H - (1+t')^{-\alpha} \fBm_{t'}^{H'} | $ so that \eqref{eq:diffB} becomes
\begin{align*}
(1+t')^{-\alpha} |\fBm_t^H-\fBm_{t'}^{H'}|  &\leq \mathbf{C} \Big(|(1+t')^{-\alpha} - (1+t)^{-\alpha}| (1+t)^\alpha  \left(t^\eta \wedge t^a \right)^{1-\frac{4}{ap}} \\ 
& \quad + \left(|t'-t|^{\eta} \wedge |t'-t|^a + |H'-H|\right)^{1-\frac{4}{ap}} \Big) .
\end{align*}
Now \eqref{eq:tinc} applied twice (first with $h = \eta(1-\frac{4}{ap})$ and then with $h=a-\frac{4}{p}$) yields
\begin{align*}
(1+t')^{-\alpha} |\fBm_t^H-\fBm_{t'}^{H'}| \leq \mathbf{C} \, \left( C \left(|t'-t|^{\eta(1-\frac{4}{ap})} \wedge |t'-t|^{a-\frac{4}{p}} \right) + \left(|t'-t|^{\eta} \wedge |t'-t|^a + |H'-H|\right)^{1-\frac{4}{ap}} \right) .
\end{align*}
Since we assumed $\eta<a$, it follows that
\begin{align*}
|\fBm_t^H-\fBm_{t'}^{H'}| & \leq \mathbf{C} (1+t')^{\alpha} \left(|t'-t|^{\eta} \wedge |t'-t|^a + |H'-H|\right)^{1-\frac{4}{ap}}  \\
& \leq \mathbf{C} (1+t')^{\alpha + \eta} \left( 1 \wedge| t'-t|^a + |H'-H|\right)^{1-\frac{4}{ap}} .
\end{align*}
Besides, we have $1-\frac{4}{ap} \ge 1-\varepsilon$ and $\left( 1 \wedge| t'-t|^a + |H'-H|\right) \leq 2$,  thus by setting $\eta = \varepsilon a$ and $\alpha= \bar{a}+\varepsilon a$, we conclude that
\begin{align*}
|\fBm_t^H-\fBm_{t'}^{H'}| 
& \leq \mathbf{C} (1+t')^{2 \varepsilon a + \bar{a}}  \left( 1 \wedge|t'-t|^a + |H'-H|\right)^{1-\varepsilon} . 
\end{align*}

\subsection{Proof of Proposition \ref{cor:supH-bound}}\label{subsec:proofPropSupH}
First, notice that it suffices to prove the result for $q$ large enough. Let $t,t'\ge 0$ and denote ${\varphi(H) := B_t^H - B_{t'}^H}$. 
Apply the classical Garsia-Rodemich-Rumsey's lemma with the choice ${\psi(x)=x^q}$, $q>2$ and $p(x)=x$, with the notations of \cite{GRR}. We get
\begin{align*}
\mathbb{E} \sup_{\substack{H,H'\in \mathcal{H}\\H\neq H'}} \frac{| \varphi(H)-\varphi(H') |^q}{|H-H'|^{q-2}} & \leq C \int_{\mathcal{H}} \int_{\mathcal{H}} \frac{\mathbb{E} | \varphi(h) - \varphi(h') |^q}{|h-h'|^{q}} \diff h \diff h' \\
& \leq C \int_{\mathcal{H}} \int_{\mathcal{H}} \frac{\left( \mathbb{E} | \fBm_{t}^h-\fBm_{t'}^h - \fBm_{t}^{h'} + \fBm_{t'}^{h'}|^2 \right)^{q/2}}{|h-h'|^{q}} \diff h \diff h' . 
\end{align*}
In view of Proposition \ref{prop:ub-justfBm-rectangular}, we further obtain
\begin{align*}
\mathbb{E} \sup_{\substack{H,H'\in \mathcal{H}\\H\neq H'}} \frac{| \varphi(H)-\varphi(H') |^q}{|H-H'|^{q-2}} & \leq C \left( |t'-t|^{q\min(\mathcal{H})} \vee |t'-t|^{q\max(\mathcal{H})} \right) [\log^q(|t'-t|)+1]  .
\end{align*}
By fixing a particular $H_{0}\in \mathcal{H}$, it follows that
\begin{align*}
\mathbb{E} \sup_{H\in \mathcal{H}} | \varphi(H) |^q & \leq C \left( |t'-t|^{q\min(\mathcal{H})} \vee |t'-t|^{q\max(\mathcal{H})} \right) [\log^q(|t'-t|)+1] + C\, \mathbb{E}|\varphi(H_{0})|^q \\
& \leq  C \left( |t'-t|^{q\min(\mathcal{H})} \vee |t'-t|^{q\max(\mathcal{H})} \right) [\log^q(|t'-t|)+1] + C\, |t-t'|^{q H_{0}}.
\end{align*}

\section{Regularity of ergodic fractional SDEs}\label{sec:application}

\subsection{Regularity of the solutions}\label{sec:application-reg}
Let $B$ be an $\R^d$-valued fBm, i.e. an $\R^d$-valued process indexed by $(t,H)\in \R_{+}\times (0,1)$ with i.i.d. entries, each having the same law as \eqref{eq:fBm}. 
Consider the $\R^d$-valued SDE:
\begin{align}\label{eq:SDE}
Y^H_{t} &  = Y_{0}  + \int_{0}^t b(Y^H_{s}) \diff s + \fBm_t^H .
\end{align}

\begin{remark}\label{rk:boundedDrift}
When $b$ is a Lipschitz function, this equation has a unique solution on $\R_{+}$. If in addition $b$ is bounded, then applying Gr\"onwall's inequality and Theorem~\ref{thm:whole-regularity} gives that almost surely, for all $ t' \ge t \ge 0$ and all $H,H' \in \mathcal{H}$,
\begin{align*}
| Y_t^H - Y_{t'}^{H'} |  \leq \mathbf{C}\, (1+t')  \left( 1 \wedge |t-t'|^{\min(\mathcal{H})} + |H-H'| \right)^{1-\varepsilon} e^{Ct}. 
\end{align*}
\end{remark}

We seek to improve the previous bound when $b$ is dissipative, that is when
\begin{center}
$b \in \mathcal{C}^{1}(\mathbb{R}^d ;\mathbb{R}^d)$ and there exist constants $\kappa, K > 0$ such that for any $x,y \in \mathbb{R}^d$,
\begin{align}\label{eq:drift-coerciv}
    \langle b(x) - b(y),  x-y \rangle \leq -\kappa | x - y |^2 ~\text{  and  }~ | b(x) - b(y) | \leq K | x-y |.
\end{align} 
\end{center}
Examples of dissipative functions $b$ that satisfy \eqref{eq:drift-coerciv} include $b(x)=-\xi x$ for some $\xi>0$ (in this case $Y$ is the fractional Ornstein-Uhlenbeck process presented in the next paragraph). In general, for $f$ a $K_f$-Lipschitz function with $K_f <\xi$, the function $b(x)= - \xi x + f(x)$ satisfies \eqref{eq:drift-coerciv}. 
~~

First, we analyse the fractional Ornstein-Uhlenbeck (OU) process. Then the main result about $|Y_t^H - Y_{t'}^{H'}|$ is stated in Theorem \ref{thm:regularity-SDE} and is based on a comparison with the OU process.

For $H\in (0,1)$, the fractional OU process is the solution to the following equation:
\begin{align}\label{eq:o-u-def}
U_t^H = U_0^H -\int_0^t U_s^H \diff s +  \fBm_t^H .
\end{align}
We recall that this equation admits a unique solution which can be written as $U_t^H = e^{-t} U_0^H + \int_0^t e^{-(t-s)} d B_s^H$ (see \cite[Proposition A.1]{cheridito2003fractional}). Besides, this process admits a unique stationary solution $\overline{U}^{H}$ given by $\overline{U}^H_{t} = \int_{-\infty}^t e^{-(t-s)} \diff B^H_{s}$, which corresponds to $U_0^H=\int_{-\infty}^0 e^{s} d B_s^H$ (see the beginning of \cite[Section 2]{cheridito2003fractional}). In view of this expression, we  consider $\overline{U}$ as a random field indexed by $(t,H)$. 

\begin{proposition}\label{thm:regularity-OU}
Let $\mathcal{H}$ be a compact subset of $(0,1)$. For $\varepsilon \in (0,1)$ and $p\ge1$, there exists a random variable $\mathbf{C}$ with a finite moment of order $p$ such that almost surely, for all $t' \ge t \ge 0$ and $H,H' \in \mathcal{H}$,
\begin{align*}
|\overline{U}_t^{H} - \overline{U}_{t'}^{H'}| \leq \mathbf{C}\, (1+t')^{\varepsilon} \left( 1 \wedge |t'-t|^{\min(\mathcal{H})} +|H-H'|\right)^{1-\varepsilon} .
\end{align*}
\end{proposition}

\begin{proof}
The scheme of proof is similar to Theorem~\ref{thm:whole-regularity}: we introduce the auxiliary process 
\begin{align}\label{eq:Ubb-def}
\overline{\mathbb{U}}_t^H = (1+t)^{-\alpha} \overline{U}_t^H, 
\end{align}
for some $\alpha \in (0,1)$, compute the moments of $\overline{\mathbb{U}}_t^H-\overline{\mathbb{U}}_{t'}^{H'}$ and finally apply the GRR Lemma. First, we have from Lemma \ref{lem:Ubb-increments} that for $p\geq 1$ and $t' \ge t$, 
\begin{align}\label{eq:upperb-increments-Ubb}
\mathbb{E} | \overline{\mathbb{U}}_t^H-\overline{\mathbb{U}}_{t'}^{H'} |^{p} \leq C  (1+t)^{-p\alpha} \left( 1 \wedge |t-t'|^{p \min(\mathcal{H})} + |H-H'|^{p} \right).
\end{align}
By Kolmogorov's continuity theorem (see e.g. \cite[Th. 2.3.1]{Khosh} in the multiparameter setting) and the previous inequality, $\overline{\mathbb{U}}$ and $\overline{U}$ admit a continuous modification on any compact subset of $\R_{+}\times \mathcal{H}$, and we still denote by $\overline{\mathbb{U}}$ and $\overline{U}$ these continuous modifications. Let $p > \frac{4}{\varepsilon \min(\mathcal{H})}$ and let $\eta \in (\frac{1}{p}, \min(\mathcal{H}))$.
Apply Lemma \ref{lem:GRR} with $a=\min(\mathcal{H})$, $b=1$, $c=\eta$ and $d=1$: we get that for $T \ge 0$,
\begin{align*}
\mathbb{E} \sup_{\substack{t,t' \in [0,T] \\H,H' \in \mathcal{H}\\(t,H)\neq (t',H')}} \frac{|\overline{\mathbb{U}}_t^H-\overline{\mathbb{U}}_{t'}^{H'}|^{p}}{\left( 1\wedge |t-t'|^a + |H-H'| \right)^{p-4/a}}  & \leq C \int_0^T \int_0^T \int_{\mathcal{H}} \int_{\mathcal{H}} \frac{\mathbb{E} | \overline{\mathbb{U}}_s^h-\overline{\mathbb{U}}_{s'}^{h'} |^{p} }{|s-s'|^{ \eta p}\wedge|s-s'|^{ a p} + |h-h'|^{p}} \diff h' \diff h \diff s' \diff s \\
& =: C \left( I_1 + I_2 + I_3 \right),
\end{align*}
where $I_{1}$ is the integral for $s'$ between $0$ and $(s-1)\vee 0$; $I_{2}$ for $s'$ between $(s-1)\vee 0$ and $(s+1)\wedge T$; and $I_{3}$ for $s'$ between $(s+1)\wedge T$ and $T$.

Let us first bound $I_1$ and omit the similar proof for $I_3$. We have
\begin{align*}
I_1 \leq C \int_0^T \int_{0}^{(s-1) \vee 0} \sup_{h,h' \in \mathcal{H}} \mathbb{E} (|\overline{\mathbb{U}}_s^h|^{p} + |\overline{\mathbb{U}}_{s'}^{h'}|^{p})\, |s-s'|^{-\eta p}\, \diff s' \diff s.
\end{align*} 
For any $h\in \mathcal{H}$, $\mathbb{E} |\overline{\mathbb{U}}_t^{h}|^{p} = (1+t)^{-p \alpha} \mathbb{E} |\overline{U}_0^{h}|^{p} \leq C (1+t)^{-p \alpha}$ for some $C$ independent of $h$ in $\mathcal{H}$, hence one gets
\begin{align*}
I_1  & \leq C \int_1^T \int_{0}^{s-1} \Big(  (1+s)^{-p\alpha} + (1+s')^{-p \alpha} \Big) |s-s'|^{-\eta p}\,  \diff s' \diff s\\
& \leq C  \int_1^T \left( (1+s)^{-p \alpha} (1-s^{1-\eta p}) + \int_{0}^{s-1} (1+s')^{-p \alpha} (s-s')^{-\eta p} \diff s' \right) \diff s .
\end{align*}
Using that $\eta p > 1$, we get
\begin{align*}
I_1 \leq C (T^{-p \alpha + 1}+1) ~\mbox{and} ~ I_3 \leq C (T^{-p \alpha + 1}+1).
\end{align*}
Choosing $\alpha \geq \frac{1}{p}$, it follows that $I_1+I_{3}$ is bounded uniformly in $T$.
Let us now bound $I_2$ from above. Using \eqref{eq:upperb-increments-Ubb} gives
\begin{align*}
I_2 & \leq C \int_0^T \int_{(s-1) \vee 0}^{(s+1) \wedge T} (1+ s \wedge s')^{-p \alpha}\diff s \diff s' \\
I_2 & \leq C (T^{-p \alpha + 1} + 1).
\end{align*}
It follows that $I_2$ is bounded uniformly in $T$ since $\alpha \geq \frac{1}{p}$. Hence, by a monotone convergence argument, the random variable
\begin{align*}
\mathbf{C}_1 = \sup_{\substack{t,t' \in \R_+ \\H,H' \in \mathcal{H}\\(t,H)\neq (t',H')}} \frac{|\overline{\mathbb{U}}_t^H-\overline{\mathbb{U}}_{t'}^{H'}|}{\left( |t-t'|^\eta \wedge |t-t'|^a + |H-H'| \right)^{1-\frac{4}{ap}}} 
\end{align*}
has a finite moment of order $p$. Therefore, almost surely for all $t,t' \ge 0$ and $H,H' \in \mathcal{H}$,
\begin{align}\label{eq:asBoundUbb}
| \overline{\mathbb{U}}_t^H - \overline{\mathbb{U}}_{t'}^{H'} | \leq \mathbf{C}_1 \left( |t-t'|^{\eta} \wedge |t-t'|^{a}+|H-H'|\right)^{1-\frac{4}{ap}} .
\end{align}
For $H'=\frac{1}{2}$, $t'=0$ and $t\geq 1$, we have that
\begin{align}\label{eq:asBoundUbb2}
| \overline{\mathbb{U}}_t^H | & \leq \mathbf{C}_1 \left( t^{\eta}+|H-\tfrac{1}{2}|\right)^{1-\frac{4}{ p a}}  + |\overline{\mathbb{U}}_0^{\frac{1}{2}}| \nonumber \\
& \leq \mathbf{C}_2 \, (t+1)^{\eta} ,
\end{align}
where $\mathbf{C}_2 :=\mathbf{C}_1 +|\overline{\mathbb{U}}_0^{\frac{1}{2}}|$ also has a finite moment of order $p$.

Now for $t' \ge t$, one has
\begin{align}\label{eq:1stineqUbar}
(1+t')^{-\alpha}  | \overline{U}_t^H-\overline{U}_{t'}^{H'} | \leq | (1+t)^{-\alpha} - (1+t')^{-\alpha}|\, |\overline{U}_t^H|  + |\overline{\mathbb{U}}_t^H - \overline{\mathbb{U}}_{t'}^{H'}|.
\end{align}
In view of \eqref{eq:asBoundUbb2}, $|\overline{U}_t^H| \leq \mathbf{C}_2 (1+t)^{\alpha+\eta} \leq \mathbf{C}_{2} (1+t)^{1+\eta}$. Hence using \eqref{eq:tinc} with $h=1$ yields $| (1+t)^{-\alpha} - (1+t')^{-\alpha}|\, |\overline{U}_t^H| \leq \mathbf{C}_2 (t'-t)\, (1+t)^{-\alpha+\eta}$. If $t'-t\geq 1$, it also comes directly that $| (1+t)^{-\alpha} - (1+t')^{-\alpha}|\, |\overline{U}_t^H| \leq \mathbf{C}_2 (1+t)^{\eta}$. Thus we get that for any $t'\geq t$, 
\begin{align}\label{eq:2ndIneqUbar}
| (1+t)^{-\alpha} - (1+t')^{-\alpha}|\, |\overline{U}_t^H| \leq \mathbf{C}_2  (1+t)^{\eta} \left(1\wedge (t'-t)\right).
\end{align} 
It remains to plug \eqref{eq:2ndIneqUbar} and \eqref{eq:asBoundUbb} in \eqref{eq:1stineqUbar} to obtain that
\begin{align*}
 | \overline{U}_t^H-\overline{U}_{t'}^{H'} | \leq \mathbf{C}_2 (1+t)^{\alpha+\eta} (1 \wedge |t'-t|) + \mathbf{C}_1 \left( |t-t'|^{\eta} \wedge |t-t'|^{a}+|H-H'|\right)^{1-\frac{4}{ap}} .
\end{align*}
Since $\mathbf{C}_2 \ge \mathbf{C}_1$ and $p > \frac{4}{a \varepsilon}$, we have
\begin{align*}
 | \overline{U}_t^H-\overline{U}_{t'}^{H'} | \leq \mathbf{C}_2 (1+t)^{\alpha+\eta} \left( 1 \wedge |t-t'|^{a}+|H-H'|\right)^{1-\varepsilon} .
\end{align*}
Choosing $\alpha = \frac{\varepsilon}{2}$ (which satisfies the constraint $\alpha \geq \frac{1}{p})$ and $\eta=\frac{\varepsilon}{2}$ (which satisfies $\eta>\frac{1}{p}$) yields
\begin{align*}
 | \overline{U}_t^H-\overline{U}_{t'}^{H'} | \leq \mathbf{C}_2 (1+t)^{\varepsilon} \left( 1 \wedge |t-t'|^{a}+|H-H'|\right)^{1-\varepsilon} .
\end{align*}
\end{proof}

We can now deduce the long time regularity in $H$ of the solution to \eqref{eq:SDE} by comparing its increments to those of the stationary fractional OU process.

\begin{theorem}\label{thm:regularity-SDE}
Let $\mathcal{H}$ be a compact subset of $(0,1)$. For each $H \in \mathcal{H}$, let $Y^{H}$ be the solution to \eqref{eq:SDE}, with the drift $b$ satisfying \eqref{eq:drift-coerciv}.
\begin{enumerate}[(i)]
\item  Let $\varepsilon \in (0,1)$ and $p \ge 1$. There exists a random variable $\mathbf{C}$ with a finite moment of order $p$ such that almost surely, for all $t' \ge t \ge 0$ and $H,H' \in \mathcal{H}$,
\begin{align*}
|Y_t^{H} - Y_{t'}^{H'}| \leq \mathbf{C}\, (1+t)^{\varepsilon}\, \left( 1 \wedge |t'-t|^{\min(\mathcal{H})} +|H-H'|\right)^{1-\varepsilon} .
\end{align*}
\item Let $p \ge 1$. There exists a constant $C$ such that for all $t' \ge t \ge 0$ and $H,H' \in \mathcal{H}$,
\begin{align*}
 \mathbb{E} \Big|Y_t^{H} - Y_{t'}^{H'}\Big|^p \leq C\, \left(1 \wedge |t'-t|^{p \min(\mathcal{H})} + |H-H'|^p\right).
\end{align*}
\end{enumerate}
\end{theorem}

\begin{remark}
Since \eqref{eq:SDE} is an SDE with additive noise, the regularity in $t'-t$ and $H'-H$ in $(i)$ of the previous theorem comes directly from the regularity of the noise, thus the same comments about optimality of the exponents as in Remark~\ref{rk:optimality} can be formulated here as well. \\
As for the behaviour when $t$ increases, it is known for the stationary Ornstein-Uhlenbeck process that the running maximum of this process behaves almost surely asymptotically as $\sqrt{2\log(t)}$ (see \cite{Pickands}). So the term $(1+t)^\varepsilon$ we obtain might be slightly sub-optimal, but it does hold for any $t\geq 0$.\\
As for the moments of the increments obtained in $(ii)$, we see that the term $(1+t)^{\varepsilon}$ is no longer here. The Lipschitz regularity in the Hurst parameter is optimal, as well as the time regularity, which is exactly that of the fBm.
\end{remark}

\begin{proof}
The proof is based on a comparison with the Ornstein-Uhlenbeck process defined in~\eqref{eq:o-u-def} with initial condition $Y_0$. Recall that $\overline{U}^H$ denotes the stationary solution of \eqref{eq:o-u-def}. Let $t' \ge t \ge 0$ and $H,H' \in \mathcal{H}$, then
\begin{align}\label{eq:triangle-ineq}
| Y_{t'}^{H'} -Y_t^H | \leq | Y_{t'}^{H'} -Y_t^{H'} | + | Y_{t}^{H} -Y_t^{H'} | .
\end{align}

\paragraph{Bound on $| Y_{t'}^{H'} -Y_t^{H'} |$.} We proceed slightly differently depending whether $t'-t \leq 1$ or $t'-t > 1$.
First, assume that $t'-t \leq 1$. Using \eqref{eq:drift-coerciv}, there is
\begin{align}
 | Y_{t'}^{H'} -Y_t^{H'} -U_{t'}^{H'}+U_{t}^{H'}| 
 & \leq \int_t^{t'} |b(Y_s^{H'})-b(U_s^{H'})| \diff s + \int_t^{t'} \left( | b(U_s^{H'})| + |U_s^{H'}| \right) \diff s \nonumber  \\
 & \leq \int_t^{t'} K |Y_s^{H'}-U_s^{H'}| \diff s + \int_t^{t'} \left( K |U_s^{H'}|+|b(0)| + |U_s^{H'}| \right) \diff s \nonumber \\
 & \leq \int_t^{t'} K |Y_s^{H'}-U_s^{H'}| \diff s+ C \left( \int_t^{t'} |U_s^{H'}|  \diff s+ |t'-t| \right) . \label{eq:time-reg-sde-0}
\end{align}
Use \eqref{eq:SDE} and \eqref{eq:o-u-def} to get that $Y_s^{H'}-U_s^{H'} = \int_0^s ( b(Y_r^{H'}) + U_r^{H'} ) \diff r$. Since $b$ is locally bounded and $Y^{H'}$ and $U^{H'}$ are continuous, the process $y\mapsto Y_s^{H'}-U_s^{H'} = \int_0^s ( b(Y_r^{H'}) + U_r^{H'} ) \diff r$ is differentiable. Hence using first \eqref{eq:drift-coerciv} and then Young's inequality, we get
\begin{align*}
\frac{\diff}{\diff t} |Y_s^{H'}-U_s^{H'} |^2 & = 2 \langle Y_s^{H'}-U_s^{H'}, b(Y_s^{H'})-b(U_s^{H'})+b(U_s^{H'})+U_s^{H'} \rangle \\
& \leq -2 \kappa |Y_s^{H'}-U_s^{H'}|^2 + \kappa  |Y_s^{H'}-U_s^{H'}|^2 + \frac{1}{\kappa} |b(U_s^{H'})+U_s^{H'}|^2 \\
& \leq - \kappa  |Y_s^{H'}-U_s^{H'}|^2 + C \left( |U_s^{H'}|^2 +1 \right) .
\end{align*}
It follows that
\begin{align}
|Y_s^{H'}-U_s^{H'} |^2 \leq C \int_0^s e^{-\kappa(s-r)}  \left( |U_r^{H'}|^2 +1 \right) \diff r . \label{eq:time-compar-ou}
\end{align}
Injecting \eqref{eq:time-compar-ou} in \eqref{eq:time-reg-sde-0} gives
\begin{align*}
 | Y_{t'}^{H'} -Y_t^{H'} -U_{t'}^H+U_{t}^H| & \leq C \int_t^{t'}  \left( \int_0^s e^{-\kappa(s-r)}  \left( |U_r^{H'}|^2 +1 \right) \diff r\right)^{\frac{1}{2}}  \diff s+ C \left( \int_t^{t'} |U_s^{H'}| \diff s+ |t'-t| \right) 
\end{align*}
and
\begin{align*}
\begin{split}
 | Y_{t'}^{H'} -Y_t^{H'}| & \leq C \int_t^{t'}  \left( \int_0^s e^{-\kappa(s-r)}  \left( |U_r^{H'}|^2 +1 \right) \diff r \right)^{\frac{1}{2}} ds+ C \left( \int_t^{t'} |U_s^{H'}| \diff s+ |t'-t| \right) \\ 
 & \quad +|U_{t'}^H-U_{t}^H| .
 \end{split}
\end{align*}
Using $U_{t'}^{H'}-\overline{U}_{t'}^{H'} = e^{-t'}(Y_0-\overline{U}_{0}^{H'})$ and
$U_{t'}^{H'}-U_t^{H'} - \overline{U}_{t'}^{H'}+\overline{U}_t^{H'} = (e^{-t'}-e^{-t}) ( Y_0 -\overline{U}_0^{H'})$ yields that for $t'-t \leq 1$,
\begin{align}
\begin{split}
 | Y_{t'}^{H'} -Y_t^{H'}| & \leq C \int_t^{t'}  \left( \int_0^s e^{-\kappa(s-r)}  \left( |\overline{U}_r^{H'}|^2 +1 \right) \diff r \right)^{\frac{1}{2}} \diff s+ C \left( \int_t^{t'} |\overline{U}_s^{H'}| \diff s+ |t'-t| \right) \\ 
 & \quad +|\overline{U}_{t'}^{H'}-\overline{U}_{t}^{H'}| +C \left( 1+|Y_0|+|\overline{U}_0^{H'}| \right) |t'-t|  . \label{eq:time-reg-sde-1}
 \end{split}
\end{align}
Now assume $t'-t \ge 1$. Using \eqref{eq:time-compar-ou}, 
\begin{align*}
| Y_{t'}^{H'} -Y_t^{H'}| & \leq |Y_{t'}^{H'}-U_{t'}^{H'}| + |Y_{t}^{H'}-U_t^{H'}| + |U_{t'}^{H'} -U_{t}^{H'} | \nonumber \\
\begin{split}
& \leq C \left( \int_0^{t'} e^{-(t'-r)} \left( |U_r^{H'}|^2+1 \right) \diff r \right)^{\frac{1}{2}} +  C \left( \int_0^{t} e^{-(t-r)} \left( |U_r^{H'}|^2+1 \right) \diff r \right)^{\frac{1}{2}}  \\ 
& \quad +  |U_{t'}^{H'} -U_{t}^{H'} | 
\end{split}
\end{align*}
and as before comparing to the stationary OU process gives
\begin{align}
\begin{split}
| Y_{t'}^{H'} -Y_t^{H'}| 
& \leq C \left( \int_0^{t'} e^{-(t'-r)} \left( |\overline{U}_r^{H'}|^2+1 \right) \diff r \right)^{\frac{1}{2}} +  C \left( \int_0^{t} e^{-(t-r)} \left( |\overline{U}_r^{H'}|^2+1 \right) \diff r \right)^{\frac{1}{2}}  \\ 
& \quad +  |\overline{U}_{t'}^{H'} -\overline{U}_{t}^{H'} |+ C \left( 1+|Y_0|+|\overline{U}_0^{H'}| \right) .  \label{eq:time-reg-sde-2}
\end{split}
\end{align}

\paragraph{Bound on $| Y_{t}^{H} -Y_t^{H'} |$.}
We will use the following set of notations: for $H, H'\in (0,1)$,
\begin{align*}
\mathcal{B} = \fBm^{H} - \fBm^{H'},~ \mathcal{U} = U^{H}-U^{H'} ~\mbox{and }~ \overline{\mathcal{U}} = \overline{U}^{H}-\overline{U}^{H'}.
\end{align*}
The idea is to establish the following inequality: with $\kappa$ from \eqref{eq:drift-coerciv},
\begin{align}\label{eq:comparison-o-u}
| Y_t^{H}-Y_t^{H'} |^{2p} 
& \leq C \left( e^{-2pt}  | \overline{\mathcal{U}}_{0} |^{2p} + | \overline{\mathcal{U}}_{t} |^{2p} +  C \int_0^t e^{-\kappa(t-r)} |\overline{\mathcal{U}}_{r}  |^{2p} \diff  r \right) .
\end{align}
By the triangle inequality, it suffices to prove that
\begin{align*}%
| Y_t^{H}-Y_t^{H'} - \overline{\mathcal{U}}_{t} |^{2p}  \leq C  \int_0^t e^{-\kappa (t-r)}  |\overline{\mathcal{U}}_{r}|^{2p}\, \diff r.
\end{align*}
We work first with the nonstationary process $\mathcal{U}$. 
As before, the process $t\mapsto Y_t^{H}-Y_t^{H'} - \mathcal{U}_{t} = \int_{0}^t (b(Y^{H}_{s}) - b(Y^{H'}_{s}) + U^{H}_{s} -U^{H'}_{s}) \diff s$ is differentiable. Hence \eqref{eq:drift-coerciv} yields
\begin{align*}
\frac{\diff}{\diff t} | Y_t^{H}-Y_t^{H'} - \mathcal{U}_{t}|^2  & = 2 \langle  Y_t^{H}-Y_t^{H'} - \mathcal{U}_{t}, \, b(Y_t^{H})-b(Y_t^{H'}) + \mathcal{U}_{t} \rangle \\
& \leq - 2 \kappa |Y_t^{H}-Y_t^{H'} |^2 - 2 |\mathcal{U}_{t}|^2 + 2(1+K) | Y_t^{H}-Y_t^{H'}| |\mathcal{U}_{t}| .
\end{align*}
Apply Young's inequality to get
\begin{align*}
2(1+K) | Y_t^{H}-Y_t^{H'}|\, |\mathcal{U}_{t}| \leq \kappa |Y_t^{H}-Y_t^{H'}|^2 + \frac{(1+K)^2}{\kappa} |\mathcal{U}_{t}|^2 .
\end{align*}
It follows that
\begin{align*}
\frac{\diff}{\diff t} | Y_t^{H}-Y_t^{H'} - \mathcal{U}_{t}|^2 
& \leq - \kappa \left(|Y_t^{H}-Y_t^{H'} |^2 + |\mathcal{U}_{t}|^2\right) + \left(\kappa-2 + \frac{(1+K)^2}{\kappa}\right) |\mathcal{U}_{t}|^2 \\
& \leq -\kappa  | Y_t^{H}-Y_t^{H'} - \mathcal{U}_{t}|^2 + C |\mathcal{U}_{t}|^2 .
\end{align*}
Hence Gr\"onwall's lemma followed by Jensen's inequality give
\begin{align*}
| Y_t^{H}-Y_t^{H'} - \mathcal{U}_{t}|^{2p} \leq C \int_0^t e^{-\kappa(t-r)} |\mathcal{U}_{r}|^{2p} \diff  r .
\end{align*}
Now a similar inequality holds with the stationary OU process:
\begin{align*}
| Y_t^{H}-Y_t^{H'} |^{2p} & \leq C | \mathcal{U}_{t} |^{2p} +  C \int_0^t e^{-\kappa(t-r)} |\mathcal{U}_{r}|^{2p} \diff  r \\
& \leq  C\, |\overline{\mathcal{U}}_{t} |^{2p} +  C \int_0^t e^{-\kappa(t-r)} |\overline{\mathcal{U}}_{r}|^{2p} \diff  r   + C\, | \overline{\mathcal{U}}_{t} - \mathcal{U}_{t}  |^{2p} +  C \int_0^t e^{-\kappa(t-r)} |\overline{\mathcal{U}}_{r} - \mathcal{U}_{r} |^{2p} \diff  r  .
\end{align*}
Notice that $U^{H}_{t} = e^{-t}Y_{0} + \int_{0}^t e^{-(t-s)} \diff B^{H}_{s}$ and $\overline{U}^{H}_{t} = e^{-t}\overline{U}^{H}_{0} + \int_{0}^t e^{-(t-s)} \diff B^{H}_{s}$. Hence 
$\mathcal{U}_t - \overline{\mathcal{U}}_t = e^{-t}\, \overline{\mathcal{U}}_0$ and we have \eqref{eq:comparison-o-u}.

\paragraph{Proof of $(i)$.} Let $\varepsilon \in (0,1)$ and $p \ge 1$. By
 Proposition \ref{thm:regularity-OU}  there exists a random variable $\mathbf{C}$ with a finite moment of order $p$ such that almost surely
\begin{align*}
| \mathcal{\overline{U}}_t | & \leq \mathbf{C} (1+t)^{\varepsilon} |H-H'|^{1-\varepsilon} \\
| \overline{U}_{t'}^{H'} -\overline{U}_t^{H'} | & \leq  \mathbf{C} (1+t')^{\varepsilon} \left( 1 \wedge |t'-t|^{H'} \right)^{1-\varepsilon} \\
 | \overline{U}_{t'}^{H'}  |  & \leq \mathbf{C} (1+t')^{\varepsilon} .
\end{align*}
Use the previous inequalities in \eqref{eq:time-reg-sde-1} when $t'-t\leq 1$, in \eqref{eq:time-reg-sde-2} when $t'-t\leq 1$ to get $|Y^{H'}_{t'} - Y^{H'}_{t}| \leq \mathbf{C}\, (1+t)^{\varepsilon}\, ( 1 \wedge |t'-t|^{\min(\mathcal{H})(1-\varepsilon)})$. Then use the previous inequalities in 
 \eqref{eq:comparison-o-u} to get $|Y^{H}_{t} - Y^{H'}_{t}| \leq \mathbf{C}\, (1+t)^{\varepsilon}\, |H-H'|^{1-\varepsilon} $ and recall \eqref{eq:triangle-ineq} to deduce the result of $(i)$.
\paragraph{Proof of $(ii)$.} To prove $(ii)$, apply Lemma \ref{lem:Ubar-increments} to  \eqref{eq:time-reg-sde-1}, \eqref{eq:time-reg-sde-2} and \eqref{eq:comparison-o-u}.
\end{proof}

\subsection{Regularity of ergodic means}\label{sec:application-ergodic}

Under the dissipative assumption \eqref{eq:drift-coerciv} on the drift $b$, the SDE \eqref{eq:SDE} has a unique invariant measure $\mu_H$ (see \textit{e.g.} \cite{Hairer} and \cite[Lemma 3$(ii)$]{cohen2011approximation}). Moreover, from \cite[Proposition 3.3]{panloup2020general}:
\begin{align}\label{eq:converg} 
 \lim_{t \rightarrow \infty}d\Big( \frac{1}{t} \int_{0}^t \delta_{Y_s^H} \diff s, \mu_H\Big) = 0 \quad a.s.,
\end{align}
where $d$ is any distance bounded by a $p$-Wasserstein distance.

For the modelling purposes discussed in the introduction, it is interesting to know the sensitivity in $H$ of $\mu_{H}$ and of the ergodic means. Besides, knowing the H\"older regularity of $H \mapsto d( \frac{1}{t} \int_{0}^t \delta_{Y_s^H} \diff s, \mu_H)$ permits to prove the the convergence in \eqref{eq:converg} uniformly in $H$, which in turns has a direct statistical application in the estimation of the Hurst parameter (see \cite[proof of Proposition 4.3]{HRstat}).
To do so, for $d$ a distance bounded by the $p$-Wasserstein distance, it comes:
\begin{align*}
\Big| d\Big( \frac{1}{t} \int_{0}^t \delta_{Y_s^H} \diff s, \mu_H\Big)-d\Big( \frac{1}{t} \int_{0}^t \delta_{Y_s^{H'}} \diff s, \mu_{H'}\Big) \Big| 
&\leq d( \mu_H, \mu_{H'})  + d \Big(\frac{1}{t} \int_{0}^t \delta_{Y_s^H} \diff s, \frac{1}{t} \int_{0}^t \delta_{Y_s^{H'}} \diff s\Big) \\
& \leq d( \mu_H, \mu_{H'}) + \Big(\frac{1}{t} \int_0^t |Y_s^H - Y_s^{H'} |^p \diff s \Big)^{\frac{1}{p}}.
\end{align*}
The main result of this section (Theorem \ref{th:ergodic-OU}) is the regularity in $|H-H'|$ of  $\frac{1}{t} \int_0^t |Y_s^H - Y_s^{H'} |^2 \diff s$. Then by taking the limit as $t\to +\infty$, we deduce in Corollary \ref{cor:reg-invariant} the regularity of $d( \mu_H, \mu_{H'})$.

First, observe that integrating the result of Theorem~\ref{thm:regularity-SDE}, there exists a random variable $\mathbf{C}$ such that almost surely, for all $t>0$ and $H,H' \in \mathcal{H}$, there is 
\begin{align*}
\frac{1}{t} \int_0^t |Y_s^H - Y_s^{H'} |^{2} \, \diff s \leq \mathbf{C} (1+t)^{2 \varepsilon } |H-H'|^{2-2\varepsilon} . 
\end{align*}
This result is not optimal since for any $H$, $Y^H$ is ergodic and so the left hand side converges as $t \to \infty$. Thus in Theorem \ref{th:ergodic-OU}, we will follow a different approach to get rid of the term $(1+t)^{2\varepsilon}$. We shall use again a GRR argument and rely on the fact that the variance of the ergodic means decreases over time (Lemma \ref{lem:young-V}). As in Section \ref{sec:application-reg}, we start by studying the ergodic means of the stationary fractional OU process and then rely on a comparison to conclude.

\begin{proposition}\label{prop:ergodic-OU}
Let $\mathcal{H}$ be a compact subset of $(0,1)$. Recall that $\overline{U}^{H}$ is the stationary Ornstein-Uhlenbeck. Let $ \beta \in  (0,1)$, and $p \ge 1$. There exists a random variable $\mathbf{C}$ with finite moment of order $p$ such that almost surely, for all $t,t' \ge 0$ and $H,H',K,K' \in \mathcal{H}$,
\begin{align*}
\Big| \frac{1}{t+1} \int_0^{t+1} | \overline{U}_s^{H} &-\overline{U}_s^{K}|^2 \diff s -  \frac{1}{t'+1} \int_0^{t'+1} |\overline{U}_{s}^{H'}-\overline{U}_{s}^{K'}|^2 \diff s \Big|  \\
& \leq \mathbf{C} \, (1+|t-t'|^{1-\beta}) \left( 1  \wedge |t-t'| + |H-H'| + |K-K'| \right)^{\beta}.
\end{align*}
\end{proposition}

\begin{proof}
For $t\geq 0$ and $H,K\in \mathcal{H}$, define
\begin{align}\label{eq:defV}
V_t^{H,K} = \frac{1}{t+1} \int_0^{t+1} | \overline{U}_s^H - \overline{U}_s^K |^2 \diff s - \mathbb{E} |\overline{U}_0^H - \overline{U}_0^K |^2 ,
\end{align}
and observe that
\begin{align*}
\Big| \frac{1}{t+1} \int_0^{t+1} | \overline{U}_s^{H}-\overline{U}_s^{K}|^2 \diff s &-  \frac{1}{t'+1} \int_0^{t'+1} |\overline{U}_{s}^{H'}-\overline{U}_{s}^{K'}|^2 \diff s \Big|\\
&\leq |V_t^{H,K}-V_{t'}^{H',K'}| +  \mathbb{E} \left| | \overline{U}_0^H - \overline{U}_0^K |^2 - | \overline{U}_0^{H'} - \overline{U}_0^{K'} |^2 \right|.
\end{align*}
By Lemma \ref{lem:Ubar-increments}, we have $\EE| \overline{U}_0^H - \overline{U}_0^{H'} |^2 \leq C |H-H'|^2$ and $\EE|\overline{U}_0^H|^2 \leq C$. Hence
\begin{align*}
 \mathbb{E} \Big| | \overline{U}_0^H - \overline{U}_0^K |^2 - |\overline{U}_0^{H'} - \overline{U}_0^{K'} |^2 \Big| &=  \mathbb{E} \Big| \langle \overline{U}_0^H - \overline{U}_0^{H'} - (\overline{U}_0^K- \overline{U}_0^{K'}) , \,  \overline{U}_0^H + \overline{U}_0^{H'} - (\overline{U}_0^K+ \overline{U}_0^{K'}) \rangle \Big| \nonumber\\
& \leq  C\,\Big(\EE| \overline{U}_0^H - \overline{U}_0^{H'}|^2\Big)^{\frac{1}{2}} + C\, \Big(\EE|\overline{U}_0^K- \overline{U}_0^{K'}|^2 \Big)^{\frac{1}{2}}  \nonumber\\
& \leq C\, (|H-H'| + |K-K'|).
\end{align*}
Hence it remains to bound $|V_t^{H,K}-V_{t'}^{H',K'}|$. Let $T \ge 1$, $\beta< 1$ and $\varepsilon\in (0,1)$. Let $p \in \mathbb{N}$ be even with
\begin{align*}
p > \frac{6}{\varepsilon \beta} \quad \text{and} \quad p > \frac{3}{\Big( 4(1-\max(\mathcal{H})) \wedge 1 \Big) (1-\beta)} .
\end{align*}
Apply Lemma \ref{lem:GRR} with $a=1$, $b=\beta$, $c=\eta \in (\frac{1}{p},\beta)$ and $d=2$ to get that
\begin{equation}\label{eq:GRRV}
\begin{split}
& \mathbb{E} \sup_{\substack{ t, t' \in [0,T] \\ H,H',K,K' \in \mathcal{H} \\ (t,H,K) \neq (t',H',K') } } \frac{| V_t^{H,K}-V_{t'}^{H',K'} |^p }{(|t-t'|^\eta \wedge |t-t'|+ \frac{1}{2}|H-H'|^\beta + \frac{1}{2} |K-K'|^\beta)^{p-\frac{6}{\beta}} } \\ & 
\leq C \int_{[0,T]^2\times\mathcal{H}^4} \frac{\mathbb{E}| V_s^{h,k}-V_{s'}^{h',k'} |^p}{(|s-s'|^\eta \wedge |s-s'|+ \frac{1}{2} |h-h'|^{\beta} + \frac{1}{2}|k-k'|^\beta )^{ p}} \diff h' \diff k' \diff h \diff k \diff s' \diff s  =: C \, A_{T} .
\end{split}
\end{equation}
Decompose $A_T$ as
\begin{align*}
A_T = I_1 + I_2 + I_3 \, ,
\end{align*}
where $I_{1}$ is the integral for $s'$ between $0$ and $(s-1)\vee 0$; $I_{2}$ for $s'$ between $(s-1)\vee 0$ and $(s+1)\wedge T$; and $I_{3}$ for $s'$ between $(s+1)\wedge T$ and $T$. 

\smallskip

For $I_1$, we have
\begin{align*}
I_1 & \leq C \int_1^T \int_0^{s-1} \int_{\mathcal{H}^4}\Big( \mathbb{E}|V_s^{h,k}|^p + \mathbb{E} | V_{s'}^{h',k'} |^p \Big) |s-s'|^{-\eta p}\,  \diff h' \diff k' \diff h \diff k \diff s' \diff s .
\end{align*}
In view of Lemma \ref{lem:young-V}, there is
\begin{align*}
I_1  \leq &\, C\,  \int_1^T \int_0^{s-1}  \Big( (\log(s+1)+1)^p (s+1)^{-\frac{p}{3} \left( 1 \wedge (4-4\max(\mathcal{H}))  \right) } \\ 
&+(\log(s'+1)+1)^p (s'+1)^{-\frac{p}{3} \left( 1 \wedge (4-4\max(\mathcal{H}))  \right) } \Big)  (s-s')^{-\eta p} \diff s' \diff s  .
\end{align*}
Proceeding with the same computations for $I_{3}$ and using that $\eta p > 1$, we get
\begin{align*}
I_1 + I_{3} \leq C\, \left( 1 + T^{-\frac{p}{3} \left( 1 \wedge (4-4\max(\mathcal{H})) \right)  + 1}  (\log(T+1)+1) ^p \right) .
\end{align*}
For $I_2$, we always have $|s-s'|^\eta \wedge |s-s'| = |s-s'|$. We use Lemma \ref{lem:ergodic-U2} to get
\begin{align*}
I_2 & \leq C\, \int_0^T \int_{(s-1)\vee 0}^{(s+1) \wedge T} \Big( (\log(s \wedge s'+1)+1)^p \\ 
& \quad \times (1+ s \wedge s')^{-\frac{p}{3} \left( 1 \wedge (4-4\max(\mathcal{H})) \right) (1-\beta)} + (1+s\wedge s')^{-p}\Big) \diff s' \diff s \\
& \leq C\, \left(1+ T^{-\frac{p}{3} \left( 1 \wedge (4-4\max(\mathcal{H})) \right) (1-\beta) + 1 } (\log(T+1)+1)^p + T^{-p+1} \right) .
\end{align*}
Since $p > \frac{3}{ \left( 4(1-\max(\mathcal{H})) \wedge 1 \right) (1-\beta)}$ and $\beta\in (0,1)$, 
\begin{align*}
-\frac{p}{3}\left(1 \wedge (4-4\max(\mathcal{H})) \right) +1  & < -\frac{p}{3} \left(1 \wedge (4-4\max(\mathcal{H} ) \right) (1-\beta) + 1 <0.
\end{align*}
Hence the powers of $T$ in the upper bounds on $I_1,I_2$ and $I_3$ are negative, thus $I_1,I_2$ and $I_3$ are bounded uniformly in $T$. In view of \eqref{eq:GRRV}, it follows that for any $T>0$,
\begin{align*}
 \mathbb{E} \sup_{\substack{ t, t' \in [0,T] \\ H,H',K,K' \in \mathcal{H} \\ (t,H,K) \neq (t',H',K') } } \frac{| V_t^{H,K}-V_{t'}^{H',K'} |^p}{(|t-t'| \wedge |t-t'|^\eta + |H-H'|^\beta + |K-K'|^\beta)^{p-\frac{6}{ \beta}} }  \leq C .
\end{align*} 
Thus by a monotone convergence argument, letting $T \to + \infty$, the random variable
\begin{align*}
\mathbf{C}_1 := \sup_{\substack{ t, t' \in \R_+ \\ H,H',K,K' \in \mathcal{H} \\ (t,H,K) \neq (t',H',K') } } \frac{| V_t^{H,K}-V_{t'}^{H',K'} | }{(|t-t'| \wedge |t-t'|^\eta + |H-H'|^\beta + |K-K'|^\beta)^{1-\frac{6}{p \beta}} } 
\end{align*}
has a finite moment of order $p$. 
Since $p> \frac{6}{\beta \varepsilon}$ and
\begin{align*}
(|t-t'| \wedge |t-t'|^\eta + |H-H'|^\beta + |K-K'|^\beta) \leq (2 + 1 \vee |t-t'|^\eta ),
\end{align*}
we deduce that
\begin{align*}
| V_t^{H,K}-V_{t'}^{H',K'} | & \leq \mathbf{C}_1 (|t-t'| \wedge |t-t'|^\eta + |H-H'|^\beta + |K-K'|^\beta)^{1-\varepsilon} \frac{(2+1\vee |t-t'|^\eta )^{1-\frac{6}{p\beta}}}{(2+1\vee |t-t'|^\eta)^{1-\varepsilon}} .
\end{align*}
Thus there exists $\mathbf{C}$ with a finite moment of order $p$ such that
\begin{align*}
| V_t^{H,K}-V_{t'}^{H',K'} | &\leq \mathbf{C} (1+|t-t'|^{\eta(\varepsilon-\frac{6}{p \beta})}) \left(|t-t'| \wedge |t-t'|^\eta + |H-H'|^\beta + |K-K'|^\beta \right)^{1-\varepsilon}\\
&\leq \mathbf{C} (1+|t-t'|^{\eta(1-\frac{6}{p \beta})}) \left(|t-t'| \wedge 1 + |H-H'|^\beta + |K-K'|^\beta \right)^{1-\varepsilon}.
\end{align*}
Now choosing $\eta = \frac{2}{p}$ (that satisfies the constraint $\eta\in(\frac{1}{p},\beta)$), we get $\eta(1-\frac{6}{p\beta}) \leq \frac{2}{p}< \frac{\beta\varepsilon}{3}<1-\beta(1-\varepsilon)$. Hence $(1+|t-t'|^{\eta(1-\frac{6}{p \beta})}) \leq C (1+|t-t'|^{1-\beta(1-\varepsilon)})$ and we get 
\begin{align*}
| V_t^{H,K}-V_{t'}^{H',K'} | &\leq \mathbf{C} (1+|t-t'|^{1-\beta(1-\varepsilon)}) \left(|t-t'| \wedge 1 + |H-H'|^\beta + |K-K'|^\beta \right)^{1-\varepsilon},
\end{align*}
and the result is obtained by replacing $\beta(1-\varepsilon)$ with $\beta$.
\end{proof}

Using the previous proposition, we deduce the main result of this section.

\begin{theorem}\label{th:ergodic-OU}
Let $\mathcal{H}$ be a compact subset of $(0,1)$. For each $H \in \mathcal{H}$, let $Y^{H}$ be the solution of \eqref{eq:SDE} with the drift $b$ satisfying \eqref{eq:drift-coerciv}. Let $ \beta \in  (0,1)$ and $p \ge 1$. There exists a random variable $\mathbf{C}$ with a finite moment of order $p$ such that almost surely, for all $t\ge 0$ and $H,H' \in \mathcal{H}$,
\begin{align*}%
\frac{1}{t+1} \int_0^{t+1} | Y_s^{H}-Y_s^{H'}|^2 \diff s  \leq \mathbf{C}\,   |H-H'|^{\beta} .
\end{align*}
\end{theorem}

\begin{proof}
Integrating Inequality \eqref{eq:comparison-o-u} over $[0,t+1]$, we obtain
\begin{align*}
\frac{1}{t+1} &\int_0^{t+1} | Y_s^{H}-Y_s^{H'} |^{2} \diff s \\
&\leq \frac{C}{t+1}  \Big(|\overline{U}_0^H - \overline{U}_0^{H'} |^2 \int_{0}^{t+1} e^{-2s}\diff s + \int_0^{t+1} \big\{|\overline{U}_s^H - \overline{U}_s^{H'} |^2 + \int_{0}^s e^{-\kappa(s-r)} | \overline{U}_r^H - \overline{U}_r^{H'}|^2 \diff r\big\} \diff s \Big)\\
& \leq \frac{C}{t+1}  \Big(|\overline{U}_0^H - \overline{U}_0^{H'} |^2 + \int_0^{t+1} | \overline{U}_s^H - \overline{U}_s^{H'} |^2 \diff s \Big) .
\end{align*}
The regularity of the first term in the previous inequality is given by Proposition \ref{thm:regularity-OU} for $t=t'=0$, and the regularity of the second one is given by Proposition \ref{prop:ergodic-OU} with $H'=K=K'$ and $t'=t$.
\end{proof}

Let $\mathcal{W}$ denote the $2$-Wasserstein distance on probability measures of $\R^d$, defined  for two probability measures $\mu$ and $\nu$ by
\begin{align*}
\mathcal{W}(\mu, \nu) :=\inf \left\{\left(\mathbb{E}|X-Y|^2\right)^{\frac{1}{2}} ; \mathcal{L}(X)=\mu, \mathcal{L}(Y)=\nu\right\} .
\end{align*}
As a corollary of Theorem \ref{th:ergodic-OU}, we deduce the regularity w.r.t $H$ of the invariant measure $\mu_H$ associated to the SDE \eqref{eq:SDE}, in the $2$-Wasserstein distance.
\begin{corollary}\label{cor:reg-invariant}
Let $\mathcal{H}$ be a compact subset of $(0,1)$. For each $H \in \mathcal{H}$, let $\mu_{H}$ be the solution of \eqref{eq:SDE} with the drift $b$ satisfying \eqref{eq:drift-coerciv}. Let $ \beta \in  (0,1)$. There exists a constant $C$ with such that for any $H,H' \in \mathcal{H}$,
\begin{align}\label{eq:wasserstein}
\mathcal{W}(\mu_H, \mu_{H'} ) \leq C \,   |H-H'|^{\beta/2} .
\end{align}
\end{corollary}
\begin{proof}
Let $H$ and $H'$ in $\mathcal{H}$, then by the triangle inequality, for $t \ge 1$, we have
\begin{align*}
\mathcal{W}(\mu_H, \mu_{H'})  & \leq \mathcal{W} \left(\mu_H, \frac{1}{t+1} \int_0^{t+1} \delta_{Y_s^{H}} ds \right) + \mathcal{W}\left(\mu_{H'}, \frac{1}{t+1} \int_0^{t+1} \delta_{Y_s^{H'}} ds \right) \\ 
& \quad + \mathcal{W}\left( \frac{1}{t+1} \int_0^{t+1} \delta_{Y_s^{H}} ds, \frac{1}{t+1} \int_0^{t+1} \delta_{Y_s^{H'}} ds\right). 
\end{align*}
The integral $\int_0^t \delta_{Y_s} d s$ is to be understood as the random probability measure which associates to each Borel set $A$ the value $\int_0^t \delta_{Y_s}(A) ds$. By definition of the $2$-Wasserstein distance, it follows that
\begin{align*}
\mathcal{W}(\mu_H, \mu_{H'})  & \leq \mathcal{W}\left(\mu_H, \frac{1}{t+1} \int_0^{t+1} \delta_{Y_s^{H}} ds\right) + \mathcal{W}\left(\mu_{H'}, \frac{1}{t+1} \int_0^{t+1} \delta_{Y_s^{H'}} ds\right) \\ 
& \quad + \left( \frac{1}{t+1} \int_0^{t+1} |Y_s^{H} -Y_s^{H'} |^2 ds\right)^{1/2} . 
\end{align*} 
Using Theorem \ref{th:ergodic-OU}, we conclude that for any $\beta \in (0,1)$, there exists an integrable random variable $\mathbf{C}$ such that for any $t$,
\begin{align}\label{eq:wasserstein0}
\begin{split}
\mathcal{W}(\mu_H, \mu_{H'})  & \leq \mathcal{W}\left(\mu_H, \frac{1}{t+1} \int_0^{t+1} \delta_{Y_s^{H}} \right) + \mathcal{W}\left(\mu_{H'}, \frac{1}{t+1} \int_0^{t+1} \delta_{Y_s^{H'}} \right) \\ 
& \quad +\mathbf{C} | H-H'|^{\beta/2} . 
\end{split}
\end{align}
Moreover, by \cite[Proposition 3.3 $(i)$]{panloup2020general}, we have that almost surely
\begin{align*}
\lim_{t \rightarrow \infty }\mathcal{W}\left(\mu_H, \frac{1}{t+1} \int_0^{t+1} \delta_{Y_s^{H}}\right) = 0 ~\text{ and }~
\lim_{t \rightarrow \infty} \mathcal{W}\left(\mu_{H'}, \frac{1}{t+1} \int_0^{t+1} \delta_{Y_s^{H'}} \right) = 0 .
\end{align*}
Therefore, taking the limit as $t \rightarrow \infty$ in \eqref{eq:wasserstein0} first then taking the expectation, we deduce \eqref{eq:wasserstein}. 
\end{proof}

\section{Regularity of discrete-time fractional SDEs}\label{sec:application-discrete}

In the case of fractional noises, the invariant measure is rarely explicit and very few of its properties are known. One way to approximate it is by considering the Euler scheme associated to Equation~\eqref{eq:SDE}. This has practical applications, as mentioned in the paragraph \textbf{Statistical applications} of the Introduction. Indeed, proving the regularity in $H$ of the discrete ergodic means is useful in the study of ergodic estimators of the Hurst parameter. First we show that discrete ergodic means can be compared to the ergodic means of OU processes. Then we use the result of Section \ref{sec:application-ergodic} to conclude.

Let $\gamma \in (0,1)$ and consider the $\R^d$-valued discrete-time Stochastic Differential Equation:
\begin{align}\label{eq:SDE-discrete}
\forall t \ge 0, \ M^{H}_t & = M_0 + \int_0^{t} b(M^{H}_{s_\gamma}) \diff s + B_t^H,
\end{align}
where $b$ is a contracting drift which satisfies \eqref{eq:drift-coerciv} and where $s_\gamma := \gamma \lfloor \frac{s}{\gamma} \rfloor$ is the leftmost approximation of $s$ in the discretisation $\{k \gamma, k \in \mathbb{N}\}$. In this section, we present a result similar to Theorem \ref{th:ergodic-OU} for the process $M^H$. To that end, we first compare the process $M^H$ with the discrete-time Ornstein-Uhlenbeck process defined as
\begin{align*}%
\forall t \ge 0, \ M^{0,H}_t & = M_0 - \int_0^{t} M^{0,H}_{s_\gamma} \diff s + B_t^H .
\end{align*}
For $H,K$ in $\mathcal{H}$, we define the following processes
\begin{align*}
\mathcal{M}_t = M^H_t - M^{K}_t , \ \mathcal{M}^{0}_t =  M^{0,H}_t - M^{0,K}_t ~~\mbox{and recall}~~ \mathcal{U} = U^{H}-U^{K}, \  \overline{\mathcal{U}} = \overline{U}^{H}-\overline{U}^{K}.
\end{align*}
For the rest of this section, $C$ will denote a constant that can only depend on $p$, $\kappa$ and $K$.

\paragraph{Comparison with the discrete OU process.} First, notice that for $k  \ge 1$,
\begin{align*}
| \mathcal{M}_{k \gamma} - \mathcal{M}^0_{k \gamma} |^2 & =  | \mathcal{M}_{(k-1) \gamma} - \mathcal{M}^0_{(k-1) \gamma} |^2 + \gamma^2 |\mathcal{M}^0_{(k-1)\gamma} + b(M^H_{(k-1)\gamma})-b(M^K_{(k-1)\gamma})|^2\\ 
& \quad + 2 \gamma \langle \mathcal{M}_{(k-1)\gamma}-\mathcal{M}^0_{(k-1)\gamma}, \mathcal{M}^0_{(k-1)\gamma}+b(M^H_{(k-1)\gamma})-b(M^K_{(k-1)\gamma}) \rangle \\
& \leq |\mathcal{M}_{(k-1) \gamma} - \mathcal{M}^0_{(k-1) \gamma} |^2  + 2 \gamma^2 |\mathcal{M}^0_{(k-1)\gamma} |^2 + 2 \gamma^2 K^2 |\mathcal{M}_{(k-1)\gamma} |^2 \\
&\quad + 2 \gamma \langle \mathcal{M}_{(k-1)\gamma}-\mathcal{M}^0_{(k-1)\gamma}, \mathcal{M}^0_{(k-1)\gamma}+b(M^H_{(k-1)\gamma})-b(M^K_{(k-1)\gamma}) \rangle,
\end{align*}
where we used \eqref{eq:drift-coerciv} and Young's inequality. In order to treat the last term, we rewrite it as
\begin{align*}
&2\gamma\langle \mathcal{M}_{(k-1)\gamma}, b(M^H_{(k-1)\gamma})-b(M^K_{(k-1)\gamma})\rangle  + 2\gamma\langle \mathcal{M}_{(k-1)\gamma} , \mathcal{M}^0_{(k-1)\gamma}\rangle \\
&\quad - 2\gamma \langle \mathcal{M}^0_{(k-1)\gamma},  b(M^H_{(k-1)\gamma})-b(M^K_{(k-1)\gamma})\rangle -2\gamma | \mathcal{M}^0_{(k-1)\gamma} |^2 .
\end{align*}
Now we invoke \eqref{eq:drift-coerciv} and Young's inequality again to bound the previous quantity by
\begin{align*}
&2\gamma\left(- \kappa | \mathcal{M}_{(k-1) \gamma}  |^2 + \frac{\varepsilon}{2} | \mathcal{M}_{(k-1) \gamma} |^2 + \frac{1}{\varepsilon} | \mathcal{M}^0_{(k-1) \gamma} |^2 + K^2 \frac{\varepsilon}{2} | \mathcal{M}_{(k-1) \gamma} |^2\right) ,
\end{align*}
for some arbitrary $\varepsilon > 0$. Thus
\begin{align*}
| \mathcal{M}_{k \gamma} - \mathcal{M}^0_{k \gamma} |^2 & \leq |\mathcal{M}_{(k-1) \gamma} - \mathcal{M}^0_{(k-1) \gamma} |^2 +2\left(-\left( \kappa-\frac{\varepsilon(K^2+1)}{2}\right)\gamma + K^2 \gamma^2 \right) | \mathcal{M}_{(k-1) \gamma} |^2 \\ 
& \quad + (2+2\gamma^2 + \frac{2 \gamma}{\varepsilon}) | \mathcal{M}^0_{(k-1) \gamma} |^2 .
\end{align*}
We take $\varepsilon=\frac{\kappa}{(1+K^2)}$ to get 
\begin{align*}
| \mathcal{M}_{k \gamma} - \mathcal{M}^0_{k \gamma} |^2 & \leq  |\mathcal{M}_{(k-1) \gamma} - \mathcal{M}^0_{(k-1) \gamma} |^2+(-\gamma \kappa + 2 K^2 \gamma^2 ) | \mathcal{M}_{(k-1) \gamma}-\mathcal{M}^0_{(k-1) \gamma}  |^2 + C\,  | \mathcal{M}^0_{(k-1) \gamma} |^2 .
\end{align*}
Fix $\gamma_0 \in (0,1)$ such that if $\gamma < \gamma_0$, we have $0< \gamma\kappa -  2K^2 \gamma^2 <1$. There is
\begin{align*}
| \mathcal{M}_{k \gamma} - \mathcal{M}^0_{k \gamma} |^2 & \leq  \left( (1-\gamma \kappa + K^2 \gamma^2 ) | \mathcal{M}_{(k-1) \gamma} - \mathcal{M}^0_{(k-1) \gamma} |^2 + C  | \mathcal{M}^0_{(k-1) \gamma} |^2 \right) .
\end{align*}
By a direct induction, it comes that
\begin{align*}
| \mathcal{M}_{k \gamma} - \mathcal{M}^0_{k \gamma} |^2 & \leq C  \sum_{j=0}^{k-1} (1-\gamma \kappa + 2K^2 \gamma^2 )^{k-1-j}  | \mathcal{M}^0_{j \gamma} |^2 .
\end{align*}
It follows that for all $N \in \mathbb{N}^*$,
\begin{align*}
\frac{1}{N} \sum_{k=1}^N | \mathcal{M}_{k \gamma} |^2 \leq C \frac{1}{N} \sum_{k=1}^N \sum_{j=0}^{k-1} (1-\gamma \kappa + 2K^2 \gamma^2 )^{k-1-j}  | \mathcal{M}^0_{j \gamma} |^2 + \frac{1}{N} \sum_{k=1}^N |  \mathcal{M}^0_{k \gamma} |^2 .
\end{align*}
Summing over $k$ first and recalling that $\mathcal{M}^0_{0} =   \mathcal{M}_{0}$, we get
\begin{align} 
\frac{1}{N} \sum_{k=0}^N | \mathcal{M}_{k \gamma} |^2 & \leq C \frac{1}{N} \sum_{j=0}^N\frac{1- (1-\gamma \kappa+ 2 K^2 \gamma^2 )^{N-1-j}}{\gamma \kappa - 2 K^2 \gamma^2} | \mathcal{M}^0_{j \gamma} |^2 + \frac{1}{N} \sum_{k=0}^N |  \mathcal{M}^0_{k \gamma} |^2 \nonumber  \\
& \leq C \frac{1}{N} \sum_{k=0}^N |  \mathcal{M}^0_{k \gamma} |^2 . \label{eq:comp-M-discrete-ou}
\end{align}

\paragraph{Comparison of the discrete OU process and the OU process.}
Let us compare the two processes assuming they start at the same point and are driven by the same noises. 
First observe that
\begin{align*}
\mathcal{M}_{k \gamma}^0 - \mathcal{U}_{k \gamma} &= (1-\gamma)\mathcal{M}_{(k-1) \gamma}^0 - \left(\mathcal{U}_{(k-1) \gamma} - \int_{(k-1)\gamma}^{k \gamma} \mathcal{U}_{s }  \diff s \right)\\
&= (1-\gamma) (\mathcal{M}_{(k-1) \gamma}^0 - \mathcal{U}_{(k-1) \gamma} ) + \int_{(k-1)\gamma}^{k \gamma} ( \mathcal{U}_{s } - \mathcal{U}_{(k-1) \gamma}) \diff s .
\end{align*}
Using Young's and Jensen's inequalities we get
\begin{align*}
\left| \mathcal{M}_{k \gamma}^0 - \mathcal{U}_{k \gamma} \right|^2 & = \Big| (1-\gamma) ( \mathcal{M}_{(k-1) \gamma}^0 - \mathcal{U}_{(k-1) \gamma}  )  + \int_{(k-1)\gamma}^{k \gamma} ( \mathcal{U}_{s } - \mathcal{U}_{(k-1) \gamma}) \diff s \Big|^2 \\
& \leq 2 (1-\gamma)^2 \Big| \mathcal{M}_{(k-1) \gamma}^0 - \mathcal{U}_{(k-1) \gamma}\Big|^2  +4 \gamma \int_{(k-1)\gamma}^{k \gamma} | \mathcal{U}_{s }|^2  \diff s +  4 \gamma^2 |\mathcal{U}_{(k-1) \gamma} |^2  .
\end{align*}
It follows by induction that
\begin{align*}
\left|  \mathcal{M}_{k \gamma}^0 -\mathcal{U}_{k \gamma}\right|^2 & \leq  C \sum_{j=0}^{k-1} (1-\gamma)^{2(k-1-j)} \Big( \gamma \int_{j\gamma}^{(j+1) \gamma} | \mathcal{U}_{s}|^2  \diff s +  \gamma^2 |\mathcal{U}_{j \gamma}|^2 \Big) .
\end{align*}
Summing over $k$ between $1$ and $N \ge 1$, we have
\begin{align*}
\sum_{k=1}^N \left| \mathcal{M}_{k \gamma}^0- \mathcal{U}_{k \gamma}\right|^2 & \leq C  \sum_{k=1}^N\sum_{j=0}^{k-1} (1-\gamma)^{2(k-1-j)} \Big( \gamma \int_{j\gamma}^{(j+1) \gamma} | \mathcal{U}_{s}|^2  \diff s +  \gamma^2 |\mathcal{U}_{j \gamma}|^2 \Big) .
\end{align*}
Summing over $k$ first on the right-hand side, we get
\begin{align*}
\sum_{k=1}^N \left| \mathcal{M}_{k \gamma}^0- \mathcal{U}_{k \gamma}\right|^2 
& \leq C \sum_{j=0}^{N-1}  \frac{1}{\gamma}  \Big( \gamma \int_{j\gamma}^{(j+1) \gamma} | \mathcal{U}_{s }|^2  \diff s +  \gamma^2 |\mathcal{U}_{j \gamma}|^2 \Big)  
\leq  C \left(  \int_{0}^{N \gamma} | \mathcal{U}_{s}|^2  \diff s + \gamma \sum_{j=0}^{N-1}  | \mathcal{U}_{j \gamma}|^2 \right) .
\end{align*}
It follows that
\begin{align}\label{eq:comp-ou-discrete-ou}
\frac{1}{N} \sum_{k=1}^N \left|  \mathcal{M}_{k \gamma}^0- \mathcal{U}_{k \gamma}\right|^2 
& \leq  C\, \gamma \left (\frac{1}{N} \sum_{j=0}^{N-1}  |\mathcal{U}_{j \gamma}|^2 + \frac{1}{N \gamma} \int_{0}^{N \gamma} | \mathcal{U}_{s}|^2  \diff s \right) .
\end{align}
Recall that for all $t \ge 0$, we have $\mathcal{U}_{t} - \overline{\mathcal{U}_{t}} = e^{-t} \overline{\mathcal{U}_{0}}$. Therefore
\begin{align}\label{eq:dsicrete-ou-sou}
\frac{1}{N} \sum_{j=0}^{N-1}  | \mathcal{U}_{j \gamma} |^2  \leq C \left( \frac{1}{N} \sum_{j=0}^{N-1}  |\overline{\mathcal{U}_{j \gamma}}|^2  + | \overline{\mathcal{U}_{0}}|^2  \right) .
\end{align}
Finally, combining \eqref{eq:dsicrete-ou-sou} with \eqref{eq:comp-ou-discrete-ou} and \eqref{eq:comp-M-discrete-ou}, we conclude that
\begin{align}\label{eq:discrete-final-comp}
\frac{1}{N} \sum_{k=0}^N | M_{k \gamma}^H-M_{k \gamma}^{K} |^2 \leq C\, \Big( \frac{1}{N} \sum_{j=0}^{N-1}  |\overline{U}^H_{j\gamma} - \overline{U}^{K}_{j \gamma}|^2 + | \overline{U}_0^H - \overline{U}_0^{K} |^2 + \frac{1}{N \gamma} \int_{0}^{N \gamma} | U_s^H - U_s^{K}|^2  \diff s \Big) .
\end{align}
The regularity of the second term in the right-hand side is given by Proposition \ref{thm:regularity-OU} and the regularity of the third term is given by Theorem~\ref{th:ergodic-OU}. To bound the first term, we need a discrete version of Theorem~\ref{th:ergodic-OU}. We introduce the process
\begin{align}\label{eq:V-discret}
\mathcal{V}^{H,K}_t = \frac{1}{t+1} \int_0^{t+1} |\overline{U}^H_{s_\gamma} -\overline{U}^{K}_{s_\gamma}  |^2 \diff s - \mathbb{E}|\overline{U}_0^H-\overline{U}_0^{K}|^2 ,\quad  t \ge 0, \, H,K \in \mathcal{H}.
\end{align}
This process is continuous in time, $H$ and $H'$, so the idea is to use again a GRR argument. In Lemma \ref{lem:ergodic-U2-dsicrete}, we prove upper bounds on the increments of $\mathcal{V}_t^{H,K}$ that are similar to those proved in Lemma \ref{lem:ergodic-U2} (for its continuous-time version $V_t^{H,K}$). With those bounds in mind, the proof of the following Proposition is the same as that of Proposition \ref{prop:ergodic-OU}. 

Recall that $\gamma_{0} =\sup\{\gamma\in (0,1):~ \forall \xi \in (0,\gamma),\  0<\kappa \xi - 2 K^2 \xi^2 <1\}$.

\begin{prop}\label{thm:ergodic-OU-discret}
Let $\gamma \in (0, \gamma_0)$ and $\mathcal{H}$ be a compact subset of $(0,1)$. Let $ \beta \in  (0,1)$, and $p \ge 1$. There exists a random variable $\mathbf{C}$ with a finite moment of order $p$ such that almost surely, for all $t,t' \ge 0$ and all $H,H',K,K' \in \mathcal{H}$,
\begin{align*}%
| \mathcal{V}^{H,K}_t  -  \mathcal{V}^{H',K'}_{t'}|  \leq \mathbf{C} (1+|t'-t|^{1-\beta}) \left( 1 \wedge |t-t'|  + |H-H'| + |K-K'| \right)^\beta.
\end{align*}
\end{prop}
In view of \eqref{eq:discrete-final-comp} and using Proposition \ref{thm:regularity-OU} and Proposition \ref{thm:ergodic-OU-discret} with $t=t'$ and $H'=K=K'$, we deduce the following Theorem.
\begin{theorem}\label{cor:ergodic-OU-discret}
 Let $\gamma \in (0, \gamma_0)$ and $\mathcal{H}$ be a compact subset of $(0,1)$. Let $ \beta \in  (0,1) $ and $p \ge 1$.  Recall that $M^H$ is defined in \eqref{eq:SDE-discrete}. There exists a random variable $\mathbf{C}$ with a finite moment of order $p$ such that almost surely, for all $N\in \N^*$ and $H,H' \in \mathcal{H}$,
\begin{align*}%
\frac{1}{N} \sum_{k=1}^{N} | M_{k \gamma}^{H}-M_{k \gamma}^{H'}|^2  \leq \mathbf{C}   |H-H'|^{\beta} .
\end{align*}
\end{theorem}
This result states that the ergodic means of the Euler scheme of \eqref{eq:SDE} are almost Lipschitz in the Hurst parameter. It is also a key element in the proof of \cite[Theorem 2.8]{HRstat}, which gives a rate of convergence of the estimator \eqref{eq:estimatorH}.

\appendix

\section{Long-time properties of the fractional Ornstein-Uhlenbeck process}\label{sec:app}
This section gathers results on the fractional Ornstein-Uhlenbeck process defined in \eqref{eq:o-u-def}, namely upper bounds on moments of their increments and on moments of their ergodic means.

\begin{lemma}\label{lem:Ubar-increments}
Recall that $\overline{U}^H_{t} = \int_{-\infty}^t e^{-(t-r)} \diff \fBm_r^H$. 
Let $p\ge 1$ and $\mathcal{H}$ be a compact subset of $(0,1)$. There exists a constant $C$ such that for all $H,H'$ in $\mathcal{H}$ and $t', t \ge 0$,
\begin{align*}
\mathbb{E} \big| \overline{U}_t^{H}-\overline{U}_{t'}^{H'} \big|^p \leq C \left( 1 \wedge |t-t'|^{\min(\mathcal{H})} + | H-H' | \right)^p .
\end{align*}
\end{lemma}

\begin{proof}
Introducing the pivot term $\overline{U}_{t}^{H'}$, we have
\begin{align}\label{eq:UbarA1A2}
\mathbb{E} \big| \overline{U}_t^{H}-\overline{U}_{t'}^{H'} \big|^p & \leq C \left(  \mathbb{E} | \overline{U}_t^{H'}-\overline{U}_{t'}^{H'} |^p + \mathbb{E} | \overline{U}_t^{H}-\overline{U}_{t}^{H'} |^p \right) \nonumber\\
& =: C \left( A_1 + A_2 \right).
\end{align}

Let us start by analyzing $A_1$. Without loss of generality, assume that $t'\ge t$. Since the process $\overline{U}^{H'}$ is solution to \eqref{eq:o-u-def} and is stationary, we use Jensen's inequality to get
\begin{align*}
A_1 & = \mathbb{E} \Big| \int_t^{t'} \overline{U}_s^{H'} \diff s + B_{t'}^{H'}-B_{t}^{H'} \Big|^p\\
& \leq C\, |t'-t|^{p-1}   \int_t^{t'} \mathbb{E} |\overline{U}_0^{H'}|^p\, \diff s +C\, \mathbb{E}|B_t^{H'}-B_{t'}^{H'}|^p  \\
& \leq C \left( (t'-t)^p + (t'-t)^{p H'} \right)  .
\end{align*}
So we get $A_{1}\leq C\, |t'-t|^{p H'}$ when $|t'-t| \leq 1$. When $|t'-t| \ge 1$, we use the stationarity to get
\begin{align*}
A_1 & \leq C \, \mathbb{E}|\overline{U}_t^{H'}|^p + C\, \mathbb{E}|\overline{U}_t^{H'}|^p   \leq C .
\end{align*}
Overall we obtain
\begin{align}\label{eq:UbarA1}
A_1 \leq C (1 \wedge|t'-t|)^{p \min(\mathcal{H})}.
\end{align}

We now consider $A_2$. Recall that
\begin{align*}
\overline{U}_t^H - \overline{U}_t^{H'} =\int_{-\infty}^t e^{-(t-u)} \diff (\fBm_u^H - \fBm_{u}^{H'}) .
\end{align*}
Hence by integration-by-parts, in view of Remark \ref{rk:regB-}, we get
\begin{align*}
\overline{U}_t^H - \overline{U}_t^{H'} & = \fBm_t^H - \fBm_{t}^{H'} - \int_{-\infty}^t e^{-(t-u)} (\fBm_u^H - \fBm_{u}^{H'}) \diff u \\
& = \int_{-\infty}^t e^{-(t-u)} \left( B_t^H -B_u^H - B_t^{H'} + B_u^{H'} \right)\diff u .
\end{align*}
Therefore, using the previous equality and the Gaussian property of $\overline{U}_t^H - \overline{U}_t^{H'}$, we have
\begin{align*}
\mathbb{E} |\overline{U}_t^H - \overline{U}_t^{H'}|^p & \leq C \left(\mathbb{E} |\overline{U}_t^H - \overline{U}_t^{H'}| \right)^p \\
& \leq  C\left( \int_{-\infty}^t e^{-(t-u)} \sqrt{\mathbb{E}( B_t^H -B_u^H - B_t^{H'} + B_u^{H'} )^2} \diff u \right)^p\\
& = C \left( \int_{-\infty}^t e^{-(t-u)} \sqrt{\mathbb{E}( B_{t-u}^H - B_{t-u}^{H'} )^2} \diff u \right)^p,
\end{align*}
where the last equality comes from the joint increment stationarity of $(B^H_{t}, B^{H'}_{t})_{t\geq0}$, namely that for any $s \in \R$ and any $H,H'\in \mathcal{H}$, the laws of the following two-dimensional processes coincide:
\begin{align}\label{eq:laws}
\left( \fBm_{t+s}^H - \fBm_{t}^H, \fBm_{t+s}^{H'} - \fBm_s^{H'} \right)_{t\in \R} \overset{(d)}{=} \left( \fBm_t^H, \fBm_t^{H'} \right)_{t \in \R} .
\end{align}
This property will be proven in Appendix~\ref{app:proof-ub-justfBm-rectangular}.
We can now use Proposition \ref{prop:ub-justfBm} to conclude that
\begin{align}\label{eq:UbarA2}
A_2 \leq C |H-H'|^p.
\end{align}
Hence, using \eqref{eq:UbarA1} and \eqref{eq:UbarA2} in \eqref{eq:UbarA1A2} gives the result.
\end{proof}

\begin{lemma}\label{lem:Ubb-increments} 
Recall that the process $\overline{\mathbb{U}}$ is defined in \eqref{eq:Ubb-def}. Let $p \ge 1$ and $\mathcal{H}$ be a compact subset of $(0,1)$. There exists a constant $C$ such that for all $H,H' \in \mathcal{H}$ and $t'\geq t \ge 0$,
\begin{align*}
\mathbb{E} | \overline{\mathbb{U}}_t^H-\overline{\mathbb{U}}_{t'}^{H'} |^p \leq C  (1+t)^{-p\alpha} \left( 1 \wedge |t-t'|^{\min(\mathcal{H})} + |H-H'| \right)^p.
\end{align*}
\end{lemma}
\begin{proof}
Introducing the pivot terms $(1+t)^{-\alpha} U_{t}^{H'}$ and $(1+t')^{-\alpha} U_{t}^{H'}$, we have that
\begin{align*}
\mathbb{E} | \overline{\mathbb{U}}_t^H-\overline{\mathbb{U}}_{t'}^{H'} |^p & \leq C (1+t)^{-p \alpha} \left( \mathbb{E} |\overline{U}_t^H - \overline{U}_t^{H'}|^p + \mathbb{E}  |\overline{U}_t^{H'} - \overline{U}_{t'}^{H'}|^p \right) \\
& + \left( (1+t)^{-\alpha}-(1+t')^{-\alpha} \right)^{p} \mathbb{E} |\overline{U}_t^{H'}|^p.
\end{align*}
Using the inequality
\begin{align*}
\left((1+t)^{-\alpha}-(1+t')^{-\alpha} \right)^{p} \leq (1+t)^{-p\alpha} \left(1\wedge (t'-t)^p\right),
\end{align*}
the stationarity of $\overline{U}^{H'}$ and Lemma \ref{lem:Ubar-increments}, we get the desired result.
\end{proof}

\begin{lemma}\label{lem:cov-BH-BK}
Let $\mathcal{H}$ be a compact subset of $(0,1)$ and $B$ be a $1$-dimensional fBm. Then there exists a constant $C$ such that for any $H,K \in \mathcal{H}$ and any $-\infty \leq a \leq 0 \leq c \leq d < \infty$,
\begin{align*}
\Bigg| \mathbb{E} \int_a^0 e^u \diff B^H_u \int_c^d e^v \diff B^K_v + \mathbb{E} \int_a^0 e^u \diff B^K_u \int_c^d e^v \diff B^H_v  \Bigg| \leq C \int_a^0 e^u \Big( \int_c^d e^v (v-u)^{H+K-2} \diff v \Big) \diff u .
\end{align*}
\end{lemma}
\begin{proof}
Assume first that $c=0$. By integration-by-parts, we get
\begin{align*}
\mathbb{E} \int_a^0 e^u \diff B^H_u \int_0^d e^v \diff B^K_v & = \mathbb{E} \left[ \Big( -e^a B^H_a -\int_a^0 e^u B^H_u \diff u \Big) \Big( e^d B^K_d- \int_0^d e^v B^K_v \diff v\Big) \right] \\
& = - e^{a+d}\, \mathbb{E}B_a^H B_d^K + e^{a} \int_0^d e^{v}\, \mathbb{E} B_a^H B_v^K \diff v - e^{d} \int_a^0 e^{u} \, \mathbb{E} B_u^H B_d^K \diff u \\ 
&\quad  + \int_a^0 e^{u} \int_0^d e^{v}\, \mathbb{E} B_u^H B_v^K \diff v \diff u .
\end{align*}
Integrating-by-parts with respect to $u$ yields
\begin{align*}
\mathbb{E} \int_a^0 e^u\, \diff B^H_u \int_0^d e^v \diff B^K_v = e^d \int_a^0 e^{u} \partial_u \mathbb{E} B_u^H B_d^K \diff u  - \int_a^0 e^u \int_0^d e^v\, \partial_u \mathbb{E} B_u^H B_v^K \diff v \diff u .
\end{align*}
After another integration-by-parts with respect to $v$, it now comes
\begin{align*}
\mathbb{E} \int_a^0 e^u\, \diff B^H_u \int_0^d e^v \diff B^K_v =  \int_a^0 e^u \int_0^d e^v\, \partial^2_{uv} \mathbb{E} B_u^H B_v^K \diff v \diff u .
\end{align*}
Hence one gets that
\begin{align}\label{eq:crossExp}
 \mathbb{E}\Big[ \int_a^0 e^u \diff B^H_u \int_0^d e^v \diff B^K_v + \int_a^0 e^u \diff B^K_u \int_0^d e^v \diff B^H_v  \Big] =\int_a^0 \int_0^d e^{u+v} \partial^2_{uv} \mathbb{E} \left( B_u^H B_v^K + B_u^K B_v^H \right) \diff v \diff u .
\end{align}
Now observe that
\begin{align*}
\mathbb{E} \left( B_u^H B_v^K + B_u^K B_v^H \right)  =- \mathbb{E} [(B_v^H-B_u^H) (B_v^K-B_u^K)] + \mathbb{E} B_u^H B_u^K  + \mathbb{E} B_v^H B_v^K .
\end{align*}
In view of the joint increment stationarity of $B$ (see \eqref{eq:laws}) and its integral representation \eqref{eq:fBm}, there is
\begin{align*}
\mathbb{E} &[ (B_v^H-B_u^H) (B_v^K-B_u^K)] \\
 &= \mathbb{E} B^H_{v-u} B^K_{v-u} \\
&= \frac{1}{\Gamma(H+\frac{1}{2}) \Gamma(K+\frac{1}{2})} \int_{\R} \left((v-u-s)_{+}^{H-\frac{1}{2}} - (-s)_{+}^{H-\frac{1}{2}} \right) \left((v-u-s)_{+}^{K-\frac{1}{2}} - (-s)_{+}^{K-\frac{1}{2}} \right) \diff s\\
&= (v-u)^{H+K}\, \mathbb{E} B^H_1 B^K_1,
\end{align*}
using the change of variables $\tilde{s}=\frac{s}{v-u}$ in the last equality. Since $\partial^2_{uv} \mathbb{E} B_u^H B_u^K = 0$, it follows that
\begin{align}\label{eq:boundDiffExp}
|\partial^2_{uv} \mathbb{E} \left( B_u^H B_v^K + B_u^K B_v^H \right) | \leq C (v-u)^{H+K-2} .
\end{align}
Plugging the previous inequality in \eqref{eq:crossExp} gives the result for $c=0$. 

For $c > 0$, we write
\begin{align*}
\mathbb{E} \int_a^0 e^u \diff B^H_u \int_c^d e^v \diff B^K_v  & = \mathbb{E} \int_a^0 e^u \diff B^H_u \int_0^d e^v \diff B^K_v  - \mathbb{E} \int_a^0 e^u \diff B^H_u \int_0^c e^v \diff B^K_v .
\end{align*}
And we use \eqref{eq:crossExp} to get
\begin{align*}
 \mathbb{E}&\Big[ \int_a^0 e^u \diff B^H_u \int_c^d e^v \diff B^K_v + \int_a^0 e^u \diff B^K_u \int_c^d e^v \diff B^H_v  \Big] \\
& = \int_a^0 e^u \int_0^d e^v\, \partial^2_{uv} \mathbb{E} \left(B_u^H B_v^K + B^K_{u} B^H_{v} \right) \diff v \diff u - \int_a^0 e^u \int_0^c e^v\, \partial^2_{uv} \mathbb{E} \left(B_u^H B_v^K+ B^K_{u} B^H_{v} \right) \diff v \diff u  \\
& =   \int_a^0 e^u \int_c^d e^v\, \partial^2_{uv} \mathbb{E} \left(B_u^H B_v^K + B^K_{u} B^H_{v}\right) \diff v \diff u .
\end{align*}
Using again \eqref{eq:boundDiffExp} in the previous equality yields the result.
\end{proof}

The following lemma is a generalisation of Theorem 2.3 in \cite{cheridito2003fractional}, in the sense that it allows for different Hurst parameters. 
\begin{lemma}\label{lem:asymptotic-cov}
Let $\mathcal{H}$ be a compact subset of $(0,1)$. 
For any $H,K\in (0,1)$, the process $(\overline{U}^H,\overline{U}^K)$ is stationary. Besides, there exists a contant $C$ such that for any $H,K \in \mathcal{H}$, any $t \ge 0$ and $s \geq0$,
\begin{align*}
\big|\mathbb{E} \big( \langle \overline{U}_{t}^H, \overline{U}_{t+s}^K\rangle +  \langle\overline{U}_{t+s}^H, \overline{U}_{t}^K\rangle \big)\big| = \big|\mathbb{E} \big( \langle \overline{U}_{0}^H, \overline{U}_{s}^K\rangle +  \langle\overline{U}_{s}^H, \overline{U}_{0}^K\rangle \big)\big|  \leq C  \left(1\wedge s^{2 \max(\mathcal{H})-2}\right) .
\end{align*}
\end{lemma}

\begin{remark}\label{rk:dimension}
Recall that the entries of the vector-valued process $B$ are independent. Hence it suffices to prove the previous result component by component. To avoid heavy notations, we proceed to the following proof in dimension $d=1$.
\end{remark}

\begin{proof}
First, we have by integration-by-parts that
\begin{align*}
\mathbb{E} \big( \overline{U}_{t}^H \overline{U}_{t+s}^K \big) & = \mathbb{E} \int_{-\infty}^t e^{-(t-u)} \diff B_u^H \int_{-\infty}^{t+s} e^{-(t+s-v)} \diff B^K_v \\ & = \mathbb{E}\int_{-\infty}^t e^{-(t-u)} (B_t^H-B_u^H) \diff u \int_{-\infty}^{t+s} e^{-(t+s-v)} (B^{K}_{t+s}-B^K_v )\diff v  .
\end{align*}
Thus in view of \eqref{eq:laws}, 
\begin{align*}
\mathbb{E} \big( \overline{U}_{t}^H \overline{U}_{t+s}^K \big) = \mathbb{E}\int_{-\infty}^t e^{-(t-u)} B_{t-u}^H \diff u \int_{-\infty}^{t+s} e^{-(t+s-v)} B^{K}_{t+s-v}\diff v .
\end{align*}
By the changes of variables $\tilde{u}=u-t$ and $\tilde{v}=v-t$, it follows that
\begin{align*}
\mathbb{E} \big( \overline{U}_{t}^H \overline{U}_{t+s}^K \big)  = \mathbb{E}\int_{-\infty}^0 e^{\tilde{u}} B_{-\tilde{u}}^H \diff \tilde{u} \int_{-\infty}^{s} e^{-(s-\tilde{v})} B^{K}_{s-\tilde{v}}\diff \tilde{v} .
\end{align*}
Hence by integration-by-parts again, we find
\begin{align*}
\mathbb{E} \big( \overline{U}_{t}^H \overline{U}_{t+s}^K \big) = \mathbb{E}\int_{-\infty}^0 e^{u} \diff B_{u}^H \int_{-\infty}^{s} e^{-(s-v)} \diff B^{K}_{v} = \mathbb{E} \overline{U}_{0}^H \overline{U}_{s}^K ,
\end{align*}
which proves the claim that $(\overline{U}^H,\overline{U}^K)$ is stationary since $\overline{U}$ is Gaussian.

We now bound $\mathbb{E} \big( \overline{U}_{0}^H \overline{U}_{s}^K +  \overline{U}_{s}^H \overline{U}_{0}^K \big)$. When $s\in[0,1]$, this quantity is bounded. We now assume that $s\geq1$ and use the following decomposition:
\begin{align*}
\mathbb{E} \big( \overline{U}_{0}^H \overline{U}_{s}^K\big) &= 
\mathbb{E} \int_{-\infty}^0 e^u \diff B_u^H \int_{-\infty}^s e^{-(s-v)} \diff B^K_v \\
& = e^{-s} \, \mathbb{E} \int_{-\infty}^0 e^u \diff B_u^H \int_{-\infty}^1 e^{v} \diff B^K_v + e^{-s} \, \mathbb{E} \int_{-\infty}^0 e^u \diff B_u^H \int_{1}^s e^{v} \diff B^K_v ,
\end{align*}
and similarly for $\mathbb{E} \big(\overline{U}_{s}^H \overline{U}_{0}^K \big) $. 
The quantities $|\mathbb{E} \big[\int_{-\infty}^0 e^u \diff B_u^H \int_{-\infty}^1 e^{v} \diff B^K_v\big]| = |\EE \overline{U}^H_{0} \overline{U}^K_{1}|$ and $|\EE \overline{U}^H_{1} \overline{U}^K_{0}|$ do not depend on $s$ and are bounded uniformly in $H,K\in \mathcal{H}$. Thus by the previous remark and Lemma \ref{lem:cov-BH-BK} applied to $a=-\infty$, $c=1$ and $d=s$, it comes
\begin{align}\label{eq:boundExpStatU}
\big|\mathbb{E} \big( \overline{U}_{0}^H \overline{U}_{s}^K +  \overline{U}_{s}^H \overline{U}_{0}^K \big)\big| \leq C\, e^{-s}   + C\, e^{-s} \int_{-\infty}^0 e^u \int_1^s e^v (v-u)^{H+K-2} \diff v \diff u .
\end{align}
Use the changes of variables $y=v-u,\, z=-u$ and Fubini's theorem to get
\begin{align*}
\int_{-\infty}^0 e^u \int_1^s e^v (v-u)^{H+K-2} \diff v \diff u &= \int_{0}^\infty e^{-2z} \int_{1+z}^{s+z} e^y\, y^{H+K-2} \diff y \diff z\\
&= \int_{1}^\infty e^y\, y^{H+K-2} \int_{(y-s)\vee 0}^{y-1} e^{-2z} \diff z \diff y\\
&= \frac{1}{2}\int_{1}^\infty e^{y}\, y^{H+K-2} \left(e^{-2((y-s)\vee 0)}-e^{-2(y-1)}\right) \diff y .
\end{align*}
We now split the previous integral in three:
\begin{align*}
& \int_{-\infty}^0  e^u \int_1^s e^v (v-u)^{H+K-2} \diff v \diff u \\ &= \frac{1}{2}\int_{1}^{s/2} e^{y}\, y^{H+K-2} \left(1-e^{-2(y-1)}\right) \diff y  + \frac{1}{2}\int_{s/2}^{s} e^{y}\, y^{H+K-2} \left(1-e^{-2(y-1)}\right) \diff y \\ & \quad + \frac{1}{2}\int_{s}^\infty e^{-y}\, y^{H+K-2} \left(e^{2s}-e^{2}\right) \diff y \\
& \leq \frac{1}{2} e^{s/2} + \frac{1}{2} \left(\frac{s}{2}\right)^{H+K-2} e^s + \frac{1}{2} s^{H+K-2} \, e^s .
\end{align*}
Using the previous inequality in \eqref{eq:boundExpStatU} and the inequality $e^{-s/2} + e^{-s} \leq C s^{2\max(\mathcal{H})-2}$ for $s\geq 1$ gives the desired result.
\end{proof}
\begin{lemma}\label{lem:bound-I1-I2}
Let  $\mathcal{H}$ be a compact subset of $(0,1)$. There exists a constant $C$ such that for any $H,K \in \mathcal{H}$, any $t \ge 0$,
\begin{align*}
\int_{[0,t+1]^2} \left( \mathbb{E} \langle \overline{U}_{s_1}^H -\overline{U}_{s_1}^K , \overline{U}_{s_2}^H-\overline{U}_{s_2}^K \rangle \right)^2 
\diff s_1 \diff s_2 \leq C  (\log(t+1)+1)\, (t+1)^{(4 \max(\mathcal{H})-2 ) \vee 1} ,
\end{align*}
and
\begin{align*}
  \int_{[0,t+1]^3}   &\left( \mathbb{E} \langle \overline{U}_{s_1}^H-\overline{U}_{s_1}^K , \overline{U}_{s_2}^H-\overline{U}_{s_2}^K \rangle \right)^2 \left( \mathbb{E} \langle \overline{U}_{s_2}^H-\overline{U}_{s_2}^K  , \overline{U}_{s_3}^H-\overline{U}_{s_3}^K \rangle \right)^2   \diff s_1 \diff s_2 \diff s_3\\ 
  & \leq C   (\log(t+1)+1)^2\,  (t+1)^{(8 \max(\mathcal{H})-5) \vee 1}  .
\end{align*}
\end{lemma}

\begin{proof}
Let us denote 
\begin{align*}
I_1 & := \int_{[0,t+1]^2} \left( \mathbb{E} \langle \overline{U}_{s_1}^H-\overline{U}_{s_1}^K , \overline{U}_{s_2}^H-\overline{U}_{s_2}^K \rangle \right)^2 \diff s_1 \diff s_2 \\
I_2 & := \int_{[0,t+1]^3}  \left( \mathbb{E} \langle \overline{U}_{s_1}^H-\overline{U}_{s_1}^K , \overline{U}_{s_2}^H-\overline{U}_{s_2}^K \rangle \right)^2 ~ \left( \mathbb{E} \langle  \overline{U}_{s_2}^H-\overline{U}_{s_2}^K , \overline{U}_{s_3}^H-\overline{U}_{s_3}^K \rangle \right)^2   \diff s_1 \diff s_2 \diff s_3 .
\end{align*}
Using Lemma \ref{lem:asymptotic-cov}, we have 
\begin{align*}
\left( \mathbb{E} \langle \overline{U}_{s_1}^H-\overline{U}_{s_1}^K , \overline{U}_{s_2}^H - \overline{U}_{s_2}^K \rangle \right)^2 & = \left( \mathbb{E} \langle \overline{U}_{s_1}^H, \overline{U}_{s_2}^H    \rangle - \mathbb{E}\langle \overline{U}_{s_1}^H , \overline{U}_{s_2}^K \rangle - \mathbb{E} \langle \overline{U}_{s_1}^K, \overline{U}_{s_2}^H \rangle + \mathbb{E} \langle \overline{U}_{s_1}^K, \overline{U}_{s_2}^K  \rangle  \right)^2 \\
& \leq C \left( 1 \wedge |s_1 - s_2 |^{4 \max(\mathcal{H})-4} \right) . 
\end{align*}
Therefore
\begin{align*}
I_1 \leq C  \int_{[0,t+1]^2} 1 \wedge |s_1-s_2|^{4 \max(\mathcal{H})-4} \diff s_1 \diff s_2 .
\end{align*}
The change of variables $x=s_1-s_2$ yields
\begin{align*}
I_1  & \leq C \int_{0}^{t+1} \int_{-s_2}^{t-s_2} 1 \wedge |x|^{4 \max(\mathcal{H})-4} \diff x \diff s_2  \nonumber \\
& \leq C\, (t+1)   \int_{-(t+1)}^{t+1} 1 \wedge |x|^{4 \max(\mathcal{H})-4} \diff x  \nonumber \\
& \leq C\, (t+1)\Big(   1+\int_{1}^{t+1}  x^{4 \max(\mathcal{H})-4} \diff x \Big) \nonumber \\
& \leq  C\, (t+1)  \left(  1+ \mathds{1}_{\max(\mathcal{H}) = \frac{3}{4}} (\log(t+1)+1)+ \mathds{1}_{\max(\mathcal{H}) \neq \frac{3}{4}}  (t+1)^{4 \max(\mathcal{H})-3} \right) \nonumber \\
& \leq C\, \Big( (t+1) (\log(t+1)+1) + (t+1)^{4 \max(\mathcal{H})-2} \Big).
\end{align*}
Using Lemma \ref{lem:asymptotic-cov} again, we have 
\begin{align*} 
I_2 \leq C  \int_{[0,t+1]^3} \left( 1 \wedge |s_1-s_2|^{4 \max(\mathcal{H})-4} \right) \left( 1 \wedge |s_2-s_3|^{4 \max(\mathcal{H})-4} \right)  \diff s_1 \diff s_2  \diff s_3 .
\end{align*}
The change of variables $x=s_1-s_2$, $y=s_2-s_3$ yields
\begin{align*} 
I_2 & \leq C  \int_{[-t-1,t+1]^3} \left( 1 \wedge |x|^{4 \max(\mathcal{H})-4} \right) \left( 1 \wedge |y|^{4 \max(\mathcal{H})-4} \right)  \diff x \diff y  \diff s_3  \\
& \leq C (t+1) \Big( \int_{[-t-1,t+1]} \left( 1 \wedge |x|^{4 \max(\mathcal{H})-4} \right)  \diff x \Big)^2 \\
& \leq C (t+1) \Big( (\log(t+1) +1) + (t+1)^{4 \max(\mathcal{H})-3} \Big)^2 \\
& \leq C \Big( (t+1) (\log(t+1) +1) ^2 + (t+1)^{8 \max(\mathcal{H})-5} \Big) .
\end{align*}
\end{proof}

The following Lemma states a Gaussian equality that allows us to bound the moments of the ergodic means.

\begin{lemma}\label{lem:gaussian-product}
Let $n \ge 2$ and $Z=(Z_1,...,Z_n)$ be a Gaussian vector with mean zero and $\EE Z_{i}^2 = \EE Z_{j}^2$ for any $i,j\in \{1,\dots, n\}$. Let  $P_{n}$ denote the set of partitions of $\{1,1,2,2,...,n,n\}$ into distinct ordered pairs $(i,j)$ (i.e. $i < j$) and $\widetilde{P}_n$ denote the set of partitions of $\{1,2,3,3,...,n,n\}$ into distinct ordered pairs. Then
\begin{align}
\mathbb{E} \Big[ \prod_{i=1}^n ( Z_i^2 - \mathbb{E} Z_i^2) \Big] = \sum_{p \in P_n} \alpha_{p,n} \prod_{\{i,j\} \in p} \mathbb{E} [Z_i Z_j]\label{eq:gaussian-product1} 
\end{align}
and
\begin{align}
\mathbb{E} \Big[ Z_1 Z_2 \prod_{i=3}^n (Z_i^2-\mathbb{E} Z_i^2) \Big] = \sum_{p \in \widetilde{P}_n} \beta_{p,n} \prod_{\{i,j\} \in p} \mathbb{E} [Z_i Z_j], \label{eq:gaussian-product2}
\end{align}
where $\{\alpha_{p,n} \}_{p \in P_n}$ and $\{\beta_{p,n} \}_{p \in \widetilde{P}_n}$ are constants independent of the law of $Z$.
\end{lemma}

\begin{remark}
Rigorously speaking, the set $\{1,1,2,2, \dots, n,n \}$ is to be understood as a set of $2n$ different elements, where each elements has the corresponding value ($1,1,2,2,\dots,n,n$). We formally identify an element with its value.
\end{remark}

\begin{proof}
The proof is by induction. We first handle the case $n=2$, for which \eqref{eq:gaussian-product2} trivially holds. As for \eqref{eq:gaussian-product1},
\begin{align*}
\mathbb{E} \Big[ (Z_1^2-\mathbb{E} Z_1^2)(Z_2^2- \mathbb{E} Z_2^2 ) \Big] = \mathbb{E} [Z_1^2 Z_2^2] - \mathbb{E} [Z_1^2] \mathbb{E} [Z_2^2] .
\end{align*}
By the Feynman diagram formula \cite[Theorem 1.28]{Janson}, we have 
\begin{align*}
\mathbb{E} [Z_1^2 Z_2^2]   = 2 \left( \mathbb{E}[Z_1 Z_2] \right)^2 + \mathbb{E}[Z_1^2] \mathbb{E} [Z_2^2] , 
\end{align*}
and therefore $\mathbb{E} [ (Z_1^2-\mathbb{E} Z_1^2)(Z_2- \mathbb{E} Z_2^2 ) ] = 2 \left( \mathbb{E}[Z_1 Z_2] \right)^2$.
Since $P_2 =\left\{ \{(1,2),(1,2) \} \right\}$, \eqref{eq:gaussian-product1} holds. 

~

Now let $n \geq 2$ and assume that \eqref{eq:gaussian-product1} and \eqref{eq:gaussian-product2} hold for any $m \leq n$ and any centred Gaussian vector of size $m$ such that $\EE Z_{i}^2 = \EE Z_{j}^2$ for any $i,j\in \{1,\dots, m\}$. 

First, let us prove that \eqref{eq:gaussian-product2} holds at rank $n+1$. Using the Gaussian integration-by-parts formula (see e.g. \cite[Lemma 2.1]{robert2008hydrodynamic}, applied to the  function $G(Z_2,...,Z_{n+1}):=Z_2  \prod_{i=3}^{n+1} ( Z_i^2-\EE Z_{i}^2)$), we get that
\begin{align}\label{eq:gaussianipp}
\mathbb{E} \Big[ Z_1 Z_2 \prod_{i=3}^{n+1} ( Z_i^2-\mathbb{E} Z_i^2) \Big] = 2 \sum_{m=3}^{n+1} \mathbb{E} [Z_1 Z_m] \mathbb{E} \Big[ Z_2 Z_m \prod_{\substack{k=3 \\ k \neq m }}^{n+1} (Z_k^2-\mathbb{E} Z_k^2 ) \Big] + \mathbb{E} [Z_1 Z_2] \mathbb{E} \Big[ \prod_{k=3}^{n+1} (Z_k^2-\mathbb{E} Z_k^2 ) \Big] .
\end{align}
For each $m \in \{3,...,n+1\}$, apply \eqref{eq:gaussian-product2} at rank $n$ to get that 
\begin{align*}
\mathbb{E} \Big[ Z_2 Z_m \prod_{\substack{k=3 \\ k \neq m }}^{n+1} (Z_k^2-\mathbb{E} Z_k^2 ) \Big] =  \sum_{p \in \widetilde{P}_n} \beta_{p,n} \prod_{\{i,j\} \in p} \mathbb{E} [\widetilde{Z}_i \widetilde{Z}_j] ,
\end{align*}
where $\widetilde{Z}_{1} = Z_{2}$, $\widetilde{Z}_{2} = Z_{m}$, $\widetilde{Z}_{i} = Z_{i}$ for $3\leq i<m$ and $\widetilde{Z}_{i} = Z_{i+1}$ for $m<i\leq n$. 
Moreover, applying \eqref{eq:gaussian-product1} on the term $\mathbb{E} \left[ \prod_{k=3}^{n+1} (Z_k^2-\mathbb{E} Z_k^2 ) \right]$ permits to conclude that the following equality holds:
\begin{align}\label{eq:induction2}
\mathbb{E} \Big[ Z_1 Z_2 \prod_{i=3}^{n+1} ( Z_i^2-\mathbb{E} Z_i^2) \Big] = \sum_{p \in \widetilde{P}_{n+1}} \beta_{p,n+1} \prod_{\{i,j\} \in p} \mathbb{E} [Z_i Z_j] ,
\end{align}
for some constants $\{\beta_{p,n+1}\}_{p \in \widetilde{P}_{n+1}}$. Note that the sum is indeed over $\widetilde{P}_{n+1}$, since in \eqref{eq:gaussianipp} the term $\mathbb{E} [Z_1 Z_2]$ appears only once in $\mathbb{E} [Z_1 Z_2] \mathbb{E}[ \prod_{k=3}^{n+1} (Z_k^2-\mathbb{E} Z_k^2 ) ]$ and covariances involving $Z_{1}$ and $Z_{2}$ also appear only once in $\mathbb{E} [Z_1 Z_m] \mathbb{E} [ Z_2 Z_m \prod_{k=3, k \neq m }^{n+1} (Z_k^2-\mathbb{E} Z_k^2 ) ]$. 

Finally, let us prove that \eqref{eq:gaussian-product1} holds at rank $n+1$.
Observe that
\begin{align*}
\mathbb{E} \Big[ \prod_{i=1}^{n+1} ( Z_i^2 - \mathbb{E} Z_i^2) \Bigg] = \mathbb{E} \Big[ Z_1 Z_1  \prod_{i=2}^{n+1} ( Z_i^2 - \mathbb{E} Z_i^2) \Big] - \mathbb{E} [Z_1^2]\, \mathbb{E} \Big[ \prod_{i=2}^{n+1} ( Z_i^2 - \mathbb{E} Z_i^2) \Big],
\end{align*}
then use again the Gaussian integration-by-parts (\cite[Lemma 2.1]{robert2008hydrodynamic} with the function $G(Z_1,...Z_{n+1})=Z_1 \prod_{i=2}^{n+1} \left( Z_i^2 - \EE Z_{i}^2\right)$) to get that
\begin{align*}
\mathbb{E} \Big[ Z_1^2  \prod_{i=2}^{n+1} ( Z_i^2 - \mathbb{E} Z_i^2) \Big] = 2 \sum_{m=2}^n  \mathbb{E} [Z_1 Z_m]\, \mathbb{E} \Big[ Z_1 Z_m \prod_{\substack{k=2 \\ k \neq m}}^{n+1} ( Z_i^2 - \mathbb{E} Z_i^2) \Big] + \mathbb{E} [Z_1^2]\, \mathbb{E}  \Big[ \prod_{i=2}^{n+1} ( Z_i^2 - \mathbb{E} Z_i^2) \Big] .
\end{align*}
Thus we have
\begin{align*}
\mathbb{E} \Big[ \prod_{i=1}^{n+1} ( Z_i^2 - \mathbb{E} Z_i^2) \Big] = 2 \sum_{m=2}^n  \mathbb{E} [Z_1 Z_m]\, \mathbb{E} \Big[ Z_1 Z_m \prod_{\substack{k=2 \\ k \neq m}}^{n+1} ( Z_i^2 - \mathbb{E} Z_i^2) \Big] .
\end{align*}
Apply now \eqref{eq:induction2} on each term $\mathbb{E} [ Z_1 Z_m \prod_{k=2, k \neq m}^{n+1} ( Z_i^2 - \mathbb{E} Z_i^2) ]$ to conclude that
\begin{align*}
\mathbb{E} \Big[ \prod_{i=1}^{n+1} ( Z_i^2 - \mathbb{E} Z_i^2) \Big] = \sum_{p \in P_{n+1}} \alpha_{p,n+1} \prod_{\{i,j\} \in p} \mathbb{E} [Z_i Z_j],
\end{align*}
for some constants $\{\alpha_{p,n+1}\}_{p \in P_{n+1}}$.
\end{proof}

Recall that the process $\{V_t^{H,K},~t\geq 0, (H,K)\in(0,1)^2\}$ was defined in \eqref{eq:defV}. In view of the stationarity of $\overline{U}$ (see Lemma \ref{lem:asymptotic-cov}), we have
\begin{align*}
V_t^{H,K} = \frac{1}{t+1}\int_0^{t+1} \left( | \overline{U}_s^H - \overline{U}_s^K |^2  - \mathbb{E} | \overline{U}_s^H - \overline{U}_s^K |^2 \right) \diff s .
\end{align*}

\begin{remark}\label{rmk:Pnasderangements}
Let $n \in \mathbb{N} \setminus \{0, 1\}$ and $\pi \in P_n$. We sort the elements of $\pi$ in the lexicographic order. We construct $\tilde{\pi}$, a permutation of $\{1,\dots,n\}$ as follows : 
\begin{itemize}
\item The first element of $\pi$ is of the form $(1,i_1)$ for some $i_1 \in \{ 2,\dots,n \} $. Then let $\tilde{\pi}(1)=i_1$. Besides $(1,i_1)$, there is a unique couple in $\pi$ where $i_1$ appears, which is of the form $(i_1,i_2)$ or $(i_2,i_1)$ with $i_2 \in \{1, \dots, n \} \setminus\{  i_1 \}$. Then let $\tilde{\pi}(i_1)=i_2$, and repeat the same process for $i_2$. 
 \item  Since every number appears exactly two times in $\pi$, the image of any $i_k$ cannot be $i_{k'}$ for $k' < k$ (otherwise $i_{k'}$ would appear three times in $\pi$). Therefore, the process ends when we set the pre-image of $1$.
 \item This way we have defined a cycle $c_1$ which is defined by the elements $1, i_1, i_2, \dots $
\item Then we pick the next number in $\pi$ that does not appear in the previous cycle, and we repeat the same process starting from this number.
\item Putting all the cycles together, we have obtained $\tilde{\pi}= c_1 ...c_N$ for some $N \ge 1$, the number of cycles in $\tilde{\pi}$.
\end{itemize}
Therefore, any element of $P_n$ can be seen as a derangement (a permutation without fixed point) where each pair represents a number and its image. We use this identification in the proof of Lemma \ref{lem:young-V}.
\end{remark}

\begin{lemma}\label{lem:young-V}
Let $\mathcal{H}$ be a compact subset of $(0,1)$. For any $p \in  \mathbb{N} \setminus \{0\}$, there exists a constant $C$ such that for any $t \ge 0$ and any $H,K \in \mathcal{H}$,
\begin{align*}
\mathbb{E} |V_t^{H,K}|^{2p} & \leq  C  \left( (t+1)^{ -\frac{2p}{3}\left( 1 \wedge (4-4 \max(\mathcal{H}))  \right) } \right)  (\log(t+1)+1)^{2p} .
\end{align*}
\end{lemma}

\begin{proof}
As explained in Remark~\ref{rk:dimension}, it is enough to write the proof in dimension $d=1$. By Fubini's theorem, we have 
\begin{align*}
\mathbb{E } \big|V_t^{H,K}\big|^{2p} = \frac{1}{(t+1)^{2p}} \mathbb{E}  \int_{[0,t+1]^p}   \prod_{i=1}^{2p} \left( \big( \overline{U}_{s_i}^H - \overline{U}_{s_i}^K \big)^2  - \mathbb{E} \big( \overline{U}_{s_i}^H - \overline{U}_{s_i}^K \big)^2 \right) 
  \diff s_1 \dots \diff s_{2p}   . 
\end{align*}
Apply Lemma \ref{lem:gaussian-product} with $Z_i= \overline{U}_{s_i}^H - \overline{U}_{s_i}^K$ and set $C=\displaystyle \max_{\pi \in P_{2p}} \alpha_{\pi,2p}$ to get
\begin{align}\label{eq:pintegrals}
\mathbb{E } \big|V_t^{H,K}\big|^{2p} \leq C  \sum_{\pi \in P_{2p}} \frac{1}{(t+1)^{2p}} \int_{[0,t+1]^{2p}}  \prod_{(i,j) \in \pi} \left| \mathbb{E} \left[ ( \overline{U}_{s_i}^H - \overline{U}_{s_i}^K )  ( \overline{U}_{s_j}^H - \overline{U}_{s_j}^K ) \right]  \right| \diff s_1\dots \diff s_{2p}  .
\end{align}
Let $\pi \in P_{2p}$, by Remark \ref{rmk:Pnasderangements}, we can decompose $\pi$ as a product of cycles: $\pi  = \Pi_{\ell=1}^N c_{k_\ell}$, where $N$ is the total number of cycles and for each $\ell \in \llbracket 1, N\rrbracket$, $c_{k_\ell}$ is a cycle of length $k_\ell$. For such $\pi$, we write
\begin{equation}\label{eq:pintegrals-pi}
\begin{split}
& \int_{[0,t+1]^{2p}}  \prod_{(i,j) \in \pi} \left| \mathbb{E} [ ( \overline{U}_{s_i}^H - \overline{U}_{s_i}^K )  ( \overline{U}_{s_j}^H - \overline{U}_{s_j}^K ) ] \right| \diff s_1\dots \diff s_{2p} \\ & = \prod_{\ell=1}^N \int_{[0,t+1]^{k_{\ell}}} \prod_{(i,j) \in c_{k_\ell}}  \left| \mathbb{E} [ ( \overline{U}_{s_i}^H - \overline{U}_{s_i}^K )  ( \overline{U}_{s_{j}}^H - \overline{U}_{s_{j}}^K ) ]    \right|   \diff s_{i} . 
\end{split}
\end{equation}
Let $\ell \in \llbracket 1, N \rrbracket$. We will bound the integral in the right-hand side above depending on the value of $k_\ell$: First when $k_\ell=2$, then when $k_\ell > 2$ and  $k_\ell$ even and finally $k_\ell > 2$ and $k_\ell $ odd.
\paragraph{The case $k_\ell =2$.} In this case, $c_{k_\ell}$ is a transposition. Without any loss of generality, let us write $c_{k_\ell} = \{ (1,2) , (2,1) \}$. Then we have
\begin{align*}
\int_{[0,t+1]^{k_{\ell}}} \prod_{(i,j) \in c_{k_\ell}}  \left| \mathbb{E} \left[ ( \overline{U}_{s_i}^H - \overline{U}_{s_i}^K )  ( \overline{U}_{s_{j}}^H - \overline{U}_{s_{j}}^K ) \right] \right|   \diff s_{i} & = \int_{[0,t+1]^2} \left( \mathbb{E} [ ( \overline{U}_{s_1}^H - \overline{U}_{s_1}^K )  ( \overline{U}_{s_{2}}^H - \overline{U}_{s_{2}}^K ) ] \right)^2 \diff s_1 \diff s_2 
\end{align*}
Let us define $I_1 := \int_{[0,t+1]^2} \left( \mathbb{E} [ ( \overline{U}_{s_1}^H - \overline{U}_{s_1}^K )  ( \overline{U}_{s_{2}}^H - \overline{U}_{s_{2}}^K ) ] \right)^2 \diff s_1 \diff s_2 $. We have using Lemma \ref{lem:bound-I1-I2}
\begin{align}
\int_{[0,t+1]^{k_{\ell}}} \prod_{(i,j) \in c_{k_\ell}}  \left| \mathbb{E} \left[ ( \overline{U}_{s_i}^H - \overline{U}_{s_i}^K )  ( \overline{U}_{s_{j}}^H - \overline{U}_{s_{j}}^K ) \right] \right|   \diff s_{i} 
& \leq  C (\log(t+1)+1) (t+1)^{(4 \max(\mathcal{H})-2) \vee 1} . \label{eq:cas-k=2}
\end{align}

\paragraph{When $k_\ell > 2$ and $k_\ell $ even.} Without any loss of generality, let us write $$c_{k_\ell} = \{ (1,2), (2,3),...(k_\ell,1) \}.$$ Let us denote $\overline{\mathcal{U}}_{s} := \overline{U}_{s}^H - \overline{U}_{s}^K $. Using the notation $s_{k_\ell+1}= s_1$, we have
\begin{align*}
& \int_{[0,t+1]^{k_{\ell}}} \prod_{(i,j) \in c_{k_\ell}} \left| \mathbb{E} \left[ ( \overline{\mathcal{U}}_{s_i}^H - \overline{\mathcal{U}}_{s_i}^K )  ( \overline{\mathcal{U}}_{s_{j}}^H -\overline{\mathcal{U}}_{s_{j}}^K ) \right]  \right| \diff s_{i}  \\
& =  \int_{[0,t+1]^{k_{\ell}}} \left( \prod_{ \substack{i \in \llbracket 1, k_\ell \rrbracket \\   i \text{ odd }}} \left| \mathbb{E} \left[ \overline{\mathcal{U}}_{s_i} \overline{\mathcal{U}}_{s_{i+1}} \right] \right|  \right) \left( \prod_{\substack{i \in \llbracket 1, k_\ell \rrbracket \\ i \text{ even }}} \left| \mathbb{E} \left[ \overline{\mathcal{U}}_{s_i} \overline{\mathcal{U}}{s_{i+1}} \right] \right|  \right)  \prod_{i=1}^{k_\ell} \diff s_i .
\end{align*}
Then by Young's inequality, we get
\begin{align*}
& \int_{[0,t+1]^{k_{\ell}}} \prod_{(i,j) \in c_{k_\ell}} \left| \mathbb{E} [ (\overline{\mathcal{U}}_{s_i}^H -\overline{\mathcal{U}}_{s_i}^K )  ( \overline{\mathcal{U}}_{s_{j}}^H - \overline{\mathcal{U}}_{s_{j}}^K ) ]  \right| \prod_{(i,j) \in c_{k_\ell}} \diff s_{i}  \\ 
& \leq \frac{1}{2} \int_{[0,t+1]^{k_{\ell}}} \left( \prod_{ \substack{i \in \llbracket 1, k_\ell \rrbracket \\   i \text{ odd }}} \left( \mathbb{E} [ \overline{\mathcal{U}}_{s_i} \overline{\mathcal{U}}_{s_{i+1}}] \right)^2  \right)  \prod_{i=1}^{k_\ell} \diff s_i + \frac{1}{2}\int_{[0,t+1]^{k_{\ell}}} \left( \prod_{ \substack{i \in \llbracket 1, k_\ell \rrbracket \\   i \text{ even }}}  \left( \mathbb{E} [ \overline{\mathcal{U}}_{s_i} \overline{\mathcal{U}}_{s_{i+1}}] \right)^2  \right)  \prod_{i=1}^{k_\ell} \diff s_i \\
& = \frac{1}{2} I_1^{k_\ell/2} + \frac{1}{2} I_1^{k_\ell/2} = I_1^{k_\ell/2} .
\end{align*}
Applying Lemma \ref{lem:bound-I1-I2}, we get 
\begin{align}
& \int_{[0,t+1]^{k_{\ell}}} \prod_{(i,j) \in c_{k_\ell}} \left| \mathbb{E} \left[ ( \overline{U}_{s_i}^H - \overline{U}_{s_i}^K )  ( \overline{U}_{s_{j}}^H - \overline{U}_{s_{j}}^K ) \right]  \right|  \diff s_{i}  \nonumber \\ 
& \quad \leq C (\log(t+1)+1)^{k_\ell/2} (t+1)^{ ( (4 \max(\mathcal{H}) -2) \vee 1  ) k_\ell/2}.  \label{eq:cas-k-pair}
\end{align}
\paragraph{The case $k_\ell > 2$ and $k_\ell$ odd.} Without any loss of generality, let us write $$c_{k_\ell} = \{ (1,2), (2,3),..., (k_{\ell}-2,k_{\ell}-1), (k_{\ell}-1, k_\ell), (k_\ell,1) \}.$$ 
We write
\begin{align*}
& \int_{[0,t+1]^{k_{\ell}}} \prod_{(i,j) \in c_{k_\ell}} \left| \mathbb{E} \left[ ( \overline{U}_{s_i}^H - \overline{U}_{s_i}^K )  ( \overline{U}_{s_{j}}^H - \overline{U}_{s_{j}}^K ) \right]  \right|  \diff s_{i}  \\
& =  \int_{[0,t+1]^{k_{\ell}}} \left( \prod_{ \substack{i \in \llbracket 1, k_\ell-2 \rrbracket \\   i \text{ odd }}} \left| \mathbb{E} [ \overline{\mathcal{U}}_{s_i} \overline{\mathcal{U}}_{s_{i+1}}] \right|  \right) \left( \prod_{\substack{i \in \llbracket 1, k_\ell-2 \rrbracket \\ i \text{ even }}} \left| \mathbb{E} [ \overline{\mathcal{U}}_{s_i} \overline{\mathcal{U}}_{s_{i+1}}] \right|  \right) | \mathbb{E} [ \overline{\mathcal{U}}_{s_{k_\ell-1}} \overline{\mathcal{U}}_{s_{k_\ell}} ] \, \mathbb{E} [ \overline{\mathcal{U}}_{s_{k_\ell}} \overline{\mathcal{U}}_{s_1} ]  | \prod_{i=1}^{k_\ell} \diff s_i .
\end{align*}
Using again Young's inequality, we have
\begin{align*}
& \int_{[0,t+1]^{k_{\ell}}} \prod_{(i,j) \in c_{k_\ell}} \left| \mathbb{E} \left[ ( \overline{U}_{s_i}^H - \overline{U}_{s_i}^K )  ( \overline{U}_{s_{j}}^H - \overline{U}_{s_{j}}^K ) \right]  \right|  \diff s_{i}  \\ 
& \leq \frac{1}{2} \int_{[0,t+1]^{k_{\ell}}}  \left( \prod_{ \substack{i \in \llbracket 1, k_\ell-2 \rrbracket \\   i \text{ odd }}} \left( \mathbb{E} [ \overline{\mathcal{U}}_{s_i} \overline{\mathcal{U}}_{s_{i+1}}] \right)^2 \right) \prod_{i=1}^{k_\ell} s_i \\ & \quad + \frac{1}{2}  \int_{[0,t+1]^{k_\ell}} \left( \prod_{\substack{i \in \llbracket 1, k_\ell-2 \rrbracket \\ i \text{ even }}} \left( \mathbb{E} [ \overline{\mathcal{U}}_{s_i} \overline{\mathcal{U}}_{s_{i+1}}] \right)^2 \right)  \mathbb{E} [ \overline{\mathcal{U}}_{s_{k_\ell-1} \overline{\mathcal{U}}_{s_{k_\ell}}} ]^2 \,  \mathbb{E} [ \overline{\mathcal{U}}_{s_{k_\ell}} \overline{\mathcal{U}}_{s_1} ]^2  \prod_{i=1}^{k_\ell} \diff s_i  \\
& = \frac{1}{2} (t+1) I_1^{\lceil (k_\ell-2)/2 \rceil}  + \frac{1}{2} I_1^{\lceil (k_\ell-2)/2 \rceil -1} \int_{[0,t+1]^3}   \mathbb{E} [ \overline{\mathcal{U}}_{s_{k_\ell-1} \overline{\mathcal{U}}_{s_{k_\ell}}} ]^2 \mathbb{E} [ \overline{\mathcal{U}}_{s_{k_\ell}} \overline{\mathcal{U}}_{s_1} ]^2  \diff s_{k_\ell-1} \diff s_{k_\ell} \diff s_1 .
\end{align*}
Define $I_2 = \int_{[0,t+1]^3}   \mathbb{E} [ \overline{\mathcal{U}}_{s_{1}} \overline{\mathcal{U}}_{s_{2}} ]^2 \mathbb{E} [ \overline{\mathcal{U}}_{s_{2}} \overline{\mathcal{U}}_{s_3} ]^2  \diff s_{1} \diff s_{2} \diff s_3 $, then we have
\begin{align*}
& \int_{[0,t+1]^{k_{\ell}}} \prod_{(i,j) \in c_{k_\ell}} \left| \mathbb{E} \left[ ( \overline{U}_{s_i}^H - \overline{U}_{s_i}^K )  ( \overline{U}_{s_{j}}^H - \overline{U}_{s_{j}}^K ) \right]  \right| \diff s_{i}   \leq  \frac{1}{2} (t+1) I_1^{\lceil (k_\ell-2)/2 \rceil}  +  \frac{1}{2} I_1^{\lceil (k_\ell-2)/2 \rceil -1} I_2  . 
\end{align*}
Denote $K_\ell = \lceil (k_\ell-2)/2 \rceil$. Applying Lemma \ref{lem:bound-I1-I2}, we get that
\begin{align*}
& \int_{[0,t+1]^{k_{\ell}}} \prod_{(i,j) \in c_{k_\ell}} \left| \mathbb{E} \left[ ( \overline{U}_{s_i}^H - \overline{U}_{s_i}^K )  ( \overline{U}_{s_{j}}^H - \overline{U}_{s_{j}}^K ) \right]  \right| \diff s_{i}  \\ 
& \leq C (t+1)   (\log(t+1)+1)^{K_\ell} (t+1)^{( (4\max(\mathcal{H})-2)\vee 1 ) K_\ell}   +   (\log(t+1)+1)^{K_\ell-1}  (t+1)^{( (4\max(\mathcal{H})-2) \vee 1 ) (K_\ell-1) }  \\ & \quad \times   (\log(t+1) + 1)^{2}  (t+1)^{(8 \max(\mathcal{H})-5) \vee 1 }    \\
& \leq C  (\log(t+1)+1)^{K_\ell} (t+1)^{( (4\max(\mathcal{H})-2)\vee 1 ) K_\ell + 1}   \\ & \quad + (\log(t+1)+1)^{K_\ell+1}  (t+1)^{( (4\max(\mathcal{H})-2) \vee 1 ) (K_\ell-1) + (8 \max(\mathcal{H})-5) \vee 1 } .
\end{align*}
To simplify the previous bound, let us prove upper bounds on the different powers of $(t+1)$ that appear. If $\max(\mathcal{H}) > 3/4 > 5/8$, since $K_\ell \ge 1$, we always have
\begin{align*}
\left( (4\max(\mathcal{H})-2) \vee 1 \right) (K_\ell-1) + \left( (8 \max(\mathcal{H})-5) \vee 1 \right) & =(4\max(\mathcal{H})-2)  (K_\ell -1) + (8 \max(\mathcal{H})-5) 
\\ & \leq (4\max(\mathcal{H})-2)  K_\ell + 1 \\ & = \left( (4\max(\mathcal{H})-2) \vee 1 \right) K_\ell+1.
\end{align*}
 If $5/8 \leq  \max(\mathcal{H}) \leq 3/4$, since $8 \max(\mathcal{H}) - 6  \leq 1$, we have
\begin{align*}
\left( (4\max(\mathcal{H})-2) \vee 1 \right) (K_\ell-1) + \left( (8 \max(\mathcal{H})-5) \vee 1 \right) & = (K_\ell -1)  + (8 \max(\mathcal{H})-5) \\
& = K_\ell + 8 \max(\mathcal{H}) - 6 
\\ & \leq  K_\ell  + 1 = \left( (4\max(\mathcal{H})-2) \vee 1 \right) + 1 .
\end{align*}
Finally, if $\max(\mathcal{H}) < 5/8$ then 
\begin{align*}
\left( (4\max(\mathcal{H})-2) \vee 1 \right) (K_\ell-1) + \left( (8 \max(\mathcal{H})-5) \vee 1 \right) & = (K_\ell -1)  + 1 \\
& \leq K_\ell + 1 = \left( (4\max(\mathcal{H})-2) \vee 1 \right) K_\ell+1 .
\end{align*}
Therefore,
\begin{align}
& \int_{[0,t+1]^{k_{\ell}}} \prod_{(i,j) \in c_{k_\ell}} \left| \mathbb{E} \left[ ( \overline{U}_{s_i}^H - \overline{U}_{s_i}^K )  ( \overline{U}_{s_{j}}^H - \overline{U}_{s_{j}}^K ) \right]  \right|  \diff s_{i} \nonumber  \\ 
& \leq C   (t+1)^{( (4\max(\mathcal{H})-2) \vee 1 ) K_\ell + 1}  (\log(t+1)+1)^{K_\ell+1}  \nonumber \\
& \leq C (t+1)   \left( (t+1)^{( (4\max(\mathcal{H})-2) \vee 1 ) } \right)^{\lceil (k_\ell-2)/2 \rceil} (\log(t+1)+1)^{\lceil (k_\ell-2)/2 \rceil+1} . \label{eq:cas-k-impair}
\end{align}
Define 
\begin{align*}
M & := \sum_{k_\ell = 2} 1 +\sum_{k_\ell>2, k_\ell \text{ even}}  k_\ell/2  +\sum_{k_\ell>2, k_\ell \text{ odd}}  \lceil (k_\ell-2)/2 \rceil \\
\alpha_{\mathcal{H}} & :=  (4\max(\mathcal{H})-2) \vee 1  \\
N_{\pi} & := \sum_{k_\ell>2, k_\ell \text{ odd}} 1 .
\end{align*}
Plugging \eqref{eq:cas-k=2}, \eqref{eq:cas-k-pair}, \eqref{eq:cas-k-impair} in \eqref{eq:pintegrals-pi}, we have
\begin{align*}
& \int_{[0,t+1]^{2p}}  \prod_{(i,j) \in \pi} \left| \mathbb{E} \left[ ( \overline{U}_{s_i}^H - \overline{U}_{s_i}^K )  ( \overline{U}_{s_j}^H - \overline{U}_{s_j}^K ) \right] \right| \diff s_1\dots \diff s_{2p} \\ & 
\quad \leq C  (t+1)^{N_{\pi}} \left( (t+1)^{\alpha_{\mathcal{H}}} \right)^{M}  (\log(t+1)+1)^{M+N_{\pi}}.
\end{align*}
Since $2p = k_1+ \dots + k_N$ and $\lceil (k_\ell-2)/2 \rceil =  (k_\ell-1)/2$ for $k_\ell >2$ and $k_\ell$ odd, we get $M = p- N_{\pi}/2$. Hence,
\begin{align*}
& \frac{1}{(t+1)^{2p}}\int_{[0,t+1]^{2p}}  \prod_{(i,j) \in \pi} \left| \mathbb{E} \left[ ( \overline{U}_{s_i}^H - \overline{U}_{s_i}^K )  ( \overline{U}_{s_j}^H - \overline{U}_{s_j}^K ) \right] \right| \diff s_1\dots \diff s_{2p} \\ & 
\quad \leq  C   (t+1)^{-2p}  (t+1)^{N_{\pi}} \left( (t+1)^{\alpha_{\mathcal{H}} } \right)^{p-N_{\pi}/2} (\log(t+1)+1)^{p + N_{\pi}/2 } \\
 & \quad  \leq  C   (t+1)^{\alpha_{\mathcal{H}}  p + (2-\alpha_{\mathcal{H}} )N_{\pi}/2 -2p}   (\log(t+1)+1)^{p + N_{\pi}/2} .
\end{align*}
Since $2-\alpha_{\mathcal{H}}  \ge 0$ and $3 N_{\pi} \leq 2 p$, we have $(2-\alpha_{\mathcal{H}} )N_{\pi}/2 \leq  (2-\alpha_{\mathcal{H}} )p/3 $. Using this in the previous bound, we get
\begin{align*}
& \frac{1}{(t+1)^{2p}}\int_{[0,t+1]^{2p}}  \prod_{(i,j) \in \pi} \left| \mathbb{E} \left[ ( \overline{U}_{s_i}^H - \overline{U}_{s_i}^K )  ( \overline{U}_{s_j}^H - \overline{U}_{s_j}^K ) \right] \right| \diff s_1\dots \diff s_{2p} \\ 
& \quad \leq C  \left( (t+1)^{ (2 \alpha_{\mathcal{H}} - 4) \frac{p}{3}} \right)  (\log(t+1)+1)^{p+\frac{p}{6}} ,
\end{align*}
which holds for any $\pi \in P_{2p}$. Plugging the previous bound in \eqref{eq:pintegrals} and noticing that $2 \alpha_{\mathcal{H}}-4 = -2 \vee (8 \max(\mathcal{H}) -8)$, concludes the proof.
\end{proof}

\begin{lemma}\label{lem:ergodic-U2}
Let $\mathcal{H}$ be a compact subset of $(0,1)$, $p \in \mathbb{N} \setminus \{0\}$ and $ \beta \in (0,1)$. There exists a constant $C$ such that for all $H,H', K, K'$ in $\mathcal{H}$ and $t' \ge t \ge 0$,
\begin{align*}
\mathbb{E} \Big| V_t^{H,K} - V_{t'}^{H',K'} \Big|^{2p} &   \leq C  \left( (t+1)^{ -\frac{2p}{3} \left(1 \wedge  (4-4 \max(\mathcal{H}) \right) (1-\beta)} \right)  (\log(t+1)+1)^{2p} \\ 
&\quad \times (|H-H'|^{2p \beta}+|K-K'|^{2p\beta}) +  C\, (t+1)^{-2p} |t-t'|^{2p} .
\end{align*}
\end{lemma}

\begin{proof}
In view of \eqref{eq:defV}, we have
\begin{align*}
& \mathbb{E} \Big| V_t^{H,K} - V_t^{H',K'} \Big|^{2p}  \\ & \quad  \leq  C\mathbb{E} \Big| \frac{1}{t+1} \int_0^{t+1} | \overline{U}_s^H -\overline{U}_s^K |^2-| \overline{U}_s^{H'}-\overline{U}_s^{K'} |^2 \diff s \Big|^{2p} + C\Big| \mathbb{E} | \overline{U}_0^H -\overline{U}_0^K |^2- \mathbb{E} | \overline{U}_0^{H'} -\overline{U}_0^{K'} |^2\Big|^{2p}  \\
& \quad  \leq C\, \mathbb{E} \Big| \frac{1}{t+1} \int_0^{t+1} \langle \overline{U}_s^H -\overline{U}_s^{H'} - \overline{U}_s^{K}+\overline{U}_s^{K'} , \overline{U}_s^H +\overline{U}_s^{H'} - \overline{U}_s^{K}-\overline{U}_s^{K'} \rangle \diff s \Big|^{2p}  \\
& \quad \quad + C\, | \mathbb{E} \langle \overline{U}_0^H -\overline{U}_0^{H'} - \overline{U}_0^{K}+\overline{U}_0^{K'} , \overline{U}_0^H +\overline{U}_0^{H'} - \overline{U}_0^{K}-\overline{U}_0^{K'} \rangle |^{2p} .
\end{align*}
By Cauchy-Schwarz's inequality and the stationarity of the process $\overline{U}$ (Lemma \ref{lem:asymptotic-cov}), there is
\begin{align*}
& \mathbb{E} \Big| V_t^{H,K} - V_t^{H',K'} \Big|^{2p} \\ & \quad  \leq C \mathbb{E} \Big| \frac{1}{t+1} \int_0^{t+1} \Big(\mathbb{E} | \overline{U}_s^H -\overline{U}_s^{H'} - \overline{U}_s^{K}+\overline{U}_s^{K'} |^2\Big)^{\frac{1}{2}} \diff s \Big|^{2p}  + C \Big(\mathbb{E}|\overline{U}_0^H -\overline{U}_0^{H'} - \overline{U}_0^{K}+\overline{U}_0^{K'}|^2 \Big)^{p}  \\
& \quad  \leq C  \Big| \frac{1}{t+1} \int_0^{t+1}  \Big( \mathbb{E} |\overline{U}_s^H -\overline{U}_s^{H'}|^2 \Big)^{\frac{1}{2}}  \diff s \Big|^{2p} + C\Big| \frac{1}{t+1} \int_0^{t+1}  \Big( \mathbb{E} |\overline{U}_s^{K} -\overline{U}_s^{K'} |^2 \Big)^{\frac{1}{2}}\diff s \Big|^{2p}  \\
&\quad \quad  + C \Big(\mathbb{E}  |\overline{U}_0^{K} -\overline{U}_0^{K'} |^2 \Big)^{p} + C \Big(\mathbb{E}  |\overline{U}_0^{H} -\overline{U}_0^{H'} |^2 \Big)^{p} .
\end{align*}
Using Lemma \ref{lem:Ubar-increments}, it comes
\begin{align*}
\mathbb{E} \Big| V_t^{H,K} - V_t^{H',K'} \Big|^{2p} \leq C  (|H-H'|^{2p}+|K-K'|^{2p})  \ .
\end{align*}
Hence, interpolating between the previous inequality and the result of Lemma \ref{lem:young-V},
it follows that 
\begin{align}\label{eq:H-increments-V2}
\mathbb{E} \Big| V_t^{H,K} - V_t^{H',K'} \Big|^{2p} &  \leq C\,  \left( (t+1)^{ -2p/3\left( 1 \wedge (4-4 \max(\mathcal{H}) \right) (1-\beta)} \right)  (\log(t+1)+1)^{2p}  \nonumber \\ 
&\quad \times (|H-H'|^{2p \beta}+|K-K'|^{2p \beta}) .
\end{align}
Moreover, we have for $t'\geq t$,
\begin{align*}
\mathbb{E} \Big| V_t^{H,K} - V_{t'}^{H,K} \Big|^{2p} & \leq C\, \mathbb{E} \Big( \frac{1}{t'+1} \int_{t+1}^{t'+1} | \overline{U}_s^H - \overline{U}_s^{K} |^2 \diff s \Big)^{2p} +C\, \mathbb{E} \Big( \frac{t'-t}{(t+1) (t'+1)} \int_0^{t+1} | \overline{U}_s^H - \overline{U}_s^{K} |^2 \diff s \Big)^{2p}  \\
& \leq C  (t'+1)^{-2p} |t'-t|^{2p-1} \int_{t+1}^{t'+1} \mathbb{E} | \overline{U}_s^H - \overline{U}_s^{K} |^{4p} \diff s \\ &  \quad + |t'-t|^{2p} (t+1)^{-2p-1}  \int_0^{t+1} \mathbb{E} | \overline{U}_s^H - \overline{U}_s^{K} |^{4p} \diff s ,
\end{align*}
using Jensen's inequality. Since $\overline{U}$ is a stationary process, we conclude that
\begin{align}\label{eq:t-increments-V2} 
\mathbb{E} \Big| V_t^{H,K} - V_{t'}^{H,K} \Big|^p & \leq C (t+1)^{-2p} |t'-t|^{2p} .
\end{align}
Combining \eqref{eq:H-increments-V2} and \eqref{eq:t-increments-V2}, we get the desired result.
\end{proof}
Since we are also interested in discrete ergodic means, the next Lemmas state discrete equivalents of Lemma \ref{lem:young-V} and Lemma \ref{lem:ergodic-U2}.

\begin{lemma}\label{lem:decreasing-variance-discret}
Let $\gamma \in (0,1)$ and $\mathcal{H}$ be a compact subset of $(0,1)$. Recall the process $\mathcal{V}_t^{H,K}$ defined in \eqref{eq:V-discret}. Let $p \in \mathbb{N} \setminus \{0 \}$. There exists a constant $C$ such that for all $t \ge 0$ and $H,K \in \mathcal{H}$,
\begin{align*}
\mathbb{E} |\mathcal{V}_t^{H,K}|^{2p} \leq  C\, \left( (t+1)^{-\frac{2p}{3} \left(1 \wedge (4-4 \max(\mathcal{H}) \right) } \right)  (\log(t+1)+1)^{2p} .
\end{align*}
\end{lemma}

\begin{proof}
The proof of Lemma \ref{lem:young-V} can be transcribed to a discrete setting and the same computations can be done. Let us first prove a discrete-equivalent of Lemma \ref{lem:bound-I1-I2}. Let 
\begin{align*}
I_{1}^*&:=  \int_{[0,t+1]^2} \left( \mathbb{E} \langle \overline{U}_{s_\gamma}^H-\overline{U}_{s_\gamma}^K, \overline{U}_{v_\gamma}^H-\overline{U}_{v_\gamma}^K\rangle \right)^2 \diff s \diff v   ,\\ 
I_{2}^* &:=  \int_{[0,t+1]^3} \Big( \mathbb{E} \langle \overline{U}_{s_\gamma}^H-\overline{U}_{s_\gamma}^K, \overline{U}_{v_\gamma}^H-\overline{U}_{v_\gamma}^K\rangle \Big)^2 \, \Big( \mathbb{E} \langle \overline{U}_{v_\gamma}^H-\overline{U}_{v_\gamma}^K, \overline{U}_{r_\gamma}^H-\overline{U}_{r_\gamma}^K \rangle \Big)^2 \,  \diff r \diff s \diff v  .
\end{align*}
To bound $I_1^*$, we proceed as in the proof of Lemma \ref{lem:bound-I1-I2}. Using Lemma \ref{lem:asymptotic-cov}, we have
\begin{align*}
|\mathbb{E} \langle \overline{U}_{s_\gamma}^H-\overline{U}_{s_\gamma}^K, \overline{U}_{v_\gamma}^H-\overline{U}_{v_\gamma}^K \rangle | &\leq |\EE \langle \overline{U}_{s_\gamma}^H, \overline{U}_{v_\gamma}^H \rangle| + |\EE \langle \overline{U}_{s_\gamma}^K, \overline{U}_{v_\gamma}^K \rangle | +|\EE \langle \overline{U}_{s_\gamma}^H,  \overline{U}_{v_\gamma}^K \rangle + \langle \overline{U}_{s_\gamma}^K , \overline{U}_{v_\gamma}^H \rangle |\nonumber \\
&\leq C \left(1 \wedge |s_\gamma-v_\gamma|^{2 \max(\mathcal{H})-2} \right). 
\end{align*}
It follows that
\begin{align*}
I_1^* & \leq C\,  \int_{[0,t+1]^2} 1 \wedge |s_\gamma-v_\gamma|^{4 \max(\mathcal{H})-4} \, \diff s \diff v  = C\,  \gamma^2\sum_{i,j=0}^{\gamma^{-1}(t+1)_\gamma} 1 \wedge | i \gamma - j \gamma|^{4 \max(\mathcal{H})-4} .
\end{align*}
Let $s \in [i \gamma, (i+1) \gamma]$, $ v \in [j \gamma , (j+1) \gamma]$, $r \in [k \gamma, (k+1) \gamma]$, and denote $\rho_1 =4 \max(\mathcal{H})-4$. If $|j\gamma - i \gamma | \ge 1 >\gamma$, then $|s-v| \leq |i \gamma - j \gamma| + \gamma \leq 2 |i \gamma - j \gamma|$. While if $|j\gamma - i \gamma | < 1$ then $|s-v| < 1+\gamma<2$. In this case, we have $1 \wedge |i \gamma - j \gamma |^{\rho_{1}} = 1 =  1 \wedge 2^{-\rho_{1}}|s-v|^{\rho_{1}} \leq 2^{-\rho_{1}} (1 \wedge |s-v|^{\rho_{1}})$. So overall we always have
\begin{align}\label{eq:grid-comp1}
1 \wedge |s-v|^{\rho_1} \ge 2^{\rho_1} (1 \wedge |i \gamma - j \gamma |^{\rho_1} ) .
\end{align}
Summing \eqref{eq:grid-comp1} over $i$ and $j$, we get
\begin{align*}
I_1^* & \leq C\,  \int_{[0,(t+1)_\gamma]^2} 1 \wedge |s-v|^{4 \max(\mathcal{H})-4}  \diff s \diff v .
\end{align*}
We already bounded the right-hand side above in the proof of Lemma \ref{lem:bound-I1-I2} (see $I_1$). Using that bound, we have
\begin{align*}
I_1^* & \leq C\,   (\log((t+1)_\gamma)+1) (t+1)_\gamma^{(4 \max(\mathcal{H})-2) \vee 1}   \\
& \leq C\,  (\log((t+1))+1) (t+1)^{(4 \max(\mathcal{H})-2) \vee 1}  .
\end{align*}
Similarly, for $I_2^*$, we have using \eqref{eq:grid-comp1}
\begin{align*}
I_2^* & \leq C\,  \int_{[0,t+1]^3} \Big( 1 \wedge |s_\gamma-v_\gamma|^{4 \max(\mathcal{H})-4} \Big)\Big( 1 \wedge |v_\gamma-r_\gamma|^{4 \max(\mathcal{H})-4}   \, \diff s \diff v \diff r \Big)  \\
& = C\,  \gamma^3 \sum_{i,j,k=0}^{\gamma^{-1}(t+1)_\gamma} \Big( 1 \wedge | i \gamma - j \gamma|^{4 \max(\mathcal{H})-4} \Big) \Big( 1 \wedge | j \gamma - k \gamma|^{4 \max(\mathcal{H})-4} \Big) \\
& \leq C  \int_{[0,(t+1)_\gamma]^3}  \Big( 1 \wedge |s-v|^{4 \max(\mathcal{H})-4} \Big)\Big( 1 \wedge |v-r|^{4 \max(\mathcal{H})-4} \Big)  \, \diff s \diff v \diff r .
\end{align*}
We already bounded the right-hand side above in the proof of Lemma \ref{lem:bound-I1-I2} (see $I_2$). Using that bound, we have
\begin{align*}
I_2^* & \leq C\,     (\log((t+1)_\gamma)+1)^2 (t+1)_\gamma^{(8 \max(\mathcal{H})-5) \vee 1}   \\
& \leq C\, (\log((t+1))+1)^2 (t+1)^{(8 \max(\mathcal{H})-5) \vee 1}.
\end{align*}
With these bounds in mind, we claim that Lemma \ref{lem:young-V} still holds if the integral is replaced by a discrete sum, as it relies essentially on Lemma \ref{lem:gaussian-product}, Lemma \ref{lem:bound-I1-I2} and properties of the integral that are also true for the sum. Hence
\begin{align*}
\mathbb{E} |\mathcal{V}_t^{H,K}|^{2p} \leq C\, \left( (t+1)^{-\frac{2p}{3} \left( 1 \wedge (4-4 \max(\mathcal{H})  \right) } \right)  (\log(t+1)+1)^{2p} .
\end{align*}
\end{proof}

\begin{lemma}\label{lem:ergodic-U2-dsicrete}
Let $\gamma \in (0,1)$ and $\mathcal{H}$ be a compact set of $(0,1)$. Let $p \in \mathbb{N} \setminus \{ 0 \}$ and $\beta \in (0,1)$. There exists a constant $C$ such that for all $H,H', K, K'$ in $\mathcal{H}$ and $t' \ge t \ge 0$,
\begin{align*} 
\mathbb{E} | \mathcal{V}_t^{H,K} - \mathcal{V}_{t'}^{H',K'} |^{2p}  & \leq  C\,  \left(  (t+1)^{ -\frac{2p}{3}\left( 1 \wedge (4-4 \max(\mathcal{H}) \right)(1-\beta) } \right)  (\log(t+1)+1)^{2p} \\ 
& \quad \times   (|H-H'|^{2p \beta}+|K-K'|^{2p\beta}) +  (1+t)^{-2p} |t-t'|^{2p} \Big) .
\end{align*}
\end{lemma}

\begin{proof}
This lemma is a discrete version of Lemma \ref{lem:ergodic-U2} and relies on Lemma \ref{lem:decreasing-variance-discret}. The proof of Lemma \ref{lem:ergodic-U2} can be completely transcribed to a discrete setting (it suffices to replace the dummy variable $s$ by $s_{\gamma}$ in the integrals).
\end{proof}

\section{Proof of the results on the rectangular increments}\label{app:rect-inc}
The Garsia-Rodemich-Rumsey lemma also has a version for the rectangular increments of two-parameter functions. 
\begin{lemma}\label{lem:GRRrect}
Let $d \in \mathbb{N}^*$. Let $f:\R \times \R^{d}\to \R$ be a continuous function. Let $\mathcal{I}$ be a product of $d+1$ closed intervals of $\R$.
Then for any $p>0$ and any $\theta_1, \theta_2, \eta > \frac{1}{p}$, there exists a constant $C>0$ independent of $\mathcal{I}$ such that
\begin{align*}%
& \sup_{{\substack{ x, y \in \mathcal{I}\\ x_{1} \neq y_{1}, (x_{2},\dots,x_{d}) \neq (y_{2},\dots,y_{d})} }} \frac{| \Box_x^y f |}{( |x_1-y_1|^{\theta_1- \frac{1}{p}}  \wedge  |x_1-y_1|^{\eta- \frac{1}{p}} ) \prod_{i=2}^{d+1} |x_i-y_i |^{\theta_2-\frac{1}{p}}  }  \\ 
& \leq C \left( \int_{\mathcal{I}} \frac{| \Box_z^{z'} f |^p}{( |z_1-z'_1|^{\theta_1 p+1} \wedge |z_1-z'_1|^{\eta p+1} ) \prod_{i=2}^{d+1} |z_i-z'_i|^{\theta_2 p + 1} }   \diff z_1 \diff z'_1 \dots \diff z_{d+1} \diff z'_{d+1} \right)^{ \frac{1}{p}} .
\end{align*}
\end{lemma}
\begin{proof}
This is an application of Theorem 2.3 in \cite{hu2013multiparameter}. More precisely,  using the notations of \citep{hu2013multiparameter}, this result is obtained by taking $n=d+1$, $\psi(x)=x^p$, $p_1(u)=u^{\eta+1/p} \wedge u^{\theta_1+1/p}$ and $p_k(u) =  u^{\theta_2+1/p}$ for $k=2,...,d+1$. The proof in \cite{hu2013multiparameter} is done for $\mathcal{I}= [0,1]^n$ but it can be generalized to any product of closed intervals by normalizing the variable space.
\end{proof}
\subsection{Proof of Proposition \ref{prop:ub-justfBm-rectangular}}\label{app:proof-ub-justfBm-rectangular}
We will prove here the joint increment stationarity of $(B_{t}^H,B_{t}^{H'})_{t\geq 0}$, i.e. the relation \eqref{eq:laws}. 
Provided this equality holds, we get $\mathbb{E}\big( \Box_{(t,H)}^{(t',H')} \fBm \big)^2 = \EE \big(B^{H'}_{t'-t} - B^H_{t'-t} \big)^2$, so Proposition \ref{prop:ub-justfBm} then gives the desired result.
Hence it remains to prove \eqref{eq:laws}. The processes on both side of \eqref{eq:laws} are centred Gaussian. Thus it suffices to prove that they have the same covariance matrix. The equality of the diagonal entries corresponds to a well-known property of the fBm. Hence we focus on the extra-diagonal entries. 
In view of \eqref{eq:fBm}, we get
\begin{align*}
\mathbb{E} \left[ (\fBm_{t+s}^H - \fBm_{s}^H)\,  ( \fBm_{t'+s}^{H'} - \fBm_s^{H'}) \right] 
& = \frac{1}{\Gamma(H+1/2) \Gamma(H'+1/2)} \int_{\R} \left( (t+s-u)_+^{H-\frac{1}{2}} - (s-u)_+^{H-\frac{1}{2}} \right) \\ 
&\quad  \times \left(  (t'+s-u)_+^{H'-\frac{1}{2}} - (s-u)_+^{H'-\frac{1}{2}} \right) \diff u .
\end{align*}
Apply the  change of variables $v=u-s$ to get
\begin{align*}
\mathbb{E} \left[ (\fBm_{t+s}^H - \fBm_{s}^H)\,  ( \fBm_{t'+s}^{H'} - \fBm_s^{H'}) \right]  & = \frac{1}{\Gamma(H+1/2) \Gamma(H'+1/2)} \int_{\R} \left( (t-v)_+^{H-\frac{1}{2}} - (-v)_+^{H-\frac{1}{2}} \right) \\ 
&\quad \times  \left(  (t'-v)_+^{H'-\frac{1}{2}} - (-v)_+^{H'-\frac{1}{2}} \right) \diff v \\
& = \mathbb{E} \left( \fBm_t^H \fBm_{t'}^{H'} \right) ,
\end{align*}
which proves that the covariances of the Gaussian processes involved in  \eqref{eq:laws} are equal.

\subsection{Proof of Proposition \ref{prop:variance-ub-rectangular}}\label{app:proof-variance-ub-rectangular}
First, notice that $\Box_{(t,H)}^{(t',H')} \mathbb{B} = (1+t')^{-\alpha} \Box_{(t,H)}^{(t',H')} B + \left( (1+t)^{-\alpha} - (1+t')^{-\alpha}\right) \left(B^{H}_{t} - B^{H'}_{t}\right)$. Thus in view of Propositions \ref{prop:ub-justfBm} and  \ref{prop:ub-justfBm-rectangular}, one gets
\begin{align*}
\mathbb{E}\left( \Box_{(t,H)}^{(t',H')} \mathbb{B} \right)^2 &\leq C (1+t')^{-2\alpha} \left( |t'-t|^{2H} \vee |t'-t|^{2H'} \right) (\log^2(|t'-t|)+1)\, |H-H'|^2  \\
&\quad + C \left( (1+t)^{-\alpha}-(1+t')^{-\alpha} \right)^2 \big(t^{2H}\vee t^{2H'}\big) \, (\log^2(t)+1)\,  |H-H'|^2 .
\end{align*}
Using \eqref{eq:tinc} with $h=H$ and $h=H'$, we obtain the desired result.

\subsection{Proof of Theorem \ref{thm:regularityGarsia-rectangular}}\label{subsec:ProofTh2}
We follow the same approach as the proof of Theorem \ref{thm:whole-regularity}, and use the same notations. In particular, let $p >\frac{4}{a\varepsilon}$, and let $\theta_{1} = a - \frac{1}{p}$, $\theta_{2} = 1-\frac{1}{p}$ and $\eta \in(\frac{1}{p}, \theta_1)$. We apply Lemma \ref{lem:GRRrect} for the process $\mathbb{B}_t^H$ defined on the compact set $\mathcal{K}_{T}$ and use the fact that $\Box_{(t,H)}^{(t',H')} \mathbb{B}$ is a Gaussian random variable to get that
\begin{align}\label{eq:J1-J2-J3}
B_{T} &:= \mathbb{E} \sup_{\substack{(t,H),(t',H') \in \mathcal{K}_{T}\\t\neq t', H\neq H'}}  \frac{| \Box_{(t,H)}^{(t',H')} \mathbb{B} |^p}{(|t'-t|^{\theta_1- \frac{1}{p}} \wedge |t'-t|^{\eta- \frac{1}{p}} ) |H-H'|^{\theta_2-\frac{1}{p}} } \nonumber\\
&\leq C \int_0^T \int_0^T \int_{\mathcal{H}} \int_{\mathcal{H}} \frac{\mathbb{E} \big| \Box_{(s,h)}^{(s',h')}  \mathbb{B} \big|^p }{( |s-s'|^{\theta_1 p+1}  \wedge |s-s'|^{\eta p +1} ) |h-h'|^{\theta_2 p +1}} \diff h' \diff h \diff s' \diff s \nonumber\\
&\leq C  \int_0^T \int_0^T \int_{\mathcal{H}} \int_{\mathcal{H}} \frac{ \left( \mathbb{E} | \Box_{(s,h)}^{(s',h')}  \mathbb{B} |^2 \right)^{p/2} }{( |s-s'|^{\theta_1 p+1}  \wedge |s-s'|^{\eta p +1} ) |h-h'|^{\theta_2 p +1}} \diff h' \diff h \diff s' \diff s \nonumber\\
&=: C \left( J_1 + J_2 + J_3 \right),
\end{align}
where $J_{1}$ is the integral for $s'$ between $0$ and $(s-1)\vee 0$; $J_{2}$ for $s'$ between $(s-1)\vee 0$ and $(s+1)\wedge T$; and $J_{3}$ for $s'$ between $(s+1)\wedge T$ and $T$.
\paragraph{Bound on $J_1$ and $J_3$.} 
Using $\theta_1 p+1 > \eta p +1$ and $|s'-s|>1$, one has $|s-s'|^{\theta_1 p+1}  \wedge |s-s'|^{\eta p +1} = |s-s'|^{\eta p +1}$. Hence
\begin{align*}
J_1 \leq C \int_1^T \int_0^{s-1} \int_{\mathcal{H}} \int_{\mathcal{H}} \frac{(1+s)^{-p\alpha} \left(\mathbb{E}|B_s^h -B_s^{h'}|^2\right)^{ \frac{p}{2}} + (1+s')^{-p \alpha} \left(\mathbb{E}|B_{s'}^h -B_{s'}^{h'}|^2\right)^{ \frac{p}{2}} }{|s-s'|^{\eta p+1} |h-h'|^{\theta_2 p + 1} }  \diff h' \diff h \diff s' \diff s .
\end{align*}
In view of Proposition \ref{prop:ub-justfBm}, we get
\begin{align*}
J_1 & \leq C \int_1^T \int_0^{s-1} \int_{\mathcal{H}} \int_{\mathcal{H}} \frac{(1+s)^{-p \alpha}\left( s^{ph} \vee s^{ph'} \right) (\log^p(s)+1) }{|s-s'|^{\eta p +1} |h-h'|^{\theta_2 p +1 -p} }  \\ 
&  \qquad + \frac{(1+s')^{-p \alpha}\left( (s')^{ph} \vee (s')^{ph'} \right) (\log^p(s')+1) }{|s-s'|^{\eta p+1} |h-h'|^{\theta_2 p +1 -p}} \diff h' \diff h \diff s' \diff s .
\end{align*}
Since $\theta_2 p +1 - p =0$, the integrals over $h$ and $h'$ are finite. Thus
\begin{align*}
J_1 & \leq C \int_1^T \bigg( (1+s)^{-p \alpha + p\bar{a} } (\log^p(s)+1)+ \int_0^{s-1} (1+s')^{-p \alpha + p \bar{a}}  (\log^p(s')+1) |s-s'|^{-\eta p-1} \diff s' \bigg) \diff s 
\end{align*}
and using that $-\eta p-1<-2$, it comes that
\begin{align}\label{eq:bound-J1}
J_1 & \leq C (1+ (1+ T)^{ p (\bar{a}-\alpha) + 1} (\log^p(T)+1)) .
\end{align}
Proceeding similarly, one obtains the same bound for $J_{3}$.
\paragraph{Bound on $J_2$.} Here we use Proposition \ref{prop:variance-ub-rectangular} and the fact that $|s-s'|^{\theta_1 p+1}  \wedge |s-s'|^{\eta p +1} = |s-s'|^{\theta_{1} p +1}$ when $|s'-s| \leq 1$ to deduce that
\begin{align*}
J_2 \leq C \int_0^T \int_{(s-1)\vee 0}^{(s+1) \wedge T} \int_{\mathcal{H}} \int_{\mathcal{H}} \frac{(1+s \wedge s')^{-p \alpha} |s-s'|^{p a - p \theta_1 -1} \left( 1 + \log^2|s-s'| + \log^2(s \wedge s') \right)^{ \frac{p}{2}} }{|h-h'|^{\theta_2 p + 1 -p}} \diff h' \diff h \diff s' \diff s .
\end{align*}
Since $\theta_2 p +1 - p =0$, the integrals over $h$ and $h'$ are finite. Moreover, $p a - p\theta_1 -1 =0$, hence
\begin{align*}
J_2 & \leq C \int_0^T \int_s^{(s+1) \wedge T} (1+s)^{-p \alpha} \left( 1+ |\log (|s-s'|)|^p + |\log (s)|^p \right) \diff s' \diff s  \\
& \quad+ C \int_0^T \int_{(s-1) \vee 0}^s (1+s')^{-p \alpha} \left( 1+ |\log (|s-s'|)|^p + |\log (s')|^p \right) \diff s' \diff s\\
& =: C ( J_{21} + J_{22} ) .
\end{align*}
We bound $J_{21}$ and skip the details for $J_{22}$, as similar computations work. Using the change of  variables $x=s,\, y=s'-s$, we get
\begin{align*}
J_{21} &\leq C \int_0^T \int_0^{1} (1+x)^{-p \alpha} (1+ | \log^p( y) | + | \log^p(x) | ) \diff y \diff x \\
&\leq C \int_0^T (1+x)^{-p \alpha} (1 + | \log^p(x) | )\diff x .
\end{align*}
Thus $J_{21} \leq C (1+ T^{-p \alpha + 1}) (1 + \log^p(T))$ and the same bound holds for $J_{22}$, so that
\begin{align}\label{eq:bound-J2}
J_{2} \leq C (1 + (1+T)^{-p \alpha + 1}(1 + \log^p(T))).
\end{align}
We can now plug \eqref{eq:bound-J1} and \eqref{eq:bound-J2} into \eqref{eq:J1-J2-J3}. Choosing $\alpha > \frac{1}{p} + \bar{a}$, we obtain that $B_T$ is bounded uniformly in $T \in \mathbb{R}_+$. Hence we conclude that there exists a positive random variable $\mathbf{C}$ that has a finite moment of order $p$ such that almost surely,
\begin{align}\label{eq:asboundX}
| \Box_{(t,H)}^{(t',H')} \mathbb{B} | \leq \mathbf{C}\, |H-H'|^{\theta_2- \frac{1}{p}} \left( |t-t'|^{\eta - \frac{1}{p}} \wedge |t-t'|^{\theta_1 - \frac{1}{p}} \right) .
\end{align}
Now for $t' \ge t$, one has
\begin{align*}
(1+t')^{-\alpha}  \big| \Box_{(t,H)}^{(t',H')} B \big| \leq | (1+t)^{-\alpha} - (1+t')^{-\alpha} | |B_t^H-B_t^{H'} |  + \big| \Box_{(t,H)}^{(t',H')} \mathbb{B} \big|.
\end{align*}
To control the right-hand side of the previous inequality, it remains to use \eqref{eq:tinc} and Theorem \ref{thm:whole-regularity} for the first term of the sum, and \eqref{eq:asboundX} for the second term. Overall, it comes
\begin{align*}
 \big| \Box_{(t,H)}^{(t',H')} \fBm \big| & \leq (1+t')^{\alpha} \left(|H-H'|^{\theta_{2}-\frac{1}{p}}~ \big( |t-t'|^{\eta-\frac{1}{p}}  \wedge |t-t'|^{\theta_1-\frac{1}{p}}\big) \right).
 \end{align*}
 Using $\eta< \theta_1$ and $t'\geq t$, we have
 \begin{align*}
 \big| \Box_{(t,H)}^{(t',H')} \fBm \big| & \leq (1+t')^{\alpha+ \eta-\frac{1}{p}}   \left(|H-H'|^{1-\frac{2}{p}}~ \big( 1  \wedge |t-t'|^{a-\frac{2}{p}}\big) \right) .
\end{align*}
Setting $\alpha =\bar{a} + \frac{3}{p}$, $\eta = \frac{2}{p}$ and $\varepsilon = \frac{2}{ap}$ yields the desired result.

\section*{Declarations}

\paragraph{Conflict of interest} The authors have no conflict of interest to declare.

\paragraph{Data availability} Not applicable to this article as no datasets were generated.

\end{document}